\documentclass[12pt]{article}

\usepackage{arxiv}

\usepackage[utf8]{inputenc} % allow utf-8 input
\usepackage[T1]{fontenc}    % use 8-bit T1 fonts
\usepackage{hyperref}       % hyperlinks
\usepackage{url}            % simple URL typesetting
\usepackage{booktabs}       % professional-quality tables
\usepackage{amsfonts}       % blackboard math symbols
\usepackage{nicefrac}       % compact symbols for 1/2, etc.
\usepackage{microtype}      % microtypography
\usepackage{graphicx}
\usepackage{doi}
\usepackage{xparse}

\usepackage{mathtools,amsmath,amsfonts,amssymb,amsthm,subcaption}
\usepackage{xcolor}
\usepackage{enumitem}
\usepackage[ruled]{algorithm2e}

\usepackage[
    backend		= biber,
    style		= ieee,
    natbib		= true,
    giveninits  = true,
    maxbibnames = 99,
    minbibnames = 1,
    url			= false,
    doi			= true,
    eprint		= false,
    isbn 		= false,
    clearlang	= true,
    language 	= english
    ]{biblatex}
\theoremstyle{plain}\newtheorem{proposition}{Proposition}[section]
\theoremstyle{plain}\newtheorem{theorem}{Theorem}[section]
\theoremstyle{plain}\newtheorem{lemma}{Lemma}[section]
\theoremstyle{plain}\newtheorem{corollary}{Corollary}[section]

\theoremstyle{definition}\newtheorem{definition}{Definition}[section]

\newcommand{\mc}[1]{\mathcal{#1}} 		        % \mathcal{} 		(caligraphic)
\newcommand{\mbb}[1]{\mathbb{#1}} 		        % \mathbb{} 		(blackboard)
\newcommand{\wh}[1]{\widehat{#1}}               % \widehat{}
\NewDocumentCommand{\vc}{g}{%                   % vector/matrix symbols
    \IfNoValueTF{#1}
    {%
        {\boldsymbol c}
    }
    {%
        {\boldsymbol #1}
    }%
}
\newcommand{\bmtx}[1]{\begin{bmatrix} #1 \end{bmatrix}}     % alias for bracketed matrices
\newcommand{\iprod}[2]{\langle #1,\, #2 \rangle}            % real inner product
\newcommand{\bprod}[2]{( #1,\, #2 )}                        % complex inner product
\newcommand{\nrm}[1]{\|#1\|}                                % norm
\newcommand{\set}[1]{\{#1\}}                                % curly braces
\newcommand{\Lset}[1]{\left\{#1\right\}}                    % curly braces (large)
\newcommand{\bset}[1]{[#1]}                                 % braces

\def\sinc{\operatorname{sinc}}
\def\ind{\mbb{I}}
\def\proj{\operatorname{proj}}
\def\real{\operatorname{Re}}
\def\vspan{\operatorname{span}}
\def\dom{\operatorname{dom}}
\def\supp{\operatorname{supp}}
\def\dist{\operatorname{dist}}
\def\ev{\operatorname{E}}
\def\trace{\operatorname{trace}}
\def\diag{\operatorname{diag}}
\def\cvxhull{\operatorname{conv}}
\def\tint{\operatorname{int}}                       % interior
\def\cl{\operatorname{cl}}                          % closure
\def\wscl{\operatorname{w{\tiny\ast}-cl}}           % closure / weak-*
\def\sums{\sum\nolimits}
\def\prods{\prod\nolimits}
\def\sups{\sup\nolimits}
\def\infs{\inf\nolimits}
\def\lims{\lim\nolimits}
\def\limsups{\lim\sups}
\def\bigcaps{\bigcap\nolimits}
\def\bigcups{\bigcup\nolimits}

\def\opt{^{\star}}                                  % optimal
\def\adj{^{*}}                                      % adjoint

\def\mN{\mc{N}}                                     % normal distribution
\def\mH{\mc{H}}                                     % hilbert space
\def\alp{\alpha}                                    % \alpha
\def\th{\theta}                                     % \theta
\def\eps{\varepsilon}                               % \vaerpsilon
\def\vphi{\varphi}                                  % \varphi
\def\lam{\lambda}                                   % \lambda
\def\Sig{\Sigma}                                    % \Sigma
\def\Lam{\Lambda}                                   % \Lambda

\def\xio{\xi_0}                                     % specific frequency
\def\xo{x_0}                                        % specific point
\def\uo{u_{\Omo}}                                   % low-frequency approximation
\def\ue{u_{\Ome}}                                   % extrapolated
\def\uxo{u_{\xo}}
\def\vo{v_0}                                        % specific function
\def\mopt{m\opt}                                    % multiplier / optimal
\def\mf{m_f}
\def\etaf{\eta_f}
\def\fxo{f_{\xo}}
\def\uxoo{u_{\xo,\Omo}}
\def\uxoe{u_{\xo,\Ome}}

\def\cmu{c_{\mu}}                                   % lower constant / \mu
\def\Cmu{C_{\mu}}                                   % upper constant / \mu

\def\tS{\tau_{\Sig}}                                % FPI / constant
\def\tG{\tau_G}                                     % 
\def\Kd{K_{\Delta}}                                 % FPI / set
\def\Md{M_{\delta}}                                 % FPI / map
\def\D{D_{2}}                                       % scaling operator x2
\def\Da{D_{\alpha}}                                 % scaling operator x (\alpha)
\def\Dia{D_{\alpha^{-1}}}                           % scaling operator x (1/\alpha)
\def\vA{\vc{A}}                                     % data matrix
\def\vG{\vc{G}}                                     % gram matrix
\def\vGopt{\vG\opt}                                 % gram matrix / optimal
\def\vI{\vc{I}}                                     % identity matrix
\def\vC{\vc{C}}                                     % matrix variables
\def\vU{\vc{U}}
\def\vV{\vc{V}}                                     
\def\vW{\vc{W}}
\def\vX{\vc{X}}                                     
\def\vY{\vc{Y}}
\def\vZ{\vc{Z}}
\def\vLam{\vc{\Lam}}
\def\vSig{\vc{\Sigma}}
\def\vd{\vc{d}}                                     % matrix diagonal
\def\ve{\vc{e}}                                     % vector of exponentials
\def\vf{\vc{f}}                                     % vector of Fourier transforms
\def\vp{\vc{p}}                                     % projection
\def\vs{\vc{s}}                                     % diagonal values
\def\vv{\vc{v}}                                     % vector variables
\def\vx{\vc{x}}                                     
\def\vy{\vc{y}}
\def\vz{\vc{z}}
\def\lamX{\vc{\lambda}_X}                           % vector of eigenvalues
\def\whh{\wh{h}}
\def\whu{\wh{u}}
\def\whv{\wh{v}}

\def\whphi{\wh{\phi}}
\def\whpsi{\wh{\psi}}
\def\R{\mbb{R}}                                     % reals
\def\N{\mbb{N}}                                     % positive integers
\def\No{\mbb{N}_0}                                  % non-negative integers
\def\Z{\mbb{Z}}                                     % integers
\def\C{\mbb{C}}                                     % complex

\def\Rd{\R^d}                                       % R^d
\def\Rn{\R^n}                                       % R^n
\def\Rnp{\Rn_+}                                     % R^n / positive orthant
\def\Zd{\Z^d}                                       % Z^d
\def\Cd{\C^d}                                       % C^d
\def\Cn{\C^n}                                       % C^n

\def\Cnn{\C^{n\times n}}                            % square matrices
\def\Hnn{\mbb{H}^{n\times n}}                       % Hermitian matrices
\def\Unn{\mbb{U}(n)}                                % unitary matrices

\def\Sd{\mbb{S}^{d-1}}                              % d-sphere

\def\mX{\mc{X}}
\def\mQ{\mc{Q}}                                     % hypercube
\def\mQd{\mQ_d}                                     % hypercube / in Rd
\def\Omo{\Omega_0}                                  % low-frequency set
\def\Omor{\Omo^{\rho}}                              % low-frequancy set / erosion
\def\Ome{\Omega_{e}}                                % extended frequencies
\def\Omlo{\Omega_{\textrm{lo}}}                     % low-frequency set
\def\Omhi{\Omega_{\textrm{hi}}}                     % high-frequency set
\def\mU{\mc{U}}
\def\mF{\mc{F}}                                     % fourier data
\def\ZF{Z_{\mF}}                                    % fourier data / common zeros
\def\ZFr{\ZF^{\rho}}                                % fourier data / common zeros / neighborhood
\def\VF{V}                                          % fourier data / generated space
\def\SVF{S}                                         % fourier data / worst-case set
\def\SVFopt{\SVF\opt}                               % fourier data / worst-case set / subdiff
\def\CVF{C}                                         % fourier data / worst-case set / coords
\def\So{\SVF_0}                                     % fourier data / worst-case set / countable 
\def\BorSVF{\mc{B}(\SVF)}                           % fourier data / worst-case set / borel sigma-algebra
\def\PS{\mc{P}(\SVF)}                               % fourier data / worst-case set / probability measures on S

\def\Ewc{E_S}                                       % worst-case error
\def\Ewcopt{\Ewc^{\star}}                           % worst-case error / infimum
\def\LE{L^2}                                        % square-integrable
\def\LInf{L^\infty}                                 % essentially bounded
\def\LER{\LE(\R)}
\def\LERd{\LE(\Rd)}
\def\LEQR{\LE_{\mQ}(\R)}
\def\LEQRd{\LE_{\mQd}(\Rd)}
\def\LEOmo{\LE(\Omo)}
\def\LEOmor{\LE(\Omor)}

\def\LInfRd{\LInf(\Rd)}

\def\LInfOmo{\LInf(\Omo)}

\def\lE{\ell_2}
\def\lInf{\ell_\infty}
\def\lEn{\lE^n}
\def\lEp{\lE^p}
\def\BE{B_2}
\def\cBE{\bar{B}_2}
\def\BInf{B_{\infty}}
\def\cBInf{\bar{B}_{\infty}}
\def\cBEd{\cBE^d}
\def\cBInfd{\cBInf^d}
\newcommand{\absE}[1]{|#1|}
\newcommand{\LabsE}[1]{\left|#1\right|}
\newcommand{\eprod}[2]{\iprod{#1}{#2}}
\newcommand{\enrm}[1]{|#1|}
\newcommand{\bprodLE}[2]{\bprod{#1}{#2}_{\LE}}
\newcommand{\iprodLE}[2]{\iprod{#1}{#2}_{\LE}}      
\newcommand{\nrmLE}[1]{\nrm{#1}_{\LE}}

\newcommand{\nrmLInf}[1]{\nrm{#1}_{\LInf}}

\newcommand{\nrmVF}[1]{\nrm{#1}_{\VF}}
\newcommand{\bprodlEn}[2]{\bprod{#1}{#2}_{\lEn}}
\newcommand{\iprodlEn}[2]{\iprod{#1}{#2}_{\lEn}}
\newcommand{\nrmlEn}[1]{\nrm{#1}_{\lEn}}

\newcommand{\bprodlEp}[2]{\bprod{#1}{#2}_{\lEp}}

\newcommand{\nrmlEd}[1]{\nrm{#1}_{\lE^d}}
\newcommand{\bprodlEm}[2]{\bprod{#1}{#2}_{\lE^m}}
\newcommand{\nrmlEm}[1]{\nrm{#1}_{\lE^m}}
\newcommand{\nrmop}[1]{\nrm{#1}_{\mathrm{op}}}

\newcommand{\bprodF}[2]{\bprod{#1}{#2}_{F}}
\newcommand{\iprodF}[2]{\iprod{#1}{#2}_F}
\newcommand{\nrmF}[1]{\nrm{#1}_F}
\newcommand{\bprodS}[2]{\bprod{#1}{#2}_{\Sig}}
\newcommand{\iprodS}[2]{\iprod{#1}{#2}_{\Sig}}
\newcommand{\nrmS}[1]{\nrm{#1}_{\Sig}}

\title{Adaptive multipliers for extrapolation in frequency}

\date{\today}	    % Here you can change the date presented in the paper title
\author{
    \href{https://orcid.org/0009-0004-1247-2598}{\includegraphics[scale=0.06]{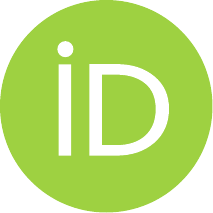}\hspace{1mm}Diego~Castelli~Lacunza} \\
	Institute for Mathematical and Computational Engineering\\
	Pontificia Universidad Cat\'olica de Chile\\
	Santiago, Chile \\
	\texttt{decastelli@uc.cl} \\
	\And
    \href{https://orcid.org/0000-0002-2533-2509}{\includegraphics[scale=0.06]{orcid.pdf}\hspace{1mm}Carlos~A.~Sing~Long} \\
        Institute for Mathematical and Computational Engineering and \\
        Institute for Biological and Medical Engineering \\
	Pontificia Universidad Cat\'olica de Chile\\
	Santiago, Chile \\
	\texttt{casinglo@uc.cl}
}

\renewcommand{\headeright}{Preprint}
\renewcommand{\undertitle}{Preprint}
\renewcommand{\shorttitle}{}

\hypersetup{
    pdftitle        ={Adaptive multipliers for extrapolation in frequency},
    pdfsubject      ={math.FA, eess.SP},
    pdfauthor       ={D.~Castelli~Lacunza, C.~A.~Sing~Long},
    pdfkeywords     ={Fourier extrapolation, refinement equation, Fourier multipliers, multiresolution analysis, adaptive filters},
}

\begin{document}
\maketitle

\begin{abstract}
    Resolving the details of an object from coarse-scale measurements is a classical problem in applied mathematics. This problem is usually formulated as extrapolating the Fourier transform of the object from a bounded region to the entire space, that is, in terms of performing {\em extrapolation in frequency}. This problem is ill-posed unless one assumes that the object has some additional structure. When the object is compactly supported, then it is well-known that its Fourier transform can be extended to the entire space. However, it is also well-known that this problem is severely ill-conditioned.

    In this work, we assume that the object is known to belong to a collection of compactly supported functions and, instead performing extrapolation in frequency to the entire space, we study the problem of extrapolating to a larger bounded set using dilations in frequency and a single Fourier multiplier. This is reminiscent of the refinement equation in multiresolution analysis. Under suitable conditions, we prove the existence of a worst-case optimal multiplier over the entire collection, and we show that all such multipliers share the same canonical structure. When the collection is finite, we show that any worst-case optimal multiplier can be represented in terms of an Hermitian matrix. This allows us to introduce a fixed-point iteration to find the optimal multiplier. This leads us to introduce a family of multipliers, which we call \(\Sigma\)-multipliers, that can be used to perform extrapolation in frequency. We establish connections between \(\Sigma\)-multipliers and multiresolution analysis. We conclude with some numerical experiments illustrating the practical consequences of our results.
\end{abstract}

% keywords can be removed

%% --- FORMATTING
\newlength{\pskip}              % paragraph skip
\setlength{\pskip}{4pt}

\linespread{1.1}

%% ============================================================================
\parskip = 1pt
% \tableofcontents
%% ============================================================================

\parskip = \pskip

%% ============================================================================
\section{Introduction}
\label{sec:intro}
%% ============================================================================

Resolving the details of an object from coarse-scale measurements is a classical problem in applied mathematics, arising is problems as diverse as statistical downscaling for weather prediction~\cite{wan_debias_2023,wan_statistical_2024}, \(k\)-space extrapolation for Nuclear Magnetic Resonance (NMR) spectroscopy~\cite{haupt_removal_1996}, low-frequency extrapolation for seismic data in geophysics~\cite{sun_low-frequency_2018,chiu_low-frequency_2020,jia_learning_2024}, and Fourier extrapolation for the super-resolution of Synthetic Aperture Radar (SAR) images~\cite{brito_sar_2003,zhang_superresolution_2013}. Perhaps the most well-known application is super-resolution in optical imaging~\cite{toraldo_di_francia_resolving_1955,barnes_object_1966} particularly the super-resolution of point-sources~\cite{candes_super-resolution_2013,candes_towards_2014,fernandez-granda_super-resolution_2016}.

This problem is typically modeled using the Fourier transform of a function \(u\) on \(\Rd\) representing the underlying object. The coarse-scale measurements are often modeled as the restriction of the Fourier transform to a bounded {\em low-frequency set}, representing the {\em low-frequency content} about \(u\), whereas the details correspond to the values of its Fourier transform outside this set. In other words, it is a problem of extrapolation in the frequency domain. We informally call this {\em extrapolation in frequency} (see also~\cite{benedetto_super-resolution_2020}).

It is clear that this problem is ill-posed. In fact, we may impute the missing values for the Fourier transform \(\whu\) of \(u\) using {\em any} method while remaining consistent with its low-frequency content. However, and perhaps surprisingly, if we assume that \(u\) belongs to a specific class of functions, then \(u\) is uniquely characterized in this class by its low-frequency content.

The most important results along these lines are based on {\em analytic continuation}. If \(u\) is compactly supported, then the celebrated Paley-Wiener-Schwartz theorem implies that its Fourier transform is the restriction of an entire function of exponential type to the reals (cf.~\cite[Theorem~10]{paley_fourier_1934} and~\cite[Theorem~7.3.1.]{hormander_analysis_2003}). If \(d = 1\) then the {\em uniqueness principle} for analytic functions~\cite[Chapter~V]{gamelin_complex_2001} implies that \(u\)
is uniquely characterized in this class by its low-frequency content if the low-frequency set contains an accumulation point. If \(d > 1\) then the Fourier transform of \(u\) is real-analytic on \(\Rd\). In this case, if the low-frequency set is an open set, then \(u\) is uniquely characterized in this class by its low-frequency content~\cite[cf. Section~4.1]{krantz_primer_2002}.

Although impressive, extrapolating the low-frequency content is an ill-conditioned problem. At first glance, this seems to be a consequence of the fact that we need the derivatives of all orders at a point to evaluate a power series. However, the series of publications by D.~Slepian, H.~O.~Pollak and H.~E.~Landau shows that this is not the case, while yielding insight into the ill-conditioning of this problem~\cite{slepian_prolate_1961,landau_prolate_1961,landau_prolate_1962,slepian_prolate_1964,slepian_prolate_1978}. For brevity, we present their main results for \(d = 1\). Let \(\Delta x, \Delta \xi > 0\) and let \(\mQ\) and \(\Omo\) be centered closed intervals of length \(\Delta x\) and \(\Delta \xi\) respectively. Denote as \(\LEQR\) the closed subspace of \(\LER\) consisting of functions that vanish on \(\mQ^c\). As argued before, the Fourier transform \(\whu\) of any \(u\in \LEQR\) is uniquely characterized by \(\wh{u}|_{\Omo}\). In particular, there exists a strictly decreasing sequence \(\set{\lam_n}_{n\in\No}\) of positive numbers strictly bounded above by \(1\) and an orthonormal collection \(\set{\psi_n}_{n\in \No}\) in \(\LEQR\) such that
\begin{equation}\label{eq:intro:slepianExpansion}
    \whu = \sums_{n\in\No} \frac{1}{\lam_n} \left(\int_{-\Delta\xi/2}^{\Delta\xi/2} \psi_n(\xi)^* \whu(\xi)\, d\xi \right) \psi_n.
\end{equation}
This expansion yields a closed-form representation for \(\whu\) in terms of \(\whu|_{\Omo}\). However, as the sequence \(\set{\lam_n}_{n\in\No}\) tends to zero, it is not continuous with respect to \(\whu|_{\Omo}\). In fact, for any \(n\in\No\) the pair \((\lam_n,\psi_n)\) is an eigenpair of the compact integral operator \(K : \LER \to \LER\)
\[
    K\psi_n(\xi) = \Delta x \int_{-\Delta\xi/2}^{\Delta\xi/2} \sinc(\Delta x (\xi - \xi')) \psi_n(\xi')\, d\xi'.
\]
It can be shown that the sequence \(\set{\lam_n}_{n\in\No}\) has a sharp transition near \(n \approx \Delta x \Delta\xi\) after which the values \(\lam_n\) are exponentially small. This quantitative bound shows that the representation~\eqref{eq:intro:slepianExpansion} is not stable and that, even if we assume that \(u\in \LEQR\), the problem of extrapolation in frequency remains ill-posed. We note that, although we have presented these results in \(d=1\), these arguments were extended to the case \(d > 1\) in~\cite{slepian_prolate_1964} with a focus on restrictions to disks for \(d = 2\). Although it is not the focus of our work, this construction can also be extended to discrete sequences~\cite{slepian_prolate_1978}.

The lack of stability of the representation in~\eqref{eq:intro:slepianExpansion} often precludes its use in practical applications unless the number of terms in the expansion is truncated up to a point where the extrapolation in frequency is modest~\cite{rushforth_restoration_1968,gerchberg_super-resolution_1974}. For this reason, R.~W.~Gerchberg and A.~Papoulis proposed independently an algorithm that could perform extrapolation in frequency while remaining stable to perturbations in the data~\cite{gerchberg_super-resolution_1974,papoulis_new_1975}. This stability should be interpreted in the sense that the algorithm produces an approximation that is be accurate up to a frequency inversely proportional to the magnitude of the perturbation. Their algorithm is an instance of the {\em Projection Onto Convex Sets (POCS) algorithm}~\cite{sezan_image_1983}. If we let \(C\subset \LER\) be the closed affine set
\[
    C = \set{v\in \LER:\,\, \whv|_{\Omo} = \whu|_{\Omo}}
\]
and we let \(\proj_{\LEQR} : \LER \to\LER\) and \(\proj_{C} : \LER\to \LER\) denote their orthogonal projectors, then the {\em Gerchberg-Papoulis (GP) algorithm} yields the sequence \(\set{v^{(k)}}_{k\in\No}\) given by
\[
    v^{(k+1)} = \proj_{\LEQR}(\proj_C(v^{(k)}))
\]
from a known \(v^{(0)}\in \LER\). The sequence \(\set{v^{(k)}}_{k\in\No}\) converges weakly to a limit point in \(C\cap \LEQR\)~\cite[Corollary~5.23]{bauschke_convex_2011}. Since \(u\) is the only point in this intersection, we conclude that \(v^{(k)} \to u\) weakly as \(k\to\infty\) (see also~\cite{papoulis_new_1975}).
The iteration can be written equivalently as
\[
    \whv^{(k+1)} = h_{\mQ} \ast (\chi_{\Omo^c}\whv^{(k)} + \chi_{\Omo} \whu)
\]
where \(h_{\mQ}\) is the Fourier transform of \(\chi_{\mQ}\). As critical property of this iteration is that \(\whv^{(k)} = h_{\mQ} \ast \whv^{(k)}\) for \(k\in\N\). It follows from this that the GP algorithm is equivalent to using gradient descent to solve (see~\cite{jain_extrapolation_1981} for the discrete case)
\[
    \begin{aligned}
        & \underset{}{\text{minimize}}
        & & \frac{1}{2}\nrmLE{\chi_{\Omo}(h_{\mQ}\ast \whv - \whu)}^2
        & \text{subject to}
        & &\whv\in \LER.
    \end{aligned}
\]
Hence, the stability of the GS algorithm can be explained in terms of the underlying variational problem that it solves, that is, the algorithm will find the best trade-off between a function that is concentrated on \(\mQ\) and one that matches the low-frequency content \(\whu|_{\Omo}\).

From these results, we infer that extrapolation in frequency {\em can} be stable if we are willing to restrict ourselves to a specific class of functions and to extrapolate up to a finite frequency. This leads us to consider the following formulation. Let \(\alpha > 1\). Given a collection of functions \(\mU\) in \(\LEQR\) we seek a map \(T\) that extrapolates in frequency from \(\Omo\) to \(\alpha\Omo\), that is, for which
\begin{equation}\label{eq:intro:extrapolationAsDilation}
    u\in\mU :\,\, T(\whu|_{\Omo}) = \whu|_{\alpha\Omo}.
\end{equation}
This condition can be written in terms of the {\em dilation operator}
\[
    \Da \whu(\xi) = \whu(\alpha\xi)
\]
as
\[
    u\in\mU :\,\, T(\whu|_{\Omo}) = (\Da\whu)|_{\Omo}.
\]
This is a commutation relation: the right-hand side is the composition between the map \(T\) and the restriction to \(\Omo\) whereas the left-hand side is the restriction to \(\Omo\) composed with the dilation \(\Da\). Our key observation is that for a {\em single} \(u\) there exists a {\em multiplier} that satisfies the above, namely,
\[
    T(\whu|_{\Omo}) = m \whu|_{\Omo} \quad\mbox{for}\quad m =\frac{\Da\whu|_{\Omo}}{\whu|_{\Omo}}.
\]
This multiplier is characterized by the equation
\[
    \Da\whu|_{\Omo} = m \whu|_{\Omo}
\]
which is reminiscent of the {\em refinement equation}~\cite{de_boor_construction_1993} in multiresolution analysis~\cite{mallat_multiresolution_1989,meyer_ondelettes_1997}. Although it is unlikely that a multiplier will work for all the elements in the collection, it is of interest to understand the limits of this approach. In fact, a multiplier that performs well must encode information about the structure of the low-frequency content for the entire collection \(\mU\). Motivated by this observation, our goal is to study the extent to which one can construct an {\em optimal multiplier} in a suitable sense so that~\eqref{eq:intro:extrapolationAsDilation} holds approximately.

%% ============================================================================
\subsection{Contributions}
%% ============================================================================

In this work, given a class of function \(\mU\) of functions in \(\LERd\) compactly supported in the unit cube \(\mQ\), a low-frequency set \(\Omo\), and a factor \(\alpha > 1\), we aim to find a suitable multiplier \(m\) defined on \(\Omo\) such that
\[
    u\in\mU:\,\, m \whu|_{\Omo} \approx \Da \whu|_{\Omo}
\]
holds in a suitable sense. Our contributions are as follows.
\begin{enumerate}[leftmargin=*,label=(\roman*)]
    \item{{\bf Optimal multipliers in the worst-case:} Under suitable definitions of a {\em worst-case approximation error} we can prove the existence of an {\em optimal multiplier}. Furthermore, we show that these optimal multipliers have a canonical structure similar to a pointwise projection.
    }
    \item{{\bf \(\Sigma\)-multipliers:} When \(\mU\) is finite we show that the structure optimal multiplier obtained for a variety of choice of worst-case error is always the same. This family of multipliers, which we call \(\Sigma\)-multipliers, can be parameterized in terms of an Hermitian matrix of a size proportional to the size of \(\mU\).
    }
    \item{{\bf Implementation:} We introduce a fixed-point iteration to find the optimal multiplier when \(\mU\) is finite. We prove that this iteration converges to the optimal multiplier for a wide range of definitions for the worst-case error. We provide numerical evidence that the iteration still converges when the integrals involved are replaced by a numerical approximation.
    }
    \item{{\bf Connection to multiresolution:} As our approach is inspired by multiresolution analysis, we show that a suitable choice of \(\Sigma\)-multiplier induces a multiresolution. In the case of an optimal multiplier, we can think about this as a form of multiresolution {\em adapted} to \(\mU\). This adaptation is distinct from that introduced in~\cite{han_adaptive_2011}.
    }
\end{enumerate}

%% ============================================================================
\subsection{Structure}
%% ============================================================================

The paper is organized as follows. In Section~\ref{sec:adaptiveMultipliers} we present our main theoretical results. We first introduce a suitable topology for the space of Fourier transforms of the elements of \(\mU\) to then define a suitable notion of the worst-case error. We prove the existence of optimal multipliers minimizing the worst-case error, and we determine their canonical structure. In Section~\ref{sec:adaptiveMultipliersFiniteCollection} we consider the case of a finite collection \(\mU\) and we show that the family of possible optimal multipliers can be represented in terms of a Hermitian matrix. We call these \(\Sigma\)-multipliers. This allows us to introduce a suitable fixed-point iteration to find the optimal multiplier. We prove the convergence of this iteration under mild assumptions. In Section~\ref{sec:connectionsToMultiresolution} we explore the connection between \(\Sigma\)-multipliers and multiresolutions. In particular, we provide conditions under which such a multiplier induces a multiresolution. In Section~\ref{sec:numericalExperiments} we present numerical experiments illustrating two applications. Finally, in Section~\ref{sec:conclusion} we conclude with a brief discussion about some directions for future research.

%% ============================================================================
\section{Preliminaries}
%% ============================================================================

The underlying spaces that we will use in this work are assumed to be real vector spaces. This will be the case even when they comprise complex-valued functions. If \(\mH\) is a complex Hilbert space endowed with the complex inner product \(\bprod{\cdot}{\cdot}_{\mH}\) then \(\mH\) can be reduced to a real Hilbert space with the real inner product \(\iprod{\cdot}{\cdot}_{\mH} := \real(\bprod{\cdot}{\cdot}_{\mH})\). For example, the canonical complex inner product on \(\LERd\) is
\[
    \bprodLE{u}{v} = \int u(x)^* v(x)\, dx
\]
whereas the corresponding real inner product is
\[
    \iprodLE{u}{v} = \real\left(\int u(x)^* v(x)\, dx\right).
\]
Given a collection \(\mF\) of complex-valued functions defined on a set \(S\) we denote as \(\vspan(\mF)\) the space generated by their linear combinations. If \(\mH\) is a Hilbert space and \(S\subset \mH\) then we denote its closure with respect to the norm (or uniform) topology as \(\cl(S)\). Similarly, we denote its closure with respect to the weak-* (or weak) topology as \(\wscl(S)\).

We denote as \(\bar{\R}\) the set of extended real numbers. Let \(X\) be any real vector space, let \(S\subset X\) and let \(f: X\to \bar{\R}\). The {\em convex hull} of \(S\) is the set of all finite convex combinations of elements of \(S\) and it is denoted as \(\cvxhull(X)\). We say that \(f\) is {\em proper} if it does not take the value \(-\infty\) and we say that \(f\) is {\em convex} if \(f(\th x + (1-\th) y) \leq \th f(x) + (1-\th) f(y)\) for any \(x,y\in X\) and \(\th\in [0, 1]\). Its {\em domain} \(\dom(f)\) is the subset of \(X\) on which it does not take the value \(+\infty\).

In \(\Cn\) we let \(\BE\) and \(\cBE\), and \(\BInf\) and \(\cBInf\) denote the open and closed unit balls in the \(\lE\)-norm and \(\lInf\)-norm respectively. Finally, we denote \(\mX_n = \set{1,\ldots, n}\).

%% ============================================================================
\section{Adaptive multipliers for extrapolation in frequency}
\label{sec:adaptiveMultipliers}
%% ============================================================================

Let \(\LEQRd\) be the closed subspace of \(\LERd\) consisting of functions with support on the unit cube \(\mQd\). We make the choice of the unit cube for simplicity as our results still hold if we replace \(\mQd\) for any compact set with non-empty interior. Let~\(\mU := \set{u_{\gamma}:\, \gamma\in \Gamma}\) be either a finite or countable collection of linearly independent functions in \(\LEQRd\) contained in a bounded set. These represent the class of objects of interest. To simplify the notation and the exposition, denote by \(\mF := \set{f_{\gamma}:\, \gamma \in \Gamma}\) the collection of Fourier transforms of functions in \(\mU\). Since these are Fourier transforms of compactly supported functions, they belong to \(\LERd\) and, by the Payley-Wiener-Schwartz theorem, they are the restrictions to \(\Rd\) of entire functions of exponential type in \(\Cd\)~\cite[Theorem~4.9]{stein_introduction_1975}. From the uniqueness principle it follows that non-trivial linear combinations of elements in \(\mF\) cannot vanish on an open set in~\(\Cd\)~\cite[Theorem~1.19]{range_holomorphic_1986}. Note that this does not directly imply that this result holds for open sets in \(\Rd\) as it has empty interior in \(\Cd\). However, it follows that the functions in \(\mF\) are real-analytic on \(\Rd\). From this, non-trivial linear combinations of functions in \(\mF\) cannot vanish on an open subset of \(\Rd\)~\cite[cf.~Section~4.1]{krantz_primer_2002}.

Let \(\Omo \subset \Rd\) be a compact set with non-empty interior. From now on this will be the set associated to {\em low-frequencies}. Typical choices are \(\Omo = \cBInfd\), \(\Omo = \cBEd\) and angular sectors of the form
\[
    \Omo := \set{\xi\in\Rd:\,\, \eprod{\omega}{\xi} \geq \theta\enrm{\xi},\,\, r_{\min} \leq \enrm{\xi} \leq r_{\max}}
\]
for some \(\theta\in (0, 1)\), \(\omega\in\Sd\) and \(0 < r_{\min} < r_{\max}\).

The {\em low-resolution content} of a function \(f\in\mF\) is represented by the restriction \(f|_{\Omo}\) of \(f\) to \(\Omo\). Let \(\alp > 1\). In this work, {\em extrapolation in frequency by a factor \(\alp\)} implies finding an approximation of \(f|_{\alp\Omo}\) that yields an approximation of \(f|_{\alp\Omo\cup\Omo}\). This motivates defining the set \(\Ome = \Omo\cup \alp\Omo\) of {\em extrapolated frequencies}. Although strictly speaking only the values at \(\alp\Omo \setminus\Omo\) are being extrapolated, this distinction will not play any role in our theoretical developments. However, they will play a role in our numerical experiments in Section~\ref{sec:numericalExperiments}.

Approximating \(f|_{\alp\Omo}\) from \(f|_{\Omo}\) can be done in many different ways. We focus specifically on approximating \(f|_{\alp\Omo}\) as best as possible for each \(f\in \mF\) using a single multiplier \(m\), i.e.,
\[
    \forall\, f\in\mF,\, \xi \in\Omo:\quad  m(\xi) f(\xi) = f(\alpha\xi).
\]
{\em If} such a multiplier exists, we can extrapolate \(f|_{\Omo}\) to \(\alp\Omo\) as
\[
    f\in\mF,\,\xi\in \alp\Omo:\quad f|_{\alp\Omo}(\xi) = m\left(\frac{1}{\alp}\xi\right) f|_{\Omo}\left(\frac{1}{\alp}\xi\right)
\]
which is reminiscent of the well-known {\em refinement equation}~\cite{daubechies_orthonormal_1988,christensen_introduction_2016}.

This can be interpreted in terms of a convolution in the space domain. Consider first the decomposition
\[
    \xi \in \Ome:\,\, f|_{\Ome}(\xi) = f|_{\Omo}(\xi) + \chi_{\alp\Omo\setminus\Omo}(\xi) m\left(\frac{\xi}{\alp}\right) f|_{\Omo}\left(\frac{\xi}{\alp}\right).
\]
It is now useful to define formally the {\em dilation by \(\alp\)}
\[
    \Da f(\xi) := f(\alp\xi).
\]
If we let \(\uo\) and \(\ue\) be the inverse Fourier transform of \(f\chi_{\Omo}\) and \(f\chi_{\Ome}\) respectively, and if we let \(\eta\) be the inverse Fourier transform of \(\chi_{\Omo\setminus \alp^{-1}\Omo} m\) then the above implies that
\[
    \ue = \uo + \alp^d \Da (\eta\ast \uo).
\]
Therefore, the multiplier allows us to determine the details \(\alp^d \Da (\eta\ast \uo)\) missing in \(\uo\) to obtain \(\ue\).

Our main goal is then to find a multiplier \(m: \Omo\to \C\) such that
\begin{equation}
\label{eq:multiplierConditionOnF}
    \forall\, f\in\mF,\, \xi \in\Omo:\quad  m(\xi) f(\xi) \approx \Da f(\xi)
\end{equation}
in a suitable sense. To ensure that the resulting convolution operator is continuous from \(\LERd\) to itself, we first restrict ourselves to multipliers in \(\LInfOmo\).

%% ============================================================================
\subsection{The induced space}
%% ============================================================================

We endow the space generated by the functions in \(\mF\) with a suitable topology. Let \(\nrmVF{\cdot}\) be a norm on \(\vspan(\mF)\) such that
\begin{equation}
\label{eq:conditionOnVNorm}
    \forall\, v\in \vspan(\mF):\,\, \nrmLE{v} + \sups_{\xi\in \Omo} |v(\xi)| \leq \nrmVF{v}
\end{equation}
and let \(\VF\) be the completion of \(\vspan(\mF)\) under this norm. It is apparent that elements of \(\VF\) belong to \(\LERd\) and that their restrictions to \(\Omo\) are continuous. By construction, the canonical inclusions
\[
    \VF \hookrightarrow \LInfOmo \quad\mbox{and}\quad \VF \hookrightarrow \LERd
\]
are continuous. The map \(\Da:\LERd\to \LERd\) given by
\[
    \Da v(\xi) = v(\alpha\xi)
\]
induces a continuous map \(\Da:\VF \to \LERd\) and any \(m\in\LInfOmo\) defines a continuous linear map \(m:\VF \to \LERd\). Therefore, our problem reduces to finding a multiplier \(m\in \LInfOmo\) for which
\[
    v\in \VF:\,\, \Da v \approx mv.
\]

%% ============================================================================
\subsection{The worst-case error}
%% ============================================================================

Let \(\mu\) be a measure on \(\Rd\) for which there exist constants \(\cmu, \Cmu > 0\) such that
\[
    v\in \VF:\,\, \cmu \nrmLE{v}^2 \leq \int |v(\xi)|^2\, d\mu(\xi) \leq \Cmu \nrmLE{v}^2.
\]

To quantify the discrepancy between \(mv\) and \(\Da v\) for a given \(v\in\VF\) and \(m\in\LInfOmo\) we define the {\em approximation error} as
\[
    E(v, m) = \frac{1}{2}\int_{\Omega_0} |\Da v(\xi) - m(\xi) v(\xi)|^2\, d\mu(\xi).
\]
\begin{proposition}\label{prop:approximationErrorDifferentiable}
    The approximation error
    \[
        E : \VF\times\LInfOmo \to \R
    \]
    is continuous, it is separately convex, and it is Fr\'echet differentiable with its derivative characterized by
    \begin{equation}\label{eq:approximationErrorDerivative}
        dE_{(v,m)}(\Delta v, \Delta m) = \real\left(\int_{\Omo}(\Da v(\xi) - m(\xi) v(\xi))^*  (\Da \Delta v(\xi) - m(\xi) \Delta v(\xi) -  v(\xi) \Delta m(\xi)) d\mu(\xi)\right)
    \end{equation}
    for \(\Delta v\in \VF\) and \(\Delta m \in \LInfOmo\).
\end{proposition}

\begin{proof}[Proof of Proposition~\ref{prop:approximationErrorDifferentiable}]
    \parskip = \pskip

    It is apparent that \(E\) is separately convex. We first prove that~\eqref{eq:approximationErrorDerivative} defines a continuous linear application on \(\VF\times \LInfOmo\). Let \(C_\alpha > 0\) be an upper bound on the operator norm of \(\Da:\VF\to \LERd\). Let \((v,m), (\Delta v,\Delta m)\in \VF\times \LInfOmo\) and write
    \[
        r = \Da v - m v\quad\mbox{and}\quad \Delta r = \Da \Delta v - m \Delta v -  v \Delta m.
    \]
    We have the bounds
    \[
        \int_{\Omo} |r(\xi)|^2\, d\mu(\xi) \leq \Cmu \int_{\Omo} |r(\xi)|^2\, d\xi \leq 2\Cmu C_{\alpha}\nrmVF{v}^2 + 2\Cmu \nrmVF{m}^2 \nrmVF{v}^2
    \]
    and
    \[
        \int_{\Omo} |\Delta r(\xi)|^2\, d\mu(\xi) \leq \Cmu \int_{\Omo} |\Delta r(\xi)|^2\, d\xi \leq 2\Cmu C_{\alpha}\nrmVF{\Delta v} + 2 \Cmu\nrmLInf{m}^2\nrmVF{\Delta v}^2 + 2 \Cmu\nrmVF{v}^2\nrmLInf{\Delta m}^2.
    \]
    Therefore,
    \[
        |dE_{(v,m)}(\Delta v,\Delta m)| \leq 2\Cmu\left(\int_{\Omo} |r(\xi)|^2\, d\mu(\xi)\right)((C_{\alpha} + \nrmLInf{m})\nrmVF{\Delta v} + 2\nrmVF{v}^2\nrmLInf{\Delta m}^2).
    \]
    from where the claim follows. We now prove differentiability from where continuity will follow from~\cite[Theorem~7.1.4]{brown_elements_1970}. Let \((v_0,m_0), (v,m)\in \VF\times \LInfOmo\) and let \(\Delta v = v - v_0\) and \(\Delta m = m - m_0\). Let \(r\) and \(\Delta r\) be as defined before. Then
    \begin{multline*}
        E(v, m) - E(v_0, m_0) - dE_{(v_0,m_0)}(\Delta v, \Delta m) = \frac{1}{2}\int_{\Omo} |\Delta m(\xi)|^2 |\Delta v(\xi)|^2\, d\mu(\xi) \\
        +\:\real\left(\int_{\Omo} r(\xi)^* \Delta m(\xi)\Delta v(\xi)\, d\mu(\xi)\right) + \real\left(\int_{\Omo} \Delta r(\xi)^* \Delta m(\xi)\Delta v(\xi)\, d\mu(\xi)\right)\\
        \leq \frac{1}{2}\int_{\Omo} |\Delta m(\xi)|^2 |\Delta v(\xi)|^2\, d\mu(\xi) \\
        +\: \left(\left(\int_{\Omo} |r(\xi)|^2\, d\mu(\xi)\right)^{1/2} + \left(\int_{\Omo} |\Delta r(\xi)|^2\, d\mu(\xi)\right)^{1/2}\right)\left(\int_{\Omo} |\Delta m(\xi)|^2 |\Delta v(\xi)|^2\, d\mu(\xi)\right)^{1/2}
    \end{multline*}
    Furthermore, from
    \[
        \int_{\Omo} |\Delta m(\xi)|^2 |\Delta v(\xi)|^2\, d\mu(\xi) \leq \Cmu \nrmLInf{\Delta m}^2 \nrmVF{\Delta v}^2 \leq \frac{1}{2}\Cmu \nrmLInf{\Delta m}^4 + \frac{1}{2}\Cmu\nrmVF{\Delta v}^4.
    \]
    it is apparent that for \(\delta > 0\) there exists a constant \(C_\delta > 0\) such that
    \[
        \nrmLInf{m - m_0} + \nrmVF{v - v_0} < \delta
    \]
    then
    \[
        |E(v, m) - E(v_0, m_0) - dE_{(v_0,m_0)}(\Delta v, \Delta m)| = C_\delta \nrmLInf{m - m_0}^2 + C_\delta \nrmVF{v - v_0}^2.
    \]
    We conclude that \(E\) is Fréchet differentiable at \((v_0,m_0)\) with Fréchet derivative \(dE_{(v_0,m_0)}\). Since \((v_0,m_0)\) was arbitrary, we conclude that \(E\) is Fréchet differentiable.
\end{proof}

To evaluate the overall approximation error that \(m\) incurs, we fix a compact subset \(\SVF\) in \(\VF\) to evaluate the {\em worst-case error} over this set
\[
    \Ewc(m) = \sups_{v\in \SVF}\, E(v, m).
\]
The regularity of the approximation error endows the worst-case error of the following properties which will be critical when proving the existence of minimizers.

\begin{proposition}
\label{prop:wcErrorIsProperClosedConvex}
    The worst case error \(\Ewc\) is a continuous convex function with \(\dom(\Ewc) = \LInfOmo\).
\end{proposition}

\begin{proof}[Proof of Proposition~\ref{prop:wcErrorIsProperClosedConvex}]
    \parskip = \pskip

    It is apparent that \(\Ewc\) is convex as it is the supremum of convex functions. Furthermore, \(\Ewc(m)\) is finite everywhere as for every fixed \(m\in\LInfOmo\) the map \(v\mapsto E(v,m)\) is continuous and \(\SVF\) is compact. We will now prove that \(\Ewc\) is bounded in the neighborhood of any \(m_0\in\LInfOmo\). The continuity will then follow from~\cite[Theorem~1, Section~3.2.3]{ioffe_theory_1979}.

    Let \(\set{m_n}_{n\in\N}\) be a sequence such that \(m_n\to m_0\) as \(n\to\infty\) and \(\Ewc(m_n) > n\). Without loss, we can assume that \(\nrmLInf{m_n - m_0} \leq 1\) for all \(n\in\N\). Then, for every \(n\) we may choose \(v_n\in \VF\) such that
    \[
        \Ewc(m_n) \leq E(v_n, m_n) + 2^{-n}.
    \]
    Since \(\SVF\) is compact, it is bounded, and there exist constants \(C_0,C_1 > 0\) such that \(\nrmVF{v_n} \leq C_0\) and \(\nrmVF{\Da v_n} \leq C_1\) for each \(n\in\N\). It follows that
    \[
        E(v_n, m_n) \leq \int_{\Omo}|\D v_k(\xi)|^2\,d\mu(\xi) + \int_{\Omo}(1 + |m_0(\xi)|^2)|v_k(\xi)|^2\,d\mu(\xi) \leq C_{E} C_0 + \Cmu C_1(1+\nrmLInf{m_0}^2) < \infty.
    \]
    This implies that \(\set{\Ewc(m_n)}\) is bounded, and thus \(\Ewc\) is bounded near any \(m_0\).
\end{proof}

%% ============================================================================
\subsection{Existence of worst-case optimal multipliers}
%% ============================================================================

A {\em worst-case optimal multiplier} minimizes the worst-case error, i.e., it is a minimizer for the convex problem
\begin{equation}
    \label{opt:wcErrorMinimizationLInf}
    \begin{aligned}
        & \underset{}{\text{minimize}}
        & & \Ewc(m)
        & \text{subject to}
        & & m\in \LInfOmo.
    \end{aligned}
\end{equation}
However, proving the existence of minimizers is non-trivial. The main obstruction is the existence of unbounded minimizing sequences. To prove the existence of minimizers, it is sufficient for \(\Ewc\) to be {\em coercive}~\cite[Theorem~2.11 and Remark~2.11]{barbu_convexity_2012}, that is, for \(\Ewc\) to satisfy
\[
    m\to\infty\quad\Rightarrow\quad \Ewc(m) \to \infty.
\]
Interestingly, the set
\begin{equation}
    \label{eq:def:nullSet}
    \ZF := \bigcaps_{f\in\mF} \set{\xi\in\Omega_0:\, f(\xi) = 0}
\end{equation}
of common zeros of elements of \(\mF\) hint at whether \(\Ewc\) is coercive. By construction, this is also the set of zeros common to {\em all} the elements of \(\VF\).

\begin{proposition}
\label{prop:model:eStarNotCoercive}
    If \(\ZF\neq \emptyset\) then \(\Ewc\) is not coercive and there are always unbounded minimizing sequences for \(\Ewc\).
\end{proposition}

\begin{proof}[Proof of Proposition~\ref{prop:model:eStarNotCoercive}]
    \parskip = \pskip

    Supoose that \(\ZF\neq\emptyset\) and let \(\xio\in \ZF\). Fix a continuous and compactly supported function \(h: \Rd \to \R\) such that \(0 \leq h \leq 1\), \(h(0) = 1\) and \(h(\xi) = 0\) for \(|\xi| > 1\). Define for \(\delta > 0\)
    \[
        m_{\delta}(\xi) = \sqrt{\frac{1}{\delta^{d}}h\left(\frac{\xi - \xio}{\delta}\right)}
    \]
    By construction \(\nrmLInf{m_{\delta}} = \delta^{-d/2}\) and the collection \(\set{m_\delta}_{\delta > 0}\) is unbounded in \(\LInfOmo\).

    To prove \(\Ewc\) is not coercive, consider the bound
    \[
        \Ewc(m_\delta) \leq \sups_{v\in\SVF}\, \int_{\Omo} |\Da v(\xi)|^2\,d\mu(\xi) + \sups_{v\in\SVF}\, \int_{\Omo} |m_\delta(\xi) v(\xi)|^2\, d\mu(\xi).
    \]
    The first term in the right-hand side is bounded as
    \[
        \sups_{v\in\SVF}\, \int_{\Omo} |\Da v(\xi)|^2\, d\mu(\xi) \leq \Cmu\sups_{v\in\SVF}\, \int |\Da v(\xi)|^2\, d\xi
    \]
    and \(\SVF\) is bounded and \(\Da : \VF \to \LERd\) is continuous. To control the second term in the right-hand side, first observe that
    \[
        \int_{\Omo} |m_\delta(\xi) v(\xi)|^2\, d\mu(\xi) \leq \Cmu \int_{\Omo} |m_\delta(\xi) v(\xi)|^2\, d\xi.
    \]
    We will prove that
    \begin{equation}
    \label{eq:pf:uniformControlSupOverS}
        \limsups_{\delta\to 0}\sups_{v\in\SVF}\, \int_{\Omo}|m_\delta(\xi) v(\xi)|^2 d\xi = 0.
    \end{equation}
    Let \(\set{\delta_n}_{n\in\N}\) and let \(\set{v_n}_{n\in\N}\) be such that
    \[
        \sups_{n\delta \leq 1}\sups_{v\in\SVF}\, \int_{\Omo} |v(\xi)|^2 |m_{\delta}(\xi)|^2 d\mu(\xi) \leq 2^{-n} + \int_{\Omo} |v_n (\xi)|^2 |m_{\delta_n}(\xi)|^2 d\xi.
    \]
    Hence, \(\delta_n\to 0\) as \(n\to \infty\). Since \(\SVF\) is compact, by possibly passing to a subsequence we may assume that there exists \(v_0 \in \SVF\) such that \(v_n \to v_{0}\) as \(n\to\infty\) in \(\VF\). Hence, we may bound
    \[
        \int_{\Omo} |v_n (\xi)|^2 |m_{\delta_n}(\xi)|^2 d\xi \leq 2\int_{\Omo} |(v_n - v_0)(\xi)|^2 |m_{\delta_n}(\xi)|^2 d\xi + 2\int_{\Omo} |v_0 (\xi)|^2 |m_{\delta_n}(\xi)|^2 d\xi.
    \]
    For the first term, observe that
    \[
        \int_{\Omo} |v_n(\xi) - v_0(\xi)|^2  |m_{\delta_n}(\xi)|^2 d\xi \leq\left(\int_{\BE} h(\xi) d\xi\right)\sups_{\xi\in\Omo}|v_n(\xi) - v_0(\xi)|^2 \leq
       \left(\int_{\BE} h(\xi) d\xi\right)\nrmVF{v_n - v_0}^2.
    \]
    For the second, as \(v_0(\xio) = 0\) and \(v_0\) is continuous on the compact set \(\Omo\) we conclude that
    \[
        \int_{\Omo} |v_0(\xi)|^2  |m_{\delta_n}(\xi)|^2 d\xi = \int_{\Omo\cap (\xio+\delta_n \BE)} |v_0(\xio + \delta_n\xi)|^2 d\xi \to 0
    \]
    as \(n\to\infty\). We conclude that~\eqref{eq:pf:uniformControlSupOverS} holds. Hence, by taking the sequence \(\set{m_{\delta_n}}_{n\in\N}\) we conclude that \(\nrmLInf{m_{\delta_n}}\to \infty\) as \(n\to\infty\) whereas
    \[
        \limsups_{n\to\infty}\, \Ewc(m_{\delta_n}) \leq \Cmu\sups_{v\in\SVF}\, \int |\Da v(\xi)|^2\, d\xi
    \]
    whence \(\Ewc\) is not coercive.

    To prove that there exists unbounded minimizing sequences, let \(m_0\in\LInfOmo\) and observe that from the bound
    \begin{align*}
        \Ewc(m_0 + m_\delta) &\leq \sups_{v\in\SVF}\, \int_{\Omo} |\D v(\xi) - m_0(\xi) v(\xi)|^2\, d\mu(\xi) + \sups_{v\in\SVF}\, \int_{\Omo} |m_\delta(\xi) v(\xi)|^2\, d\mu(\xi)\\
        &\leq \Ewc(m_0) + \Cmu \sups_{v\in\SVF}\, \int_{\Omo} |m_\delta(\xi) v(\xi)|^2\, d\xi
    \end{align*}
    and~\eqref{eq:pf:uniformControlSupOverS} it follows that
    \[
        \limsups_{\delta\to 0} \Ewc(m_0 + m_\delta) \leq \Ewc(m_0).
    \]
    If \(\set{m_n}_{n\in\N}\) is an unbounded minimizing sequence there is nothing to prove. If it is a bounded sequence, then we construct a sequence \(\set{\delta_n}_{n\in\N}\) such that \(\delta_n\to 0\) as \(n\to \infty\) and
    \[
        \Ewc(m_n + m_{\delta_n}) \leq \Ewc(m_n).
    \]
    This implies that \(\set{m_n + m_{\delta_n}}_{n\in\N}\) is also a minimizing sequence. Furthermore, it is unbounded, as
    \[
        \nrmLInf{m_n + m_{\delta_n}} \geq \delta_n^{-d/2} - \nrmLInf{m_{n}}.
    \]
    This proves the claim.
\end{proof}

This result may suggest to the reader that \(\ZF = \emptyset\) would imply that \(\Ewc\) is coercive or that it has no unbounded minimizing sequences. However, we have not been able to either prove or disprove this statement. Although \(\LInfOmo\) is the natural space for Fourier multipliers acting on \(\LERd\), proving the existence of minimizers to~\eqref{opt:wcErrorMinimizationLInf} directly seems fairly non-trivial.

Instead of making assumptions on \(\ZF\) our strategy will be to change the space to which the multiplier belongs. Since the elements of \(\VF\) are continuous on \(\Omo\) we will have that \(m v \in \LEOmo\) for any \(v\in \VF\) and \(m\in \LEOmo\). Therefore, we can replace \(\LInfOmo\) by \(\LEOmo\) and still obtain a satisfactory analysis by leveraging the Hilbert space structure of \(\LEOmo\) to determine the existence of minimizers. Remark that both Proposition~\ref{prop:approximationErrorDifferentiable} and Proposition~\ref{prop:wcErrorIsProperClosedConvex} hold if we define the approximation error as
\[
    E : \VF\times \LEOmo \to \R
\]
and the worst-case error as \(\Ewc : \LEOmo\to\R\). Our goal is then to solve
\begin{equation}
    \label{opt:wcErrorMinimizationLE}
    \begin{aligned}
        & \underset{}{\text{minimize}}
        & & \Ewc(m)
        & \text{subject to}
        & & m\in \LEOmo.
    \end{aligned}
\end{equation}
We can strengthen substantially the results in Proposition~\ref{prop:model:eStarNotCoercive}.

\begin{theorem}\label{thm:wcErrorCoerciveInL2}
    If \(\ZF\neq \emptyset\) then \(\Ewc:\LEOmo\to \R\) is not coercive. If there exists a finite or countable subset \(\So\subset S\) such that
    \begin{equation}
    \label{eq:coveringConditionH}
        \forall\, \xi\in \Omo,\,\exists\, h\in \So:\, h(\xi) \neq 0
    \end{equation}
    then \(\Ewc:\LEOmo\to \R\) is coercive.
\end{theorem}

\begin{proof}[Proof of Theorem~\ref{thm:wcErrorCoerciveInL2}]
    \parskip = \pskip

    It is apparent that the first statement follows from Proposition~\ref{prop:model:eStarNotCoercive}.

    To prove the second statement we first make some preliminary remarks. The restriction of every \(h\in\So\) to \(\Omo\) is a continuous function. Since \(\SVF \subset \VF\) is compact, it is bounded in \(\VF\), and thus there exists \(B >0\) such that
    \[
        \sups_{h\in\So}\sups_{\xi\in\Omo}\, |h(\xi)| \leq B.
    \]
    We prove the statement for \(\So\) countable. Write \(\So = \set{h_n:\, n\in\N}\) and let \(p\in\ell^1(\N)\). By Weierstrass' \(M\)-test,
    \[
        \phi := \lims_{N\to\infty}\sums_{n=1}^N |p(n)| |h_n|^2
    \]
    is a non-negative continuous function on \(\Omo\). If it vanished at some \(\xio \in \Omo\) then \(h_n(\xio) = 0\) for any \(n\in\N\) and it would contradict the hypothesis on \(\So\). Consequently,
    \[
        \infs_{\xi\in \Omo} \phi(\xi) > 0.
    \]
    Now, observe that
    \[
        \int_{\Omo}|\Da v(\xi) - m(\xi) v(\xi)|^2\, d\mu(\xi) \geq \int_{\Omo}||m(\xi)| |v(\xi)| - |\Da v(\xi)||^2\, d\mu(\xi).
    \]
    For \(a,b \geq 0\) it can be verified that
    \[
        (at - b)^2 = a^2 t^2 - 2ab t + b^2 \geq \frac{1}{2}a^2 t^2 - b^2
    \]
    for any \(t\geq 0\). It follows that
    \[
        \int_{\Omo}||m(\xi)| |v(\xi)| - |\Da v(\xi)||^2\, d\mu(\xi) \geq \frac{1}{2}\int_{\Omo}|v(\xi)|^2 |m(\xi)|^2\, d\xi - \int_{\Omo} |\D v(\xi)|^2\,d\mu(\xi).
    \]
    Therefore,
    \begin{align*}
        \Ewc(m) &\geq \sups_{n\in\N} E(h_n, m) \\
            &\geq \sums_{n\in\N} p_n E(h_n, m) \\
            &= \frac{1}{2} \sums_{n\in\N} |p_n| \int_{\Omo}|\Da h_n(\xi) - m(\xi) \vo(\xi))|^2\, d\mu(\xi) \\
            &\geq \frac{1}{4}\sums_{n\in\N} |p_n| \int_{\Omo}|h_n(\xi)|^2 |m(\xi)|^2\, d\mu(\xi) - \frac{1}{2}\sums_{n\in\N} |p_n|\int_{\Omo} |\Da h_n(\xi)|^2\,d\mu(\xi)\\
            &\geq \frac{1}{4}\left(\min\nolimits_{\xi\in\Omo}\phi(\xi)\right)\int_{\Omo}|m(\xi)|^2\, d\mu(\xi) - \Cmu \sups_{v\in\SVF}\, \int_{\Omo} |\Da v(\xi)|^2\,d\xi\\
            &\geq \frac{1}{4}\cmu \left(\min\nolimits_{\xi\in\Omo}\phi(\xi)\right)\int_{\Omo}|m(\xi)|^2\, d\xi - \Cmu \sups_{v\in\SVF}\, \int_{\Omo} |\Da v(\xi)|^2\,d\xi.
    \end{align*}
    We conclude that \(\Ewc\) is coercive.
\end{proof}

It is of interest to understand under which conditions we can ensure that~\eqref{eq:coveringConditionH} holds. The following result shows that it follows directly from \(\ZF = \emptyset\) when \(\SVF\) is {\em large enough} to contain at least one scaled element of \(\mF\). We omit the proof for brevity.

\begin{proposition}
    Suppose that \(\ZF = \emptyset\) and that there exists a set of non-zero complex numbers \(\set{a_\gamma:\gamma\in\Gamma}\) such that
    \begin{equation}
    \label{eq:coveringConditionF}
        \forall\, \gamma\in\Gamma:\, a_{\gamma}f_{\gamma} \in \SVF.
    \end{equation}
    Then \(\SVF\) satisfies~\eqref{eq:coveringConditionH} for
    \[
        \So := \set{a_{\gamma} f_{\gamma}:\, \gamma\in\Gamma}.
    \]
\end{proposition}

This suggests that when \(\ZF = \emptyset\) there are natural choices for \(\SVF\) for which the worst-case error will be coercive, namely, embeddings of the unit ball of a weighted \(\ell^1\)-norm of the form
\[
    S := \Lset{\sums_{\gamma\in \Gamma} a(\gamma) f_\gamma:\, \sums_{\gamma\in\Gamma}\tau(\gamma)^{-1}|a(\gamma)| <\infty}
\]
where \(\tau:\Gamma\to\R_+\) is strictly positive and tends to zero at infinity, and embeddings of hyperrectangles of the form
\[
    S := \Lset{\sums_{\gamma\in \Gamma} a(\gamma) f_\gamma:\,\gamma\in\Gamma:\,|a(\gamma)| \leq \tau(\gamma)}
\]
where \(\tau:\Gamma\to\R_+\) is strictly positive and summable.

When~\eqref{eq:coveringConditionH} does not holds we modify \(\Omo\) so that this condition holds. Let \(\rho > 0\) and define
\[
    \ZFr := \set{\xi\in \Omo:\,\, \dist(\xi, \ZF) < \rho}.
\]
Then \(\Omo\setminus\ZFr\) is also compact. If it is also non-empty for \(\rho\) sufficiently small, we may replace \(\Omo\) by \(\Omo\setminus\ZFr\) in our analysis. Our previous results apply in this case.

To conclude, under the hypotheses of Theorem~\ref{thm:wcErrorCoerciveInL2} there exists a minimizer to~\eqref{opt:wcErrorMinimizationLE}. Any minimizer is a {\em worst-case optimal multiplier} and we denote it as \(\mopt\). The minimal worst-case error will be denoted as \(\Ewcopt\). Since \(\Omo\) is compact, we have that \(\LInfOmo\subset \LEOmo\). Hence, a worst-case optimal multiplier is also a minimizer for~\eqref{opt:wcErrorMinimizationLInf} when
\[
    \mopt \in \LInfOmo\cap \LEOmo.
\]
Although we do not know if there are general conditions under which this holds, we will determine mild conditions for which it does hold for a finite collection \(\mF\).

%% ============================================================================
\subsection{The structure of the optimal multipliers}
%% ============================================================================

The minimization problem~\eqref{opt:wcErrorMinimizationLInf} is unconstrained and it suffices to analyze its first-order condition to characterize its minimizers. The function \(\Ewc\) is a continuous convex function with full domain and thus its subdifferential is non-empty everywhere~\cite[Proposition~16.14]{bauschke_convex_2011}. From the first-order optimality condition we conclude that \(\mopt\) is an optimal multiplier if
\[
    \partial \Ewc(\mopt)\ni 0.
\]
We can represent the subdifferential in terms of the error \(E\).

\begin{lemma}\label{lem:wcErrorSubdifferentialL2}
    Let
    \[
        d_m E_{(v,m)}(\Delta m) = -\real\left(\int_{\Omo}(\Da v (\xi) - m(\xi) v(\xi))^*  v(\xi) \Delta m(\xi) d\mu(\xi)\right)
    \]
    denote the partial derivative of \(E\) with respect to \(m\). For every \(m\in\LEOmo\) we have that
    \[
        \partial \Ewc(m) := \wscl(\cvxhull(\set{d_m E_{(v,m)}:\,v\in\SVF:\, E(v,m) = \Ewc(m)})).
    \]
\end{lemma}

\begin{proof}[Proof of Lemma~\ref{lem:wcErrorSubdifferentialL2}]
    \parskip = \pskip

    Let \(\set{E_v:\, v\in\SVF}\) be the collection of functions \(E_v : \LEOmo \to \R\) given by \(E_v(m) = E(v, m)\). By Proposition~\ref{prop:approximationErrorDifferentiable} the functions in the collection \(\set{E_v:\, v\in \SVF}\) are convex, with full domain, and Fréchet differentiable. In particular, Proposition~\ref{eq:approximationErrorDerivative} shows that
    \[
        \partial E_v(m) = \set{d_m E_v}
    \]
    where we have identified the subdifferential with a subset of \(\LEOmo^*\). By Proposition~\ref{prop:wcErrorIsProperClosedConvex} the worst-case error \(\Ewc := \sups_{v\in\SVF} E_v\) has full domain. Since \(\SVF\) is compact, by applying the Valadier-like formula in~\cite{correa_subdifferential_2021,correa_valadier-like_2019} we reach the desired conclusion, proving the lemma.
\end{proof}

To motivate our next result, consider the set
\[
    \SVFopt := \set{v\in S:\, E(v, \mopt) = \Ewcopt}.
\]
Suppose that there are \(v_1,\ldots, v_n\in \SVFopt\) and \(\pi_1,\ldots,\pi_n \geq 0\) such that \(\pi_1 + \ldots + \pi_n = 1\) and
\[
    0 = \sums_{i=1}^n \pi_i d_m E_{(v_i, \mopt)}(\Delta m) = -\real\left(\int_{\Omo}\left(\sums_{i=1}^n\pi_i\Da v_i (\xi) - m(\xi) v_i(\xi))^*  v_i(\xi)\right) \Delta m(\xi) d\mu(\xi)\right)
\]
for any \(\Delta m \in \LEOmo\). It then follows that
\[
    \sums_{i=1}^n\pi_i(\Da v_i - m v_i )^*  v_i = 0
\]
almost everywhere, whence
\[
    \mopt = \frac{\sums_{i=1}^n \pi_i v_i^* \Da v_i}{\sums_{i=1}^n \pi_i |v_i|^2}.
\]
It is illustrative to interpret this as the quotient of two expected values. Let \(\Pi\) be a probability measure on \(2^{\mX_n}\) such that \(\Pi(\set{i}) = \pi_i\). If we let \(X : \mX_n \to S\) be the \(S\)-valued random variable \(X(i) = v_i\) then
\begin{equation}\label{eq:wcOptimalMultiplierAsQuotientOfExpectations}
    \mopt = \frac{\ev\bset{X^* \Da X}}{\ev\bset{|X|^2}}.
\end{equation}
This is the canonical structure for worst-case optimal multipliers.

\begin{theorem}\label{thm:optimalMultiplierQuotientOfAverages}
    There exists a probability measure \(\Pi\) supported on \(\SVF\) such that
    \[
        \mopt = \frac{\int_{\SVF} v^* \Da v\, d\Pi(v)}{\int_{\SVF} |v|^2\, d\Pi(v)}.
    \]
\end{theorem}

\begin{proof}[Proof of Theorem~\ref{thm:optimalMultiplierQuotientOfAverages}]
    \parskip = \pskip

    Let \(S\) be endowed with its metric topology and let \(\BorSVF\) be the \(\sigma\)-algebra generated by its open sets. We denote as \(\PS\) the set of probability measures on \(S\) endowed with the topology of weak convergence of measures. Since \(S\) is compact, it is separable and, by Prokhorov's theorem, this topology is metrizable~\cite[Remark~13.14]{klenke_probability_2020}. Henceforth, we assume that \(\mopt\in \LEOmo\) is fixed.

    The map \(v\mapsto d_m E_{(v,\mopt)}\) is a map of the form \(S\mapsto \LERd^*\). If \(v_1,v_2\in S\) then, from the inequality
    \begin{multline*}
        d_m E_{(v_2,\mopt)}(\Delta m) - d_m E_{(v_1,\mopt)}(\Delta m)
        = -\real\left(\int_{\Omo}(\Da \Delta v (\xi) - \mopt(\xi) \Delta v(\xi))^*  v(\xi) \Delta m(\xi) d\mu(\xi)\right) \\
        -\real\left(\int_{\Omo}(\Da v_2 (\xi) - \mopt(\xi) v_2(\xi))^* \Delta v(\xi) \Delta m(\xi) d\mu(\xi)\right)
    \end{multline*}
    where \(\Delta v = v_2 - v_1\), we conclude that it is continuous when \(\LERd^*\) is endowed with its dual norm. This, together with the fact that  \(\LERd^* \cong \LERd\) is separable, implies that the map \(v\mapsto d_m E_{(v,\mopt)}\) is strongly measurable~\cite[Theorem~2, Section~2.1]{diestel_vector_1998}. Since \(S\) is compact and \(\Pi\) is finite, a straightforward argument also shows that it is Bochner integrable~\cite[Theorem~2, Section~2.2]{diestel_vector_1998}. Therefore, the map \(I_{E,\mopt}:\PS \to \LERd^*\) given by
    \[
        I_{E,\mopt} \pi := \int_S d_m E_{(v,m)}\, d\pi_i(v)
    \]
    is well-defined. Hence, using suitable Dirac masses, we deduce that
    \[
        \cvxhull(\set{d_m E_{(v,m)}:\,v\in\SVF:\, E(v,m) = \Ewc(m)}) \subset \set{I_{E,\mopt}\pi:\, \pi\in\PS}.
    \]
    It suffices to show that the right-hand side is closed under the weak-* topology in \(\LERd^*\). Since the topology of weak convergence in \(\PS\) is metrizable, it suffices to consider an arbitrary sequence \(\set{\pi_i}_{i\in\N}\) in \(\PS\) such that the sequence \(\set{\ell_i}_{i\in\N}\) in \(\LERd^*\) given by \(\ell_i := I_{E,\mopt}\pi_i\) converges to a limit \(\ell_0\in\LERd^*\). Since \(S\) is compact, the collection \(\set{\pi_i}_{i\in \N}\) is tight, whence, without loss, we may assume that it converges weakly to a limit \(\pi_0\in \PS\). Since for any \(\Delta m\in \LEOmo\) the function \(v\mapsto d_mE_{(v,\mopt)}(\Delta m)\) is continuous, we deduce that
    \[
        \lims_{i\to\infty} \int_S d_m E_{(v,m)}(\Delta m)\, d\pi_i(v) = \int_S d_m E_{(v,m)}(\Delta m)\, d\pi_0(v) = \ell_0(\Delta m).
    \]
    from where the claim follows. Hence, we have that
    \[
        \wscl(\cvxhull(\set{d_m E_{(v,m)}:\,v\in\SVF:\, E(v,m) = \Ewc(m)})) \subset \set{I_{E,\mopt}\pi:\, \pi\in\PS}.
    \]
    The remainder of the proof proceeds using the same arguments that yield~\eqref{eq:wcOptimalMultiplierAsQuotientOfExpectations}.
\end{proof}

The theorem could be refined by constraining the support of \(\Pi\) to \(\SVFopt\). However, this refinement is not necessary for our interpretation of the structure of the worst-case optimal multiplier: the expectation can be interpreted as a pointwise projection. Let \(\Pi\) be as in the theorem and recall that the evaluation on any \(\xi\in \Omo\) defines a continuous linear funcional on \(V\). Then we can define the random variables \(\Xi, \Xi_{\alpha} : \SVF\to \C\) given by \(\Xi(v) = v(\xi)\) and \(\Xi_{\alpha}(v) = v(\alpha\xi)\).  If \(\ev\bset{|\Xi(v)|^2} > 0\) then we can write
\[
    \mopt(\xi) = \frac{\ev\bset{\Xi^* \Xi_{\alpha}}}{\ev\bset{|\Xi|^2}}.
\]
Therefore, the worst-case optimal multiplier can be interpreted as finding the coefficient associated to the orthogonal projection of the random variable \(\Xi_{\alpha}\) onto the random variable \(\Xi\).

%% ============================================================================
\section{Adaptive multipliers for a finite collection of functions}
\label{sec:adaptiveMultipliersFiniteCollection}
%% ============================================================================

In practical applications the set \(\mF\) is often finite. Since its elements are linearly independent by assumption, if the collection contains \(n\) elements then \(\VF\) has dimension \(n\) and we may represent every \(v\in \VF\) as
\[
    v = \sums_{i=1}^n c_i(v) f_i
\]
where \(\Gamma = \set{1,\ldots, n}\) and \(\set{c_i}_{i=1}^n\) are the associated coordinate functionals. We interpret \(\Cn\) as the complex Hilbert space with inner product
\[
    \bprodlEn{\vx}{\vy} := \sums_{k=1}^n x_k^* y_k
\]
and we denote as \(\lEn\) the corresponding real Hilbert space. In this case, it is natural to endow \(\VF\) with the norm
\[
    \nrmVF{v} := \left(\sums_{i=1}^n |c_{i}(v)|^2\right)^{1/2} = \nrmlEn{\vc(v)},
\]
where \(\vc : \VF \to \Cn\) is coordinate map. It is apparent that this norm satisfies~\eqref{eq:conditionOnVNorm}. From now on, it will be useful to define the functions \(\vf,\Da\vf :\Rd\to\Cn\) as
\[
    \vf := \bmtx{f_1 \\ \vdots \\ f_n}\quad\mbox{and}\quad \Da\vf := \bmtx{\Da f_1 \\ \vdots \\ \Da f_n}.
\]
We now instantiate the results in Section~\ref{sec:adaptiveMultipliers} to this case. This will yield further insight into the structure of the worst-case optimal multipliers.

%% ============================================================================
\subsection{\(\Sigma\)-mutlipliers}
%% ============================================================================

The approximation error can be represented in terms of the coordinates as
\begin{align*}
    E(v,m) &:= \frac{1}{2}\int_{\Omo} \left|\Da \left(\sums_{i=1}^n c_i(v) f_i(\xi)\right) - m(\xi) \sums_{i=1}^{n} c_i(v)f_i(\xi)\right|^2\, d\mu(\xi) \\
        &= \frac{1}{2}\int_{\Omo} \left|\sums_{i=1}^n c_i(v) (\Da f_i(\xi) - m(\xi) f_i(\xi))\right|^2\, d\mu(\xi) \\
        &= \frac{1}{2}\int_{\Omo} |\iprodlEn{\vc(v)}{\Da\vf(\xi) - m(\xi)\vf(\xi)}|^2\, d\mu(\xi) \\
        &=: e(\vc(v), m).
\end{align*}
Furthermore, we can write
\[
    e(\vc, m) = \frac{1}{2}\int_{\Omo} |\iprodlEn{\vc(v)}{\Da\vf(\xi) - m(\xi)\vf(\xi)}|^2\, d\mu(\xi) = \iprodlEn{\vc}{\vG(m)\vc}
\]
where
\[
   \vG(m) = \frac{1}{2}\int_{\Omo} (\Da \vf(\xi) - m(\xi) \vf(\xi))(\Da \vf(\xi) - m(\xi) \vf(\xi))^*\, d\mu(\xi).
\]
Let \(\Hnn\) be the real Hilbert space of Hermitian matrices of size \(n\times n\) endowed with the inner product
\[
    \iprodF{\vX}{\vY} := \real(\bprodF{\vX}{\vY})\quad\mbox{where}\quad \bprodF{\vX}{\vY} := \trace(\vX^*\vY).
\]
It follows that \(\vG : \LEOmo \to \Hnn\). It will be useful to characterize the derivative of this function.

\begin{proposition}\label{prop:gramMatrixDifferentiable}
    The function \(\vG : \LEOmo \to \Hnn\) is Fréchet differentiable with
    \[
        d\vG_m(\Delta m) = -\frac{1}{2}\int_{\Omo} (\Delta m(\xi)^* (\Da \vf(\xi) - m(\xi) \vf(\xi)) \vf(\xi)^* + \Delta m(\xi) \vf(\xi) (\Da \vf(\xi) - m(\xi) \vf(\xi))^*)  d\mu(\xi)
    \]
\end{proposition}

\begin{proof}[Proof of Proposition~\ref{prop:gramMatrixDifferentiable}]
    \parskip = \pskip

    The proof follows from the expansion
    \begin{multline*}
        \vG(m + \Delta m) = \vG(m) + \int_{\Omo} |\Delta m(\xi)|^2 \vf(\xi)\vf(\xi)^* d\mu(\xi) \\
        - \int_{\Omo} \Delta m(\xi)^* (\Da \vf(\xi) - m(\xi) \vf(\xi)) \vf(\xi)^*d\mu(\xi) - \int_{\Omo} \Delta m(\xi) \vf(\xi) (\Da \vf(\xi) - m(\xi) \vf(\xi))^*  d\mu(\xi).
    \end{multline*}
    and an application of similar arguments as those used in the proof of Proposition~\ref{prop:approximationErrorDifferentiable}. We omit the details for brevity.
\end{proof}

To any choice of compact set \(\SVF\) we can associate the compact set \(\CVF := \vc^{-1}(\SVF)\) so that the worst-case error becomes
\[
    \Ewc(m) = \sups_{v\in S}\, E(v, m) = \sups_{v\in S}\, e(\vc(v), m) = \sups_{\vc\in C}\, e(\vc, m) = \sups_{\vc\in C}\, \iprodlEn{\vc}{\vG(m)\vc}.
\]
In this case, we have the following corollary.

\begin{corollary}\label{cor:wcOptimalMultipliersFiniteDimensional}
    If \(\mF\) is finite and
    \[
        \sums_{i=1}^n |f_i|^2 > 0
    \]
    on \(\Omo\) then there exists at least one worst-case optimal multiplier. Furthermore, for any optimal multiplier \(\mopt\) there exists \(\vSig\in \Hnn\) positive semidefinite such that
    \[
        \mopt = \frac{\bprodlEn{\vf}{\vSig \Da\vf}}{\bprodlEn{\vf}{\vSig\vf}}.
    \]
\end{corollary}

\begin{proof}[Proof of Corollary~\ref{cor:wcOptimalMultipliersFiniteDimensional}]
    \parskip = \pskip

    The existence of worst-case optimal multipliers follows from the hypothesis and Theorem~\ref{thm:wcErrorCoerciveInL2}. It will be useful to define the auxiliary convex function \(\vphi: \Hnn \to \R\)
    \[
        \vphi(\vX) := \sups_{\vc\in C}\, \trace(\vc\vc^* \vX)
    \]
    and \(\vGopt = \vG(\mopt)\) where \(\mopt\) is an optimal multiplier. Since \(\vphi\) is the support function of a compact set, we have that~\cite[Theorem~2.4.18]{zalinescu_convex_2002}
    \[
        \partial \vphi(\vGopt) = \cl(\cvxhull(\set{\vc\vc^*:\,\, \iprodF{\vc\vc^*}{\vGopt} = \Ewc(\mopt)})).
    \]
    From Proposition~\ref{prop:wcErrorIsProperClosedConvex} it follows that \(\Ewc\) is convex with full domain. Therefore, it is Gâteaux (or directionally) differentiable everywhere~\cite[Proposition~17.2]{bauschke_convex_2011}. The optimality condition at \(\mopt\) implies that
    \[
        \Delta m \in\LERd:\,\, \Ewc'(\mopt;\Delta m) \geq 0.
    \]
    However, this implies that~\cite[Theorem~17.9]{bauschke_convex_2011}
    \[
        0 \leq \sup_{\vSig \in \partial\vphi(\vGopt)}\, \iprodF{\vSig}{d\vG_{\mopt}(\Delta m)} = \sup_{\vSig \in \partial\vphi(\vGopt)}\, \iprodF{d\vG_{\mopt}^*(\vSig)}{\Delta m}
    \]
    where \(d\vG_{\mopt}^*\) is the adjoint of the derivative \(d\vG_{\mopt}\). It follows that there exists \(\vSig\in\partial\vphi(\vGopt)\) such that \(d\vG_{\mopt}^*(\vSig) = 0\). From Proposition~\ref{prop:gramMatrixDifferentiable} we deduce that
    \begin{align*}
        0 &= \iprodF{\vSig}{d_m \vG_{\mopt}(\Delta m)} \\
        &= -\frac{1}{2}\int_{\Omo}\real(\Delta m(\xi)^* \bprodF{\vf(\xi)}{\vSig(\Da\vf(\xi) - \mopt(\xi)\vf(\xi))}) d\mu(\xi)  \\
        &\quad\:-\frac{1}{2}\int_{\Omo}\real(\Delta m(\xi) \bprodF{\vSig(\Da\vf(\xi) - \mopt(\xi)\vf(\xi))}{\vf(\xi)})d\mu(\xi)\\
        &= -\real\left(\int_{\Omo} \Delta m(\xi) \bprodF{\vSig(\Da\vf(\xi) - \mopt(\xi)\vf(\xi))}{\vf(\xi)}\, d\mu(\xi)\right)
    \end{align*}
    for any \(\Delta m\in \LERd\). We conclude that
    \[
        \bprodF{\vSig(\Da\vf- \mopt\vf)}{\vf} = 0
    \]
    from where the claim follows.
\end{proof}

To interpret the result it will be useful to define the positive semidefinite form
\[
    \bprodS{\vx}{\vy} := \bprod{\vx}{\vSig\vy}
\]
and to denote as \(\nrmS{\cdot}\) the induced seminorm. Then a worst-case optimal multiplier can be represented as
\[
    \mopt(\xi) = \frac{\bprodS{\vf(\xi)}{\Da\vf(\xi)}}{\nrmS{\vf(\xi)}^2}
\]
for almost every \(\xi\in\Omo\). When \(\vSig\) is positive-definite, this is precisely the coefficient for the orthogonal projection of \(\Da\vf(\xi)\) onto \(\vf(\xi)\) with respect to the complex inner product \(\bprodS{\cdot}{\cdot}\). Therefore, minimizing the worst-case error {\em implicitly} finds a complex inner product on \(\Cn\) for which the optimal multiplier yields the coefficients for the {\em pointwise} orthogonal projection of \(\Da \vf\) onto \(\vf\).

This motivates the following definition.

\begin{definition}[\(\Sigma\)-multiplier]
    We say that \(m:\Omo\to \C\) is a \(\Sigma\)-multiplier if there exists \(\vSig\in\Hnn\) positive semidefinite such that
    \[
        m = \frac{\bprodS{\vf}{\Da\vf}}{\nrmS{\vf}}.
    \]
    If \(\vSig\in\Hnn\) is positive semidefinite, then we denote as \(m_{\Sig}\) the \(\Sigma\)-multiplier defined by the above expression.
\end{definition}

Any \(\Sigma\)-multiplier can be interpreted as the quotient of expected values as in~\eqref{eq:wcOptimalMultiplierAsQuotientOfExpectations}. Let \(\pi_1 \geq \ldots \geq \pi_n \geq 0\) be the eigenvalues of \(\vSig\) normalized so that their sum equals one, and let \(\Pi\) be the induced probability measure on \(2^{\mX_n}\). The eigenvectors of \(\vSig\) induce functions \(v_1,\ldots, v_n \in \VF\) from which we can define the random variable \(X: \mX_n \to \VF\) with \(X(i) = v_i\) whence
\[
    m = \frac{\sums_{i=1}^n \pi_i v_i^* \Da v_i}{\sums_{i=1}^n \pi_i |v_i|^2} = \frac{\ev\bset{X^* \Da X}}{\ev\bset{|X|^2}}.
\]
From now on, for any \(\vSig\in \Hnn\) positive semidefinite we let \(m_{\Sigma}\) denote the \(\Sigma\)-multiplier associated to it.

A \(\Sigma\)-multiplier is the quotient of two real-analytic functions, and thus may have singularities at some points on its domain. Under rather mild conditions we can show that a \(\Sigma\)-multiplier is bounded on \(\Omo\).
\begin{proposition}\label{prop:sigmaMultiplierBoundedForPSDMatrix}
    Suppose that \(\ZF = \emptyset\). If \(\vSig\in\Hnn\) is positive definite then \(m_{\Sigma}\) is real-analytic on \(\Omo\).
\end{proposition}

\begin{proof}[Proof of Proposition~\ref{prop:sigmaMultiplierBoundedForPSDMatrix}]
    \parskip = \pskip

    The result follows from the following observation. Since \(\vSig\) is positive-definite, there exists \(c > 0\) such that \(\nrmS{\vz} \geq c\nrmlEn{\vz}\) for any \(\vz\in\Cn\). Since \(\ZF = \emptyset\) this implies that \(\nrmS{\vf} \neq 0\) on \(\Omo\). Therefore, the denominator does not vanish on a neighborhood on \(\Omo\) whence \(m_{\Sigma}\) is real-analytic on \(\Omo\).
\end{proof}

%% ============================================================================
\subsubsection{Average-error criteria}
%% ============================================================================

When \(\mF\) is finite we have the representation
\[
    \Ewc(m) = \sups_{\vc \in C}\, \iprodlEn{\vc}{\vG(m)\vc}.
\]
This yields insight into two different forms of worst-case error for which the optimal multipliers are still \(\Sigma\)-multipliers. For the first, remark that we can write this equivalently as
\begin{equation}\label{eq:wcErrorGeneralized}
    \Ewc(m) = \sups_{\vW \in W}\, \iprodF{\vW}{\vG(m)} = \sigma_W(\vG(m))
\end{equation}
for the set
\[
    W = \set{\vc\vc^*:\,\, \vc\in C}
\]
where \(\sigma_W\) denotes the support function of \(W\). However, we can still define the worst-case error for compact sets \(W\in \Hnn\) that are not of the above form. The following result shows that by minimizing a worst-case error defined in this manner the optimal multiplier is \(\Sigma\)-multiplier. We omit the proof for brevity.

\begin{proposition}
    Let \(W\subset \Hnn\) be compact and suppose that it contains a neighborhood of the origin. Let \(\Ewc : \LEOmo \to \R\) be defined as in~\eqref{eq:wcErrorGeneralized}. Then there exists at least one optimal multiplier for \(\Ewc\) and every optimal multiplier is a \(\Sigma\)-multiplier.
\end{proposition}

In the finite-dimensional case it is also usual to define notions of average error instead of worst-case error. In particular, we may define
\[
    e_A(m) := \ev_{\vc\sim \Pi}\bset{e(\vc, m)} = \ev_{\vc\sim \Pi}\bset{\iprodF{\vc\vc^*}{\vG(m)}}
\]
where \(\Pi\) is a probability measure on \(\Cd\). If \(\vc\sim \Pi\) is centered, then
\[
    e_A(m) = \iprodF{\vC}{\vG(m)}
\]
for the covariance matrix \(\vC\). The eigenvectors of the covariance matrix \(\vC\) allows us to find \(v_1,\ldots, v_n\in\VF\) such that
\[
    e_A(m) = \frac{1}{2}\sums_{i=1}^n\sigma_i^2\int_{\Omo}|\Da v_i(\xi) - m(\xi) v_i(\xi)|^2 d\mu(\xi).
\]
where \(\sigma_1^2 \geq \ldots \geq \sigma_n^2 \geq 0\) are the eigenvalues of \(\vC\). It is straightforward to verify that the optimal multiplier exists. In this case, we call it the {\em average-case optimal multiplier}. Note that it is also a \(\Sigma\)-multiplier. Hence, every \(\Sigma\)-multiplier for \(\vSig\) positive definite is the average-case optimal multiplier for \(\vc\sim \mN(0,\vSig)\).

Interestingly, when \(\vc\sim \mN(0,\vI_n)\) is a real Gaussian vector then
\[
    \mopt = \frac{\sums_{i=1}^n f_i^* \Da f_i}{\sums_{i=1}^n |f_i|^2} = m_I.
\]
We call this the {\em trace multiplier}. Its advantage is that it is easily computed from the collection \(\mF\) and its optimality can be rigorously characterized by the above arguments. We have the following important result. We omit its proof for brevity.

\begin{proposition}
    If \(\ZF = \emptyset\) then the trace multiplier is bounded on \(\Omo\).
\end{proposition}

%% ============================================================================
\subsection{An optimization algorithm}
%% ============================================================================

Let
\[
    \Ewc(m) = \sigma_W(\vG(m)).
\]
Without loss, we may assume that \(W\) is convex and compact, and that it contains the origin. The proof of Corollary~\ref{cor:wcOptimalMultipliersFiniteDimensional} suggests the following method to find an optimal multiplier. We introduce the main ideas informally first. The optimal multiplier \(\mopt\) is characterized by
\[
    \exists\, \vSig\in \partial\sigma_W(\vG(\mopt)):\,\, d\vG_{\mopt}^*(\vSig) = 0.
\]
The inclusion \(\vSig\in \partial\sigma_W(\vG)\) can be represented for any \(\tG,\tS > 0\) as
\[
    \vSig \in \partial\sigma_W(\vG) \quad\Leftrightarrow\quad \tau_G\vG \in \partial\ind_{W}(\vSig)\quad\Leftrightarrow\quad \vSig = \proj_{W}(\tG\vG + \tS \vSig).
\]
The first equivalence follows from~\cite[Proposition~16.9]{bauschke_convex_2011} by using the fact that \(\sigma_W\) is proper, lower semicontinuous and convex, and Fenchel-Moreau's theorem~\cite[Theorem~13.32]{bauschke_convex_2011}. The second equivalence follows from the fact that \(\partial \ind_W = N_W\)~\cite[cf. Proposition~16.13]{bauschke_convex_2011}. Furthermore, the proof of Corollary~\ref{cor:wcOptimalMultipliersFiniteDimensional} also shows that
\[
    d\vG_m^*(\vSig) = 0\quad\Leftrightarrow\quad m = \frac{\bprodS{\vf}{\Da\vf}}{\nrmS{\vf}}.
\]
Therefore, the optimality condition for \(\mopt\) is equivalent to
\[
    \exists\, \vSig\in\Hnn:\,\, \mopt = \frac{\bprodS{\vf}{\Da\vf}}{\nrmS{\vf}^2}\quad\mbox{and}\quad \vSig = \proj_{W}(\tS \vSig + \tG \vG(\mopt)).
\]
This characterization is {\em finite-dimensional}, depending only on the existence of a suitable Hermitian matrix \(\vSig\). Implicit in this characterization is the {\em fixed-point iteration}
\begin{align}\label{eq:fixedPointIteration}
\begin{split}
    m^{(k+1)} &= \frac{\bprod{\vf}{\Da\vf}_{\Sig^{(k)}}}{\nrm{\vf}^2_{\Sig^{(k)}}}\\
    \vSig^{(k+1)} &= \proj_{W}(\tS\vSig^{(n)} + \tG\vG(m^{(k+1)}))
\end{split}
\end{align}
Although our empirical results show that this iteration works well in practice, proving its convergence requires showing that any choice \(\vSig\in W\) yields \(m_{\Sigma} \in\LERd\). Otherwise, the iteration becomes undefined. Instead of attempting to do this, we regularize this fixed-point iteration as follows. We will assume that \(W\) is {\em invariant under unitary conjugations}
\[
    \forall\, \vU\in \Unn,\,\, \vX\in W:\,\,  \vU\adj\vX \vU \in W
\]
where \(\Unn\) is the set of unitary matrices in \(\Cnn\). From this it follows that there exists a set \(D_{W}\subset \Rn\) invariant under permutations such that
\begin{equation}\label{eq:spectrallyInvariantRepresentation}
    W = \set{\vU\diag(\vd) \vU\adj:\,\, \vd\in D_{W}}.
\end{equation}
We will assume that \(D_{W}\) is convex and that it is {\em orthosymmetric}~\cite{donoho_minimax_1990}
\begin{equation}\label{eq:orthosymmetricSet}
    \forall\, \vd\in D_{W},\,\,\vs \in \set{-1,1}^n:\,\, \diag(\vs) \vd\in D_{W}.
\end{equation}
Let \(\delta > 0\) and consider the regularized optimization problem
\[
    \begin{aligned}
        & \underset{}{\text{minimize}}
        & & \sigma_W(\vG(m)) + \delta\trace(\vG(m))
        & \text{subject to}
        & & m\in \LEOmo.
    \end{aligned}
\]
Using the same arguments as before, we see that it induces the regularized fixed-point iteration
\begin{align}\label{eq:fixedPointIterationRegularized}
\begin{split}
    m^{(k+1)} &= \frac{\bprod{\vf}{\Da\vf}_{\delta I + \Sig^{(k)}}}{\nrm{\vf}^2_{\delta I + \Sig^{(k)}}}\\
    \vSig^{(k+1)} &= \proj_{W}(\tS\vSig^{(n)} + \tG\vG(m^{(k+1)}))
\end{split}
\end{align}
for \(k \in \No\). We typically choose \(\vSig^{(0)}= \vI\) so that that \(m^{(0)}\) corresponds to the trace multiplier, which by Proposition~\ref{prop:sigmaMultiplierBoundedForPSDMatrix} is bounded when \(\ZF = \emptyset\). The proof of the following theorem is deferred to Appendix~\ref{apx:pf:thm:iterationConvergence} which also provides precise bounds for the parameters \(\delta,\tS\) and \(\tG\).

\begin{theorem}\label{thm:iterationConvergence}
    Suppose that
    \[
        \trace(\vSig^{(0)}) \leq \sup\set{\trace(\vX):\,\, \vX\in W} < \infty.
    \]
    Then there exists a constant \(C_{\textrm{FPI}} > 0\) independent of \(\delta,\tS,\tG\) such that if \(\tS,\tG\) satisfy
    \[
        \tS + C_{\textrm{FPI}}\frac{\tG}{\delta^3} < 1
    \]
    then fixed-point iteration~\eqref{eq:fixedPointIterationRegularized} converges to a unique fixed point.
\end{theorem}

%% ============================================================================
\subsection{Multipliers and families of translates}
\label{sec:adaptiveMultipliersTranslates}
%% ============================================================================

Extrapolation in frequency using multipliers is often interpreted as a convolution by a filter that enhances the details contained in a low-frequency approximation. In the trivial case when \(\mF\) consists of a single element \(f\) that does not vanish in \(\Omo\) then the multiplier that exactly extrapolates the frequency information is simply
\begin{equation}\label{eq:perfectExtrapolationSingleFunction}
    \mf = \frac{f^*\Da f}{|f|^2} = \frac{\Da f}{f}
\end{equation}
Hence, the associated convolution filter allows us to enhance some details in the original function up to a factor \(\alp\). In fact, using the notation in Section~\ref{sec:adaptiveMultipliers}, we see that
\[
    f|_{\Ome}(\xi) = \Da f|_{\alp\Omo \setminus \Omo}\left(\frac{\xi}{\alp}\right) = (\mf f)|_{\alp\Omo \setminus \Omo}\left(\frac{\xi}{\alp}\right)
\]
whence
\begin{equation}\label{eq:perfectSuperresSingleFunction}
    \ue = \uo + \alp^d \Da(\etaf\ast \uo)
\end{equation}
where \(\ue\) is the inverse Fourier transform of \(f|_{\Omo}\) and \(\etaf\) is the inverse Fourier transform of \(\chi_{\Omo\setminus\alp^{-1}\Omo} \mf\). Since a convolution is {\em translation equivariant}, we would expect that there is a representation like~\eqref{eq:perfectSuperresSingleFunction} for any translate \(\uxo\) of \(u\). However, if \(\fxo(\xi) = e^{-2\pi i \xi\cdot \xo}f(\xi)\) then~\eqref{eq:perfectExtrapolationSingleFunction} yields
\[
    \frac{\fxo(\xi)^*\Da \fxo(\xi)}{|\fxo(\xi)|^2} = e^{-2\pi i (\alp - 1) \xi\cdot \xo} \mf(\xi).
\]
In this case, the associated filter is
\[
    \int \chi_{\Omo\setminus \alp^{-1}\Omo}(\xi) \mf(\xi) e^{-2(\alpha - 1)\pi i \xi\cdot \xo} e^{2\pi i \xi\cdot x}\, d\xi = \frac{1}{n} \etaf(x + (\alp-1)\xo) = \tau_{(\alp - 1)\xo} \etaf(x)
\]
whence~\eqref{eq:perfectSuperresSingleFunction} becomes
\begin{align*}
    \uxoe &= \uxoo + \alp^d \Da (\tau_{(\alp - 1)x_0} \etaf\ast \uxoo) \\
    &= \uxoo + \alp^d \Da\tau_{(\alp - 1)x_0}( \etaf\ast \uxoo)\\
    &= \uxoo + \alp^d \tau_{(\alp - 1)x_0 /\alp}\Da( \etaf \ast \uxoo).
\end{align*}
where \(\uxoo\) is the inverse Fourier transform of \(\fxo|_{\Omo}\). The translation in the second term is a consequence of the fact that extrapolation in frequency using multipliers leads to the details of a function contracted by a factor \(\alp^{-1}\). In particular, if we used \(\mf\) in~\eqref{eq:perfectExtrapolationSingleFunction} instead of \(m\) we would have obtained
\begin{equation}\label{eq:perfectSuperresMisplacedDetails}
    \uxoe \approx \uxoo + \alp^d \Da( \etaf \ast \uxoo).
\end{equation}
In this expression, the details are correct but {\em misplaced}. This effect can be studied more generally using \(\Sig\)-multipliers when \(\mF\) consists of translates of a single function, i.e.,
\[
    \mF = \set{f,f_{x_1},\ldots, f_{x_{n-1}}}\quad\mbox{with}\quad f_{x_k}(\xi) = e^{-2\pi i \xi\cdot x_k}f(\xi)\quad\mbox{for}\quad k\in\mX_{n-1}.
\]
Any \(\Sig\)-multiplier {\em must} compromise when approximating the details of each function. Since we assume that \(f\) does not vanish in \(\Omo\) we have that \(\ZF = \emptyset\). By assuming that \(\xo = 0\) we can write \(\vf = f \ve\), where \(\ve:\Omo\to \Cn\) has components \(e_k(\xi) = e^{-2\pi i \xi\cdot x_k}\). It follows that
\[
   m_{\Sig} = \frac{\bprodS{\vf}{\Da \vf}}{\nrmS{\vf}^2} = \frac{f^* \Da f}{|f|^2}\frac{\bprodS{\ve}{\Da\ve}}{\nrmS{\ve}^2} = \mf \frac{\bprodS{\ve}{\Da\ve}}{\nrmS{\ve}^2}.
\]
It is illustrative to choose some specific values for \(\vSig\). On one hand, if \(\vSig = \ve_1 \ve_1^*\) then
\[
    m_{\Sig} = \frac{f^* \Da f}{|f|^2} = \mf
\]
and we recover the multiplier in~\eqref{eq:perfectExtrapolationSingleFunction}. In this case, extrapolation in frequency would lead to~\eqref{eq:perfectSuperresMisplacedDetails} for each \(u_{x_k,\Omo}\). On the other, when \(\vSig = \vI\) then the trace multiplier becomes
\[
    m_{I}(\xi) = \frac{1}{n}\mf(\xi) \sums_{k=1}^{n-1} e^{2(\alpha -1)\pi i \xi\cdot x_k}.
\]
In this case, the multiplier \(\mf\) is modulated by complex exponentials representing translates by factors \((\alp - 1)x_k\). In particular, it is the {\em average} of the multipliers that achieve exact extrapolation for each \(f_{x_k}\). In this case, it is straightforward to see that
\begin{align*}
    u_{k,\Ome} &\approx u_{k,\Omo} + \frac{\alp^d}{n} \tau_{(\alp - 1)x_k/\alp }\Da (\etaf \ast u_{k,\Omo}) + \frac{\alp^d}{n} \sums_{j\neq k}\tau_{(\alp - 1)x_j/\alp }\Da (\etaf\ast u_{k,\Omo}).
\end{align*}
While the second term in the right-hand side corresponds to the correct details scaled by a factor \(1/n\), the third corresponds to an artifact: the correct details scaled by a factor \(1/n\) are replicated at each \((\alp - 1) x_j / \alp\).

In general, a \(\Sig\)-multiplier will take the form
\[
    m_{\Sig}(\xi) =  \mf(\xi)\frac{\sums_{k=1}^r P_k(\xi)^* P_k(\alpha\xi)}{\sums_{k=1}^r |P_k(\xi)|^2}
\]
where \(r\) is the rank of \(\vSig\) and \(P_1,\ldots, P_k\) are functions of the form
\[
    P_k(\xi) = \sums_{j=0}^{n-1} a_{k,j} e^{-2\pi i \xi\cdot x_j}.
\]
When finding the worst-case or average case optimal multipliers, an implicit selection of these functions is made. Whether there is a strategy to mitigates these effects is an open question.

%% ============================================================================
\subsection{Multipliers and discrete approximations}
\label{sec:adaptiveMultipliersDiscreteApproximation}
%% ============================================================================

In practical applications, such as image processing, the collection in \(\mF\) is induced by interpolating discrete data. Let
\[
    Q_N := \set{q\in \Zd:\,\, q_i \in \set{0,\ldots, N-1}}
\]
and suppose that the discrete data is given by functions
\[
    a_1,\ldots, a_n : Q_N \to \C.
\]
We assume that the discrete data is interpolated by the formula
\begin{equation}\label{eq:discreteData:interpolationInSpace}
    u_k(x) = \sums_{q\in Q_N} a_k(q) \vphi\left(\frac{x - \Delta x q}{\Delta x}\right)
\end{equation}
for some function \(\vphi\in\LERd\). In this case,
\begin{equation}\label{eq:discreteData:interpolationInFrequency}
    f_k(\xi) = \wh{\vphi}(\xi)\sums_{q\in Q_N} a_k(q) e^{-2\pi i \Delta x q \cdot \xi}
\end{equation}
If we let \(\ve:\Omo\to \C^p\) with \(p = N^d\) be the function with components \(e_k(\xi) = e^{-2\pi i \xi\cdot k}\) then there exists a \(n\times p\) matrix \(\vA\) such that
\[
    \vf = \vA D_{\Delta x}\ve.
\]
In this case,
\[
    m_{\Sig} = \frac{D_{\Delta x} \wh{\vphi} D_{\alpha\Delta x}\wh{\vphi}}{|D_{\Delta x} \wh{\vphi}|^2} \frac{\bprodlEn{\vf}{\vSig\vf}}{\bprodlEn{\vf}{\vSig\vf}} = \frac{D_{\Delta x} \wh{\vphi} D_{\alpha\Delta x}\wh{\vphi}}{|D_{\Delta x} \wh{\vphi}|^2} \frac{\bprodlEp{D_{\Delta x}\ve}{\vA\adj\vSig\vA D_{\alpha\Delta x}\ve}}{\bprodlEp{D_{\Delta x}\ve}{\vA\adj\vSig\vA D_{\Delta x}\ve}}.
\]
The \(\Sig\)-multiplier is the product between the multiplier in~\eqref{eq:perfectExtrapolationSingleFunction} for \(D_{\Delta x} \vphi\) and a term that depends only on a family of complex exponentials. While \(\vSig\) is \(n\times n\) the matrix \(\vA\adj\vSig \vA\) is \(p\times p\). Suppose that \(\vA\) is full-rank. When \(n\geq p\) the map \(\vSig \mapsto \vA\adj\vSig\vA\) is surjective. This follows from the fact that \(\vA\) admits a \(p\times n\) pseudo-inverse \(\vA^{+}\) from the left. Therefore, for any \(p\times p\) matrix \(\vSig'\) we may define \(\vSig = \vA^+{\adj} \vSig' \vA^{+}\) to obtain
\[
    \vA\adj\vSig\vA = \vSig'.
\]
In other words, in the regime \(n\leq p\) the samples do not constrain the choice of multiplier. In fact, it suffices to use as multiplier
\[
    m_{\Sig'} = \frac{D_{\Delta x} \wh{\vphi} D_{\alpha\Delta x}\wh{\vphi}}{|D_{\Delta x} \wh{\vphi}|^2} \frac{\bprodlEp{D_{\Delta x}\ve}{\vSig' D_{\alpha\Delta x}\ve}}{\bprodlEp{D_{\Delta x}\ve}{\vSig' D_{\Delta x}\ve}}.
\]
In contrast, when \(n > p\) then the above argument does not work. In fact, imposing \(\vA\adj \vSig \vA = \vSig'\) constrains the nullspace of \(\vSig'\) to contain that of \(\vA\) or, equivalently, the range of \(\vSig'\) must be contained in the range of \(\vA\adj\). Therefore, only when the number of elements in the collection exceeds the number of discretization points the collection enforces additional constraints on the optimal multiplier.

%% ============================================================================
\section{Connections to multiresolution}
\label{sec:connectionsToMultiresolution}
%% ============================================================================

%% ============================================================================
\subsection{Multiresolutions}
%% ============================================================================

Although multiresolution analysis (MRA) was introduced by Mallat~\cite{mallat_multiresolution_1989} and Meyer~\cite{meyer_ondelettes_1997} here we will establish a connection between our results and the more general concept of {\em multiresolution} as introduced in~\cite{de_boor_construction_1993}. Before proceeding, let \(\Delta > 0\). We say that a subspace \(S\subset \LERd\) is {\em \(\Delta\)-shift-invariant} if
\[
    x_0\in\Delta \Zd,\, u\in S:\,\, \tau_{x_0} u \in S.
\]
\begin{definition}[Multiresolution]
    Let \(\set{S_j}_{j\in\Z}\) be a collection of closed subspaces of \(\LERd\) such that for every \(j\in \Z\) the subspace \(S_j\) is \(2^{-j}\)-shift-invariant. We say that it is a {\em multiresolution} if
    \begin{enumerate}[leftmargin=42pt, label=(MR.\arabic*), parsep=\parskip]
        \item{\(\forall\, j\in\Z:\,\, S_{j} \subset S_{j+1}\).
        }
        \item{\(\cl(\cup_{j\in\Z} S_j) = \LERd\)
        }
        \item{\(\cap_{j\in\Z} S_j = \set{0}\).
        }
    \end{enumerate}
\end{definition}
The classical approach to construct a multiresolution is to assume that the space \(S_0\) is generated by a function \(\phi\in \LERd\), that is,
\[
    S_0 := \cl({\set{\tau_{k}\phi:\,\, k\in\Zd}})
\]
and that for \(j\neq 0\) the space \(S_j\) is generated by the translates of \(D_{2^{j}}\phi\), that is,
\begin{equation}\label{eq:mr:dilates}
    S_j := \cl{\set{\tau_{2^{-j} k}D_{2^{j}}\phi:\,\, k\in \Zd}}.
\end{equation}
Since it is apparent that
\[
    S_j := \cl({\set{D_{2^j}\tau_{k}\phi:\,\, k\in \Zd}})
\]
the sequence \(\set{S_j}_{j\in\Z}\) satisfies (MR.1) if
\begin{equation}\label{eq:mr:spacesInclusion}
    S_0 \subset S_1.
\end{equation}
From this construction we obtain a multiresolution under mild conditions on \(\phi\). The following proposition follows from Theorem~4.5 and Corollary 4.14 in~\cite{de_boor_construction_1993}.
\begin{proposition}
    Let \(\phi\in \LERd\), let \(\whphi\) denote its Fourier transform, and let \(\set{S_j}_{j\in\Z}\) be the sequence of subspaces of \(\LERd\) defined in~\eqref{eq:mr:dilates}. If \(\whphi\) is non-zero almost everywhere in a neighborhood of the origin then the sequence \(\set{S_j}_{j\in\Rd}\) satisfies conditions (MR.2) and (MR.3).
\end{proposition}
The inclusion~\eqref{eq:mr:spacesInclusion} is satisfied if there exists a 1-periodic function \(\whh\in \LInfRd\) such that
\begin{equation}\label{eq:mr:refinementEquation}
    D_2 \whphi = \whh \whphi
\end{equation}
where \(\whphi\) denotes the Fourier transform of \(\phi\). An equation of the form~\eqref{eq:mr:refinementEquation} is called a {\em refinement equation}~\cite{de_boor_construction_1993} and a function \(\phi\) that satisfies~\eqref{eq:mr:refinementEquation} is called a {\em refinable function}. The function \(\whh\) is called a {\em refinement mask}~\cite{de_boor_construction_1993} or {\em two-scale symbol}~\cite{christensen_introduction_2016}. It is assumed to be \(1\)-periodic. In~\cite{daubechies_orthonormal_1988} it is shown that for a given \(\whh\) the solution to~\eqref{eq:mr:refinementEquation} is formally given by the {\em cascade algorithm}~\cite{daubechies_ten_1992}
\[
    \whphi := \prods_{j=1}^{\infty} D_{2^{-j}} \whh
\]
where the product converges if \(\whh(0) = 1\) and \(|\whh| > 0\) near the origin. In this case \(\whphi(0) = 1\) and, under suitable regularity conditions for \(\whh\), \(|\whphi| > 0\) near the origin. Therefore, given a sufficiently regular refinement mask \(\whh\), we can construct a function \(\phi\) that induces a multiresolution. Our goal is to show when \(\Sigma\)-multipliers can be used for this purpose.

%% ============================================================================
\subsection{\(\Sigma\)-multipliers and multiresolution}
%% ============================================================================

We now present conditions under which a \(\Sigma\)-multiplier can be used to construct a multiresolution. From our previous discussion, it suffices to show that a \(\Sigma\)-multiplier can be used to construct a refinable function \(\whphi\in\LERd\) that is non-zero almost everywhere near the origin. From now on, let \(\Omo = [-1/2, 1/2]^d\) and let \(\alpha = 2\). To control the behavior of the multiplier on the boundary, we introduce the weight
\[
    w_N(\xi) = \prod_{k=1}^d \left(\frac{1 + e^{2\pi i \xi_k}}{2}\right)^N
\]
for \(N\in \No\)~\cite[cf. Proposition~3.3]{daubechies_orthonormal_1988}. Then, we let \(\whh_{\Sigma,N}\) be the \(1\)-periodic extension of \(w_N m_{\Sigma}\). Under very general conditions, this refinement mask yields a refinable function that induces a multiresolution.

\begin{theorem}\label{thm:sigmaMultiplierDefiniteInducesMRA}
    If \(\ZF = \emptyset\) and \(\vSig\in\Hnn\) is positive definite, then \(w_N m_{\Sigma}\) produces a refinable function that induces a multiresolution.
\end{theorem}

When the matrix \(\vSig\) is semidefinite the multiplier may have singularities even when \(m_{\Sig}\in \LEOmo\). Therefore, the above approach cannot be applied directly. However, we can extend our arguments in a straightforward manner for \(d = 1\).

\begin{theorem}\label{thm:sigmaMultiplierSemidefinite1DInducesMRA}
    Let \(d = 1\). Suppose that \(\ZF = \emptyset\) and that \(\vSig\in\Hnn\) is positive semidefinite. If \(m_{\Sig} \in \LEOmo\) then there exists \(p\in\No\) such that \(2^p w_N m_{\Sig}\) produces a refinable function that induces a multiresolution.
\end{theorem}

To our knowledge, when \(d > 1\) there is no straightforward extensions of these results. The main obstruction is that \(m_{\Sig} \in \LEOmo\) no longer implies that \(m_{\Sig}\) is bounded. In fact, the following result shows that the condition \(m_{\Sig} \in \LEOmo\) provides a very weak control of the order of the zeros of the numerator and denominator in \(m_{\Sig}\). Recall that if \(f : \Rd \to \C\) is real-analytic, then \(\xio\) is a zero of order \(p\) if \(p\) is the largest integer such that
\[
    |\gamma| < p:\,\, \partial^{\gamma} f(\xio) = 0.
\]
Note that in this case, there is at least one derivative of \(f\) of order \(p\) that does not vanish at \(\xio\).
\begin{lemma}\label{lem:sigmaMultiplierSingularityBehavior}
    Let \(m\) be a \(\Sigma\)-multiplier and suppose that \(m\in\LEOmor\). If \(\xio\in \Omo\) is a zero of \(\nrmS{\vf}^2\) of order \(q\) then it is a zero of \(\bprodS{\vf}{\Da\vf}\) of degree strictly greater than \(q - d/2\).
\end{lemma}
Therefore, the multiplier may still have singularities. Although finding additional conditions under which the multiplier is bounded is beyond the scope of our work, we provide some insight into the connection between this problem and that of the {\em \(q\)-capacity} of a set. The connection is as follows. Let \(Z\) be the set of zeros of \(\nrmS{\vf}^2\) on \(\Omo\). Then \(Z\) is a real-analytic subvariety and \(m_{\Sig}\) is real-analytic on \(\Omo\setminus Z\). If \(m_{\Sig}\) admits a real-analytic extension from \(\Omo\setminus Z\) to all of \(\Omo\) then \(Z\) is {\em removable} and \(m_{\Sig}\) is bounded on \(\Omo\). Of course, whether \(Z\) is removable or not depends on some delicate properties of this set.

Let \(K\subset \Rd\) be a compact set. Its \(q\)-capacity is
\[
    C_q(K) := \inf\Lset{\int \nrmlEd{\nabla \phi(\xi)}^q\, d\xi:\,\, \phi\in C^{\infty}_0(\Rd),\,\, \phi|_K \geq 1}.
\]
This notion can be extended to a general subset \(S\subset \Rd\) as
\[
    C_q(S) := \sup\set{C_q(K):\,\, K\subset S,\,\, \mbox{\(K\) compact}}.
\]
The \(q\)-capacity determines when a given set \(Z\) is removable for a class of harmonic functions, which are themselves real-analytic~\cite[Theorem~10, Section~2.2]{evans_partial_2010}, belonging to a suitable Sobolev space.

\begin{theorem}[Theorem~2 in~\cite{hedberg_removable_1974}]
    Let \(p \in (1, \infty)\) and let \(\mathrm{HD}^p(\Omo)\) be the space of harmonic functions in the Sobolev space \(W^{1,p}(\Omo)\). Then \(Z \subseteq \Omo\) is a removable set for any function in \(\mathrm{HD}^p(\Omo)\) if and only if \(C_q(Z)=0\) with \(1/p + 1/q = 1\).
\end{theorem}

%% ============================================================================
\subsection{Refinement masks and \(\Sigma\)-multipliers}
%% ============================================================================

In Section~\ref{sec:adaptiveMultipliersTranslates} we explicitly analyzed \(\Sig\)-multipliers induced by a family of translates. By analyzing the multipliers induced by a family of scalings we can establish a relation to refinable functions. Let \(f = \Dia f_0\). Since \(\alp > 1\) a part of the high-frequency content of \(f_0\) has been removed in \(f\) over \(\Omo\). Following the notation in Section~\ref{sec:adaptiveMultipliersTranslates} the optimal multiplier for \(f\) is
\[
    \frac{\Dia f^* \Da \Dia f}{|\Dia f|^2} = \Dia \mf
\]
where \(\mf\) is as defined in~\eqref{eq:perfectExtrapolationSingleFunction}. This suggests a form of scaling-invariance. Consider the family
\[
    \mF = \set{f,f_1,\ldots, f_{n-1}}\quad\mbox{with}\quad f_{k}(\xi) = f(\alp^{-k}\xi)\quad\mbox{for}\quad k\in\mX_{n-1}.
\]
The multiplier in~\eqref{eq:perfectExtrapolationSingleFunction} associated to \(f_{k+1}\) is
\[
    \frac{f_{k+1}^* \Da f_{k+1}}{|f_{k+1}|^2} = \frac{\Dia f_k^* \Dia \Da f_k}{|\Dia f_k|^2} = \Dia \frac{f_k^*\Da f_k}{|f_k|^2} = \ldots = D_{\alp^{-k}} \mf
\]
and the product of these multipliers resemble iterations of the cascade algorithm as
\[
    \prod_{k=0}^{n-1} \frac{f_k^* \Da f_k}{|f_k|^2} = \prods_{k=0}^{n-1} D_{\alp^{-k}} \mf.
\]
This suggests that when \(f_0\) is a refinable function some information about the Fourier series of the refinement mask is encoded in a \(\Sig\)-multiplier induced by a family of scalings as above. Let \(\alp = 2\) and suppose that \(f = \whphi\). In this case, the trace multiplier becomes
\[
    m_I = \frac{\sums_{k=0}^{n-1} D_{\alp^{-k}} \whphi^* D_{\alp^{-k +1}} \whphi}{\sums_{k=0}^{n-1}|D_{\alp^{-1}}\whphi|^2} = \frac{\sums_{k=0}^{n-1} |D_{\alp^{-k}} \whphi|^2 D_{\alp^{-k}} \whh}{\sums_{k=0}^{n-1}|D_{\alp^{-1}}\whphi|^2} = \sums_{k=0}^{n-1}\left(\frac{|D_{\alp^{-k}} \whphi|^2}{\sums_{j=0}^{n-1} |D_{\alp^{-k}} \whphi|^2}\right) D_{\alp^{-k}} \whh.
\]
Instead of the {\em product} of the terms \(D_{\alp^{-k}} \whh\) this yields a {\em pointwise convex combination} of the terms \(D_{\alp^{-k}} \whh\). In fact, the term \(D_{\alp^{-k}} \whh\) will have the largest coefficients in this combination if the term \(|D_{\alp^{-k}}\whphi|^2\) is the largest. Although this does not equal \(\whh\), it does encode information about its scalings.

Instead of considering a family of scalings, we can use the family
\[
    \mF = \set{f_k:\,\, k\in\Zd:\, |k_i| \leq N}\quad\mbox{for}\quad f_k(\xi) = e^{-2\pi i \xi\cdot k} \whphi(\xi)
\]
for some \(N\in \N\). This is a subset of the basis for \(S_0\) in a multiresolution. It is illustrative to see that the trace multiplier in this case is
\[
    m_I(\xi) = \frac{\sums_{|k|\leq N} e^{-2\pi (\alp - 1) i \xi  \cdot k} |\whphi(\xi)|^2 \whh(\xi)}{\sums_{|k|\leq N} |\whphi(\xi)|^2} = \whh(\xi) \left(\frac{1}{N^d}\sums_{|k|\leq N} e^{-2\pi (\alp -1) i \xi\cdot k}\right) = \frac{1}{N^d} \whh(\xi) D_N(\xi)
\]
where we used the fact that \(\alp = 2\) and where \(D_N\) is the Dirichlet kernel. It is clear that as \(N\to \infty\) the above converges, in a suitable sense, to a Dirac's \(\delta\) at the origin. In general, a \(\Sig\)-multiplier in this case will have the form
\[
    m_{\Sig}(\xi) = \whh(\xi) \frac{\iprodS{\ve}{\Da\ve}}{\iprodS{\ve}{\ve}}
\]
where \(\ve:\Omo\to \Cn\) with \(n=N^d\) is the function with components \(e_k(\xi) = e^{-2\pi i \xi\cdot k}\). It is apparent that there is no matrix \(\vSig\) for which the quotient is identically one, unless it is a matrix with \(\Sig_{0,0} = 1\) as its only non-zero component. In spite of this, we can recover \(|\whh|^2\) from a non-trivial \(\Sig\)-multiplier when \(h\) is compactly supported and non-negative. Let \(\Gamma = \supp(h)\) and, since \(\Gamma\subset \Zd\), let \(\mF\) be the family
\[
    \mF = \set{f_k:\,\, k\in\Gamma }\quad\mbox{for}\quad f_k(\xi) = e^{-2\pi i \xi\cdot k} \whphi(\xi).
\]
Then, let \(\vSig\) be a diagonal matrix with \(\Sig_{k,k} = h(k)\). Since \(k\in\Gamma\) and \(h\) is non-negative, we have that \(\Sig_{k,k} >0\). Then the associated multiplier is
\[
    m_{\Sig} = \frac{\sums_{k\in\Gamma} h(k) e^{-2\pi i \xi\cdot k} |\whphi(\xi)|^2 \whh(\xi)}{\sums_{k\in\Gamma} h(k) |\whphi(\xi)|^2} = \frac{1}{\whh(0)}|\whh(\xi)|^2 = |\whh(\xi)|^2
\]
where we used the fact that \(\whh(0) = 1\). This can be used as a test whether \(f_0\) corresponds to a scaling function. Perhaps more surprisingly, we can recover additional information by repeating this process. If we use
\[
    \mF = \set{f_k:\,\, k\in\Gamma }\quad\mbox{for}\quad f_k(\xi) = e^{2\pi i \xi\cdot k} \whphi(\xi)
\]
then the above multiplier simply becomes
\[
    m_{\Sig} = \frac{\sums_{k\in\Gamma} h(k) e^{2\pi i \xi\cdot k} |\whphi(\xi)|^2 \whh(\xi)}{\sums_{k\in\Gamma} h(k) |\whphi(\xi)|^2} = \frac{1}{\whh(0)}|\whh(\xi)|^2 = \whh(\xi)^2
\]
from where \(\whh\) can be recovered in this case. Finally, by considering the collection
\[
    \mF = \set{f_k:\,\, k\in\Gamma }\quad\mbox{for}\quad f_k(\xi) = e^{4\pi i \xi\cdot k} \whphi(\xi).
\]
instead. Note the translates are scaled by a factor of \(-2\) compared to the previous case. The multiplier then becomes
\[
    m_{\Sig} = \frac{\sums_{k\in\Gamma} h(k) e^{4\pi i \xi\cdot k} |\whphi(\xi)|^2 \whh(\xi)}{\sums_{k\in\Gamma} h(k) |\whphi(\xi)|^2} = \frac{1}{\whh(0)}|\whh(\xi)|^2 = \whh(\xi) \whh(2\xi).
\]
Interestingly, in this case the multiplier achieves extrapolation by two scales, that is,
\[
    v(\xi) m_{\Sig}(\xi) = \sums_{k\in\Gamma} c_k f_k(\xi) m_{\Sig}(\xi) = \sums_{k\in\Gamma} c_k e^{4\pi i \xi\cdot k} \whphi(\xi) \whh(\xi)\whh(2\xi) = \sums_{k\in\Gamma} c_k e^{4\pi i \xi\cdot k} \whphi(4\xi) = D_{\alp^2} v(\xi).
\]
In this case we have assumed knowledge of \(h\). However, whether this choice is optimal for some error function is an open question. Even in this case, this suggests that information about the refinement mask may be recovered from \(\Sig\)-multipliers induced by families of translates.

%% ============================================================================
\section{Numerical experiments}
\label{sec:numericalExperiments}
%% ============================================================================

We now perform numerical experiments to illustrate the practical implications of our theoretical results. The numerical results presented in this section can be reproduced using the code in the GitHub repository:\\ \href{https://github.com/csl-lab/adaptiveExtrapolationInFrequency}{{\tt csl-lab/adaptiveExtrapolationInFrequency}}.

%% ============================================================================
\subsection{Multiresolution}
%% ============================================================================

\begin{figure*}[!htb]
    \centering
    \begin{subfigure}[t]{0.24\textwidth}
        \includegraphics[width=0.99\textwidth]{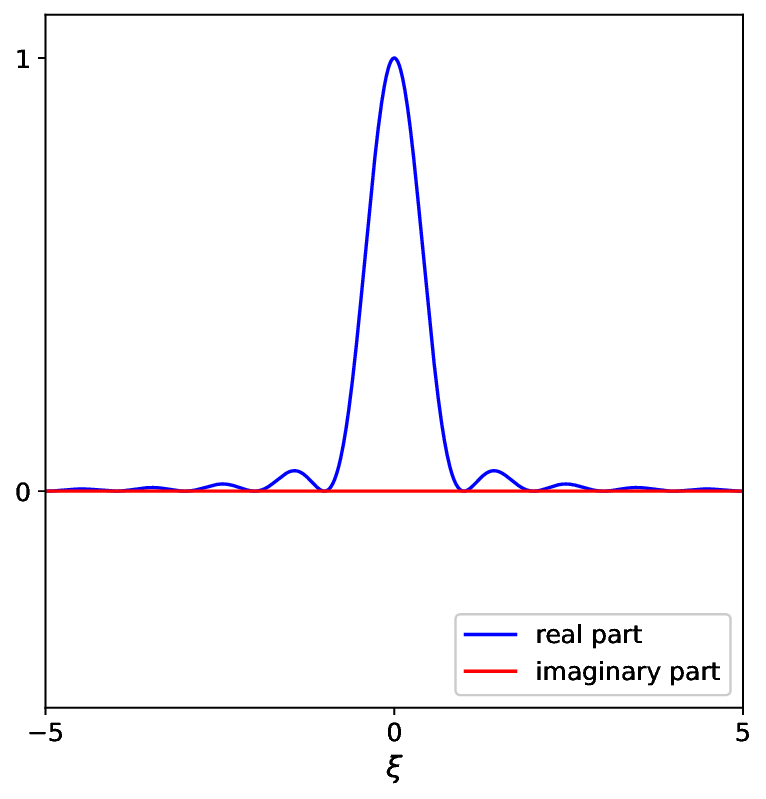}
        \caption{\(f\)}
        \label{fig:mr1d:test:f}
    \end{subfigure}%
    \begin{subfigure}[t]{0.24\textwidth}
        \includegraphics[width=0.99\textwidth]{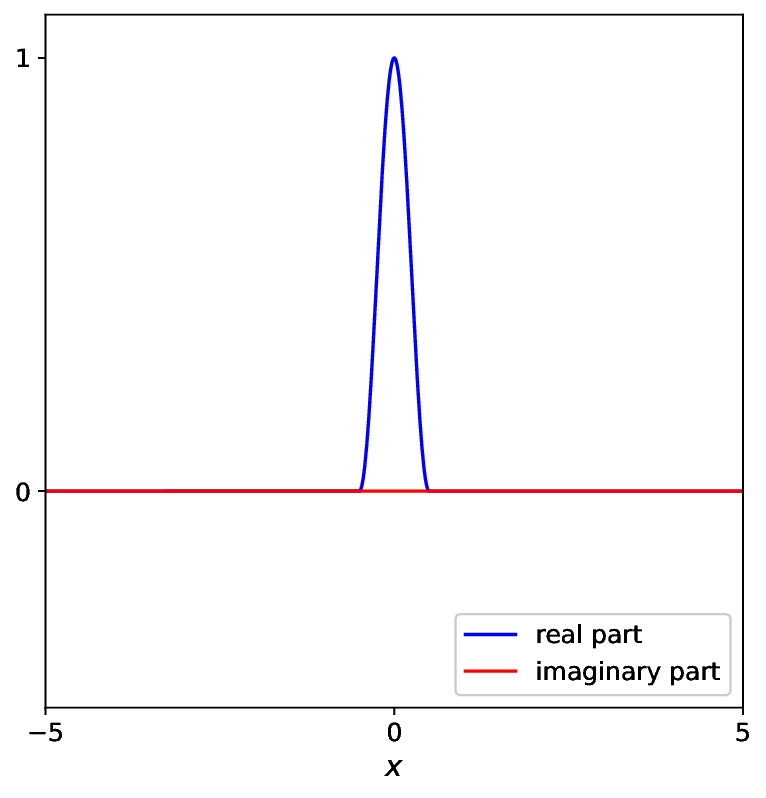}
        \caption{\(\mf\)}
        \label{fig:mr1d:test:m}
    \end{subfigure}%
    \begin{subfigure}[t]{0.24\textwidth}
        \includegraphics[width=0.99\textwidth]{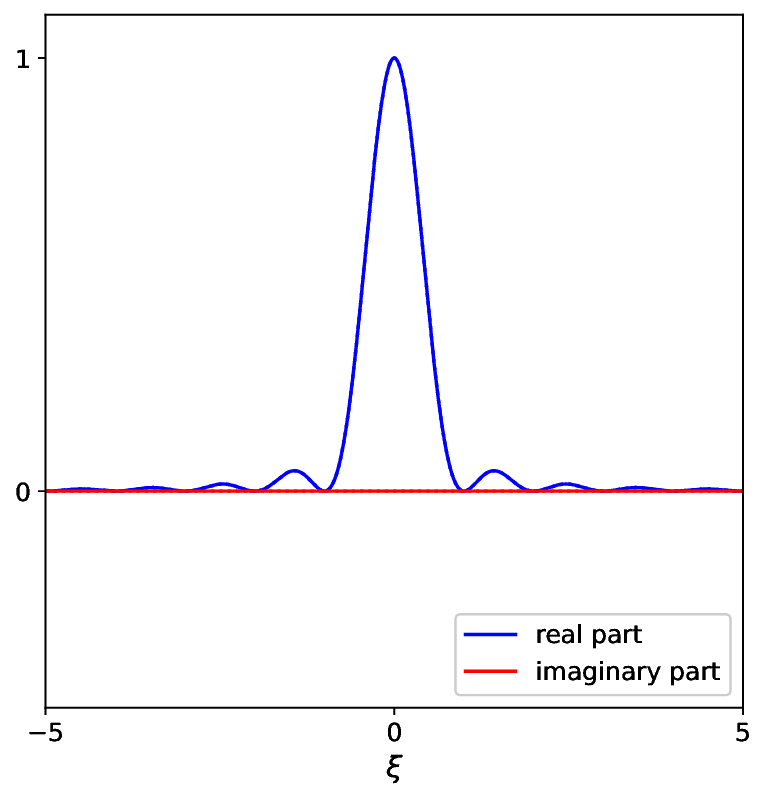}
        \caption{\(\whphi\)}
        \label{fig:mr1d:test:whphi}
    \end{subfigure}%
    \begin{subfigure}[t]{0.24\textwidth}
        \includegraphics[width=0.99\textwidth]{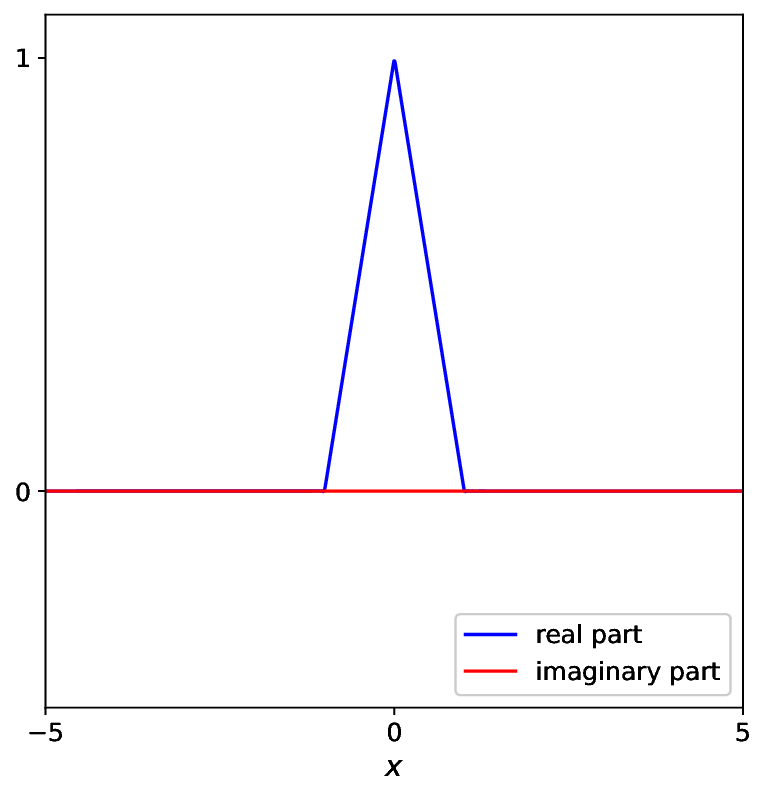}
        \caption{\(\phi\)}
        \label{fig:mr1d:test:phi}
    \end{subfigure}\\
    \begin{subfigure}[t]{0.24\textwidth}
        \includegraphics[width=0.99\textwidth]{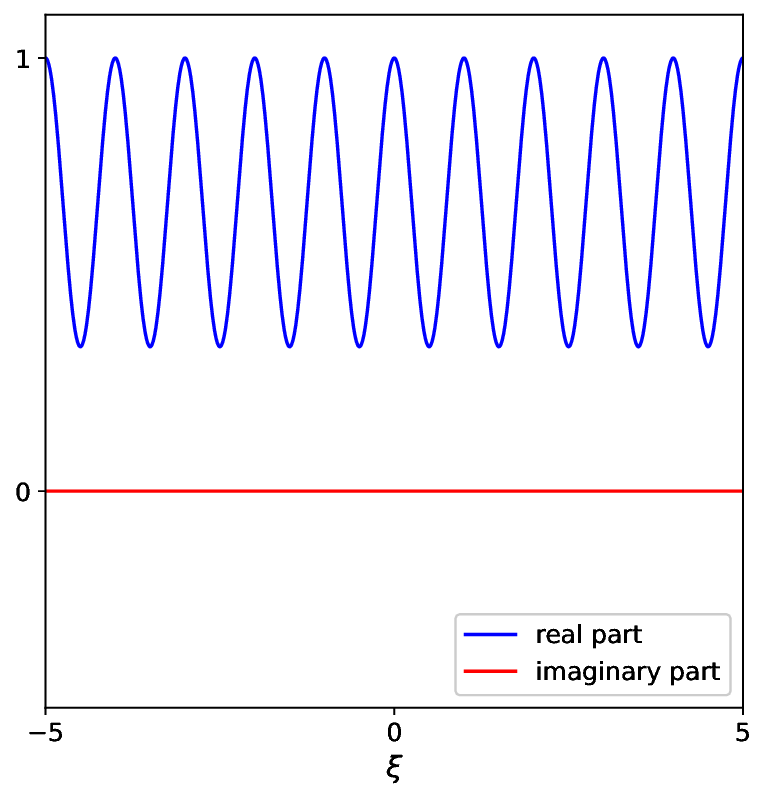}
        \caption{\(\Phi\)}
        \label{fig:mr1d:test:Phi}
    \end{subfigure}%
    \begin{subfigure}[t]{0.24\textwidth}
        \includegraphics[width=0.99\textwidth]{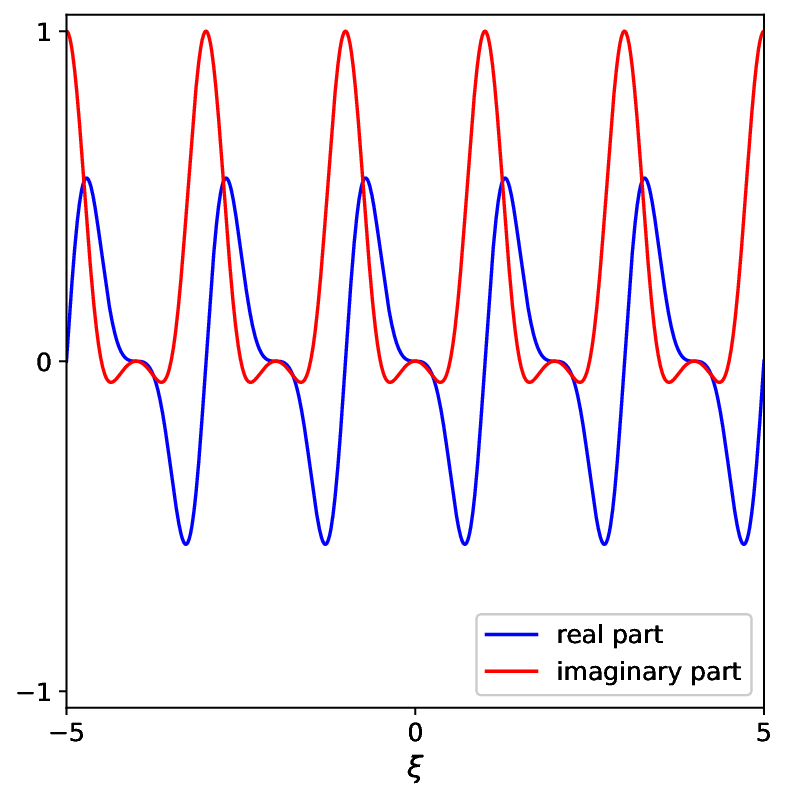}
        \caption{\(g\)}
        \label{fig:mr1d:test:g}
    \end{subfigure}%
    \begin{subfigure}[t]{0.24\textwidth}
        \includegraphics[width=0.99\textwidth]{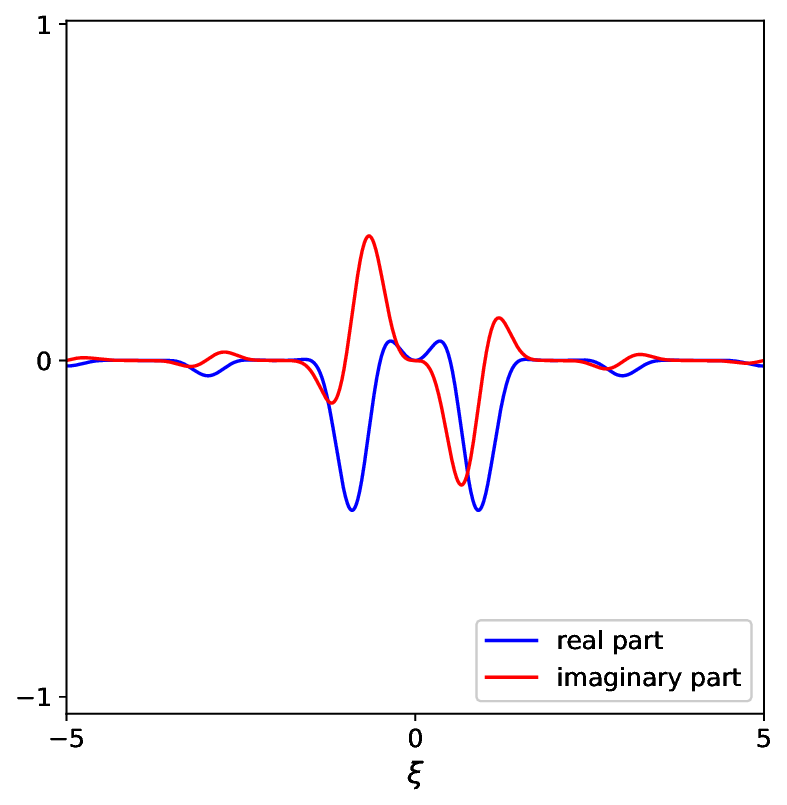}
        \caption{\(\whpsi\)}
        \label{fig:mr1d:test:whpsi}
    \end{subfigure}%
    \begin{subfigure}[t]{0.24\textwidth}
        \includegraphics[width=0.99\textwidth]{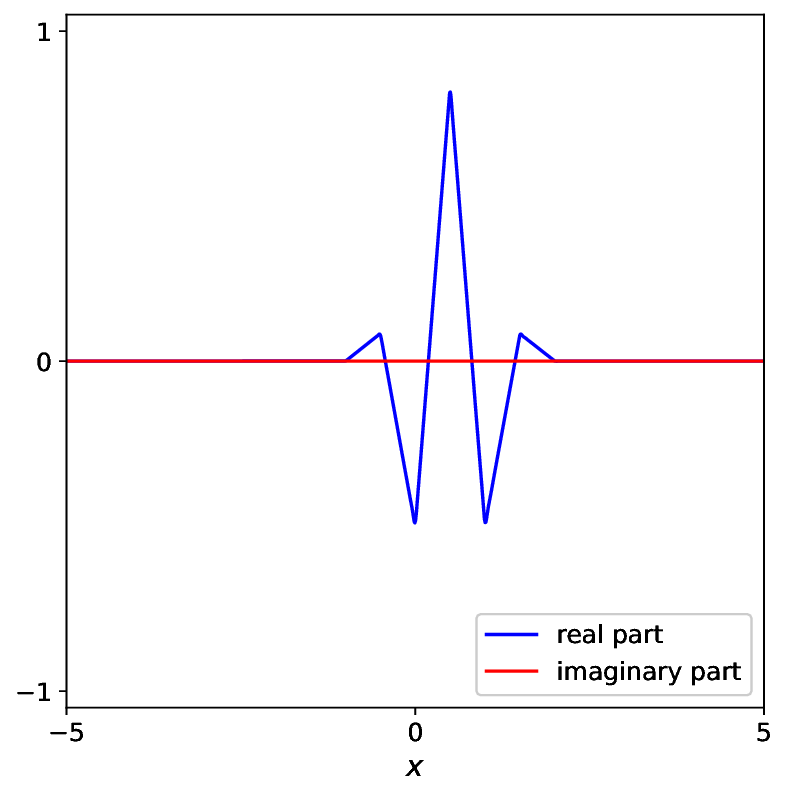}
        \caption{\(\psi\)}
        \label{fig:mr1d:test:psi}
    \end{subfigure}%

    \caption{The multiresolution induced by the collection \(\mF\) with the single element \(f = \sinc^2\) recovers the multiresolution induced by a \(B\)-spline of degree 1.
    }
    \label{fig:mr1d:test}
\end{figure*}

In Section~\ref{sec:connectionsToMultiresolution} we discussed the connections between \(\Sigma\)-multipliers and multiresolutions. In this section we explore this connection further by performing numerical experiments in \(d=1\). In all these experiments we assume that \(\Omo = [-1/2, 1/2]\).

Our first experiment shows how from a collection \(\mF\) consisting of a single element \(f\) we can recover a multiresolution. Since in this case \(n = 1\) there is only one \(\Sigma\)-multiplier that can be defined. Remark that it coincides with the multiplier in~\eqref{eq:perfectExtrapolationSingleFunction}. As an example, for \(f = \sinc^2\) (Fig.~\ref{fig:mr1d:example:f}) all the \(\Sigma\)-multipliers coincide with \(\mf\) (Fig.~\ref{fig:mr1d:example:m}). We denote as \(m\) the \(1\)-periodic extension of \(\mf\) to all of \(\R\). By approximating numerically the solution to the refinement equation using the cascade algorithm we obtain \(\whphi\) and \(\phi\) (Fig.~\ref{fig:mr1d:example:whphi} and Fig.~\ref{fig:mr1d:example:phi} show an approximation with 128 products). Since \(f = \whphi\) we see that \(f\) is a function that can be exactly extrapolated from its low-frequency content. To construct a wavelet, we follow the same strategy as in~\cite[Theorem~5.5]{de_boor_construction_1993}. First, let
\[
    \Phi = \sums_{k\in\Z} |\tau_{k}\whphi|^2
\]
be the \(1\)-periodization of \(|\whphi|^2\) (Fig.~\ref{fig:mr1d:example:Phi} shows an approximation with 257 terms). Then, we can define the wavelet mask (Fig.~\ref{fig:mr1d:test:g})
\[
    g(\xi) = m\left(\xi + \frac{1}{2}\right)^* \Phi\left(\xi + \frac{1}{2}\right)  e^{-2\pi i \xi}.
\]
Then
\[
    \whpsi = D_{1/2} g D_{1/2} \whphi
\]
is the Fourier transform \(\whpsi\) of the wavelet \(\psi\) (Fig.~\ref{fig:mr1d:example:whpsi} and Fig.~\ref{fig:mr1d:test:psi}). As expected, we obtain the \(B\)-spline of degree 1.

\begin{figure*}[!htb]
    \centering
    \begin{subfigure}[t]{0.24\textwidth}
        \includegraphics[width=0.99\textwidth]{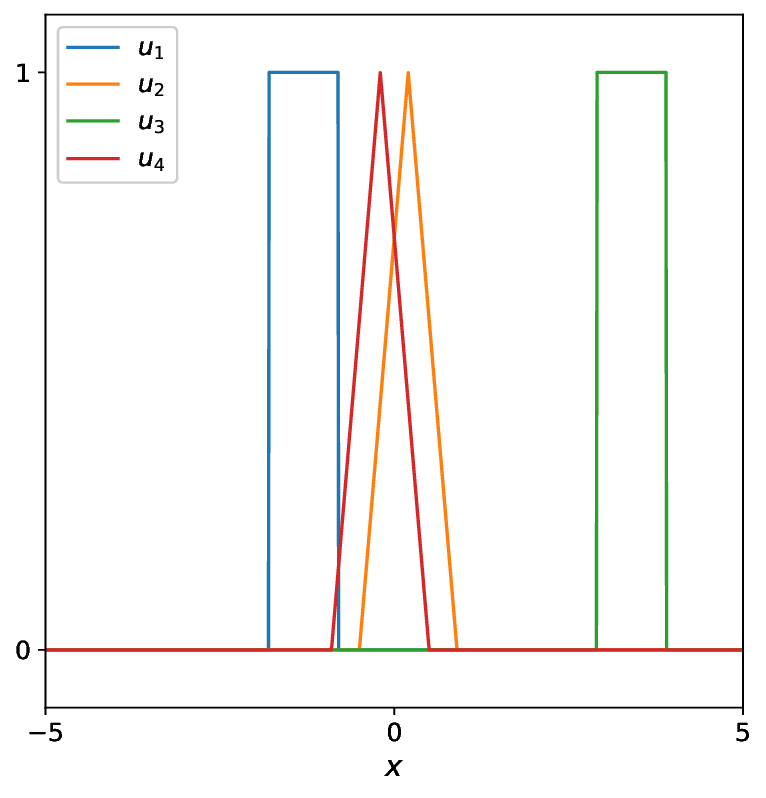}
        \caption{\(\mU\)}
        \label{fig:mr1d:example:f}
    \end{subfigure}%
    \begin{subfigure}[t]{0.24\textwidth}
        \includegraphics[width=0.99\textwidth]{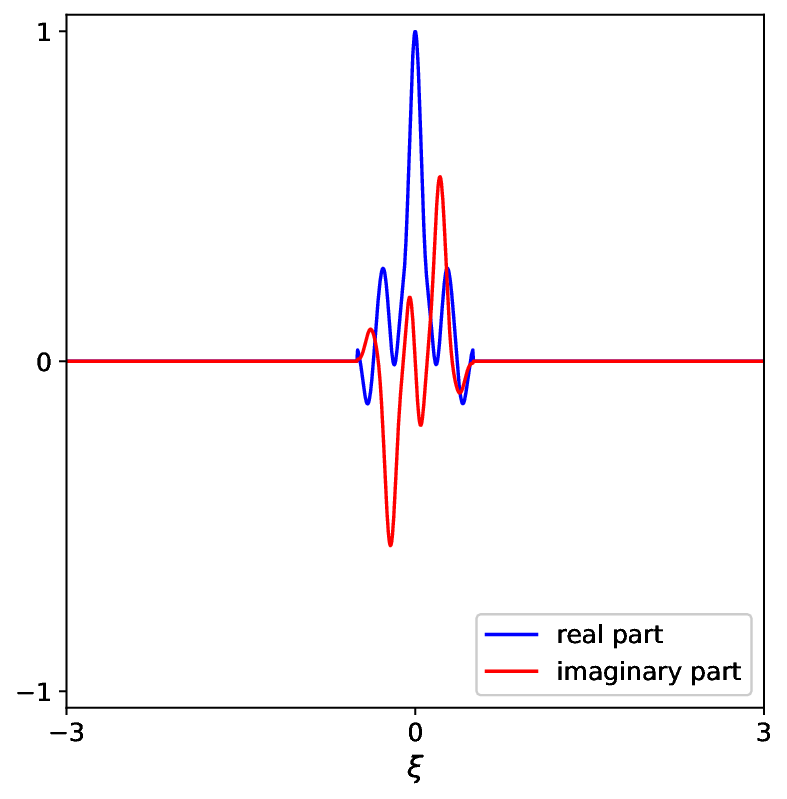}
        \caption{\(m\opt\)}
        \label{fig:mr1d:example:m}
    \end{subfigure}%
    \begin{subfigure}[t]{0.24\textwidth}
        \includegraphics[width=0.99\textwidth]{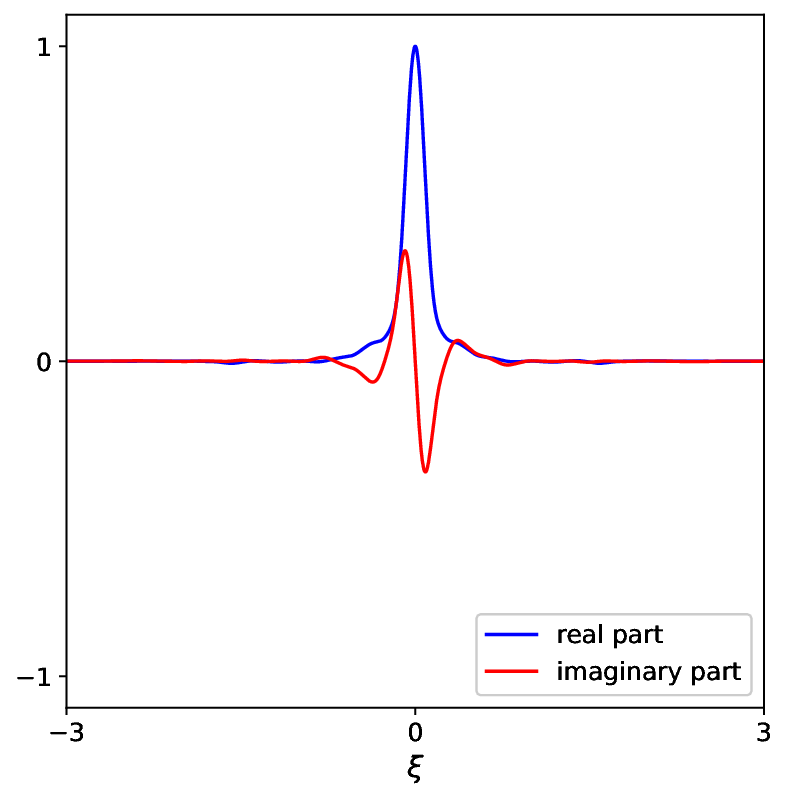}
        \caption{\(\whphi\)}
        \label{fig:mr1d:example:whphi}
    \end{subfigure}%
    \begin{subfigure}[t]{0.24\textwidth}
        \includegraphics[width=0.99\textwidth]{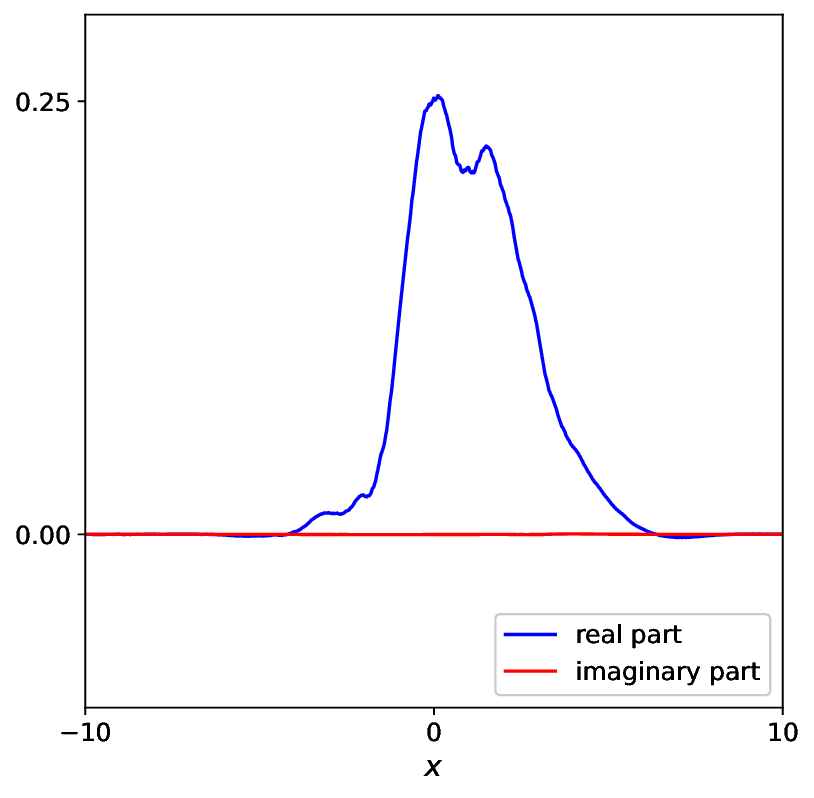}
        \caption{\(\phi\)}
        \label{fig:mr1d:example:phi}
    \end{subfigure}\\
    \begin{subfigure}[t]{0.24\textwidth}
        \includegraphics[width=0.99\textwidth]{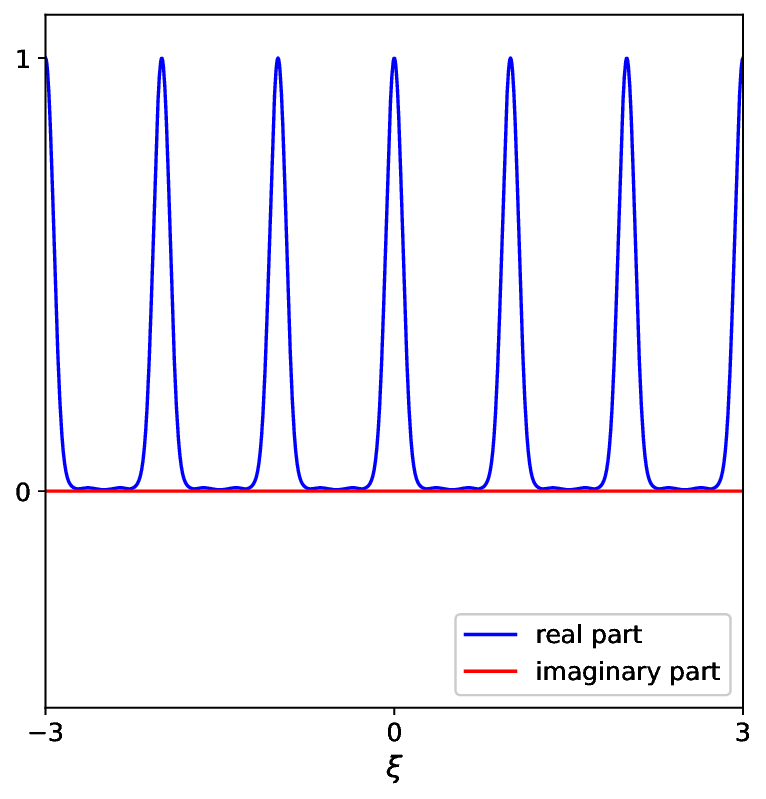}
        \caption{\(\Phi\)}
        \label{fig:mr1d:example:Phi}
    \end{subfigure}%
    \begin{subfigure}[t]{0.24\textwidth}
        \includegraphics[width=0.99\textwidth]{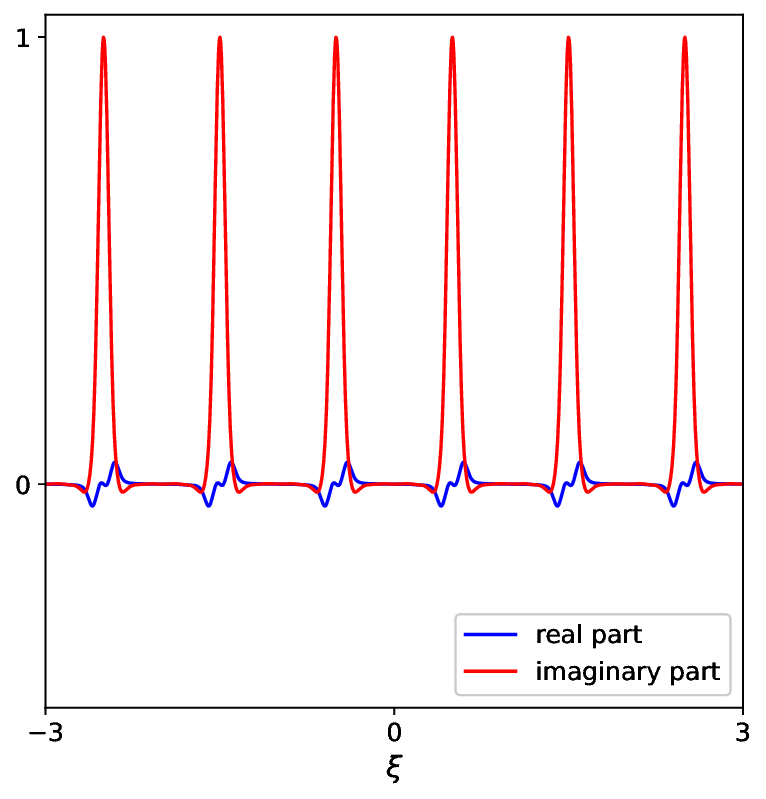}
        \caption{\(g\)}
        \label{fig:mr1d:example:g}
    \end{subfigure}%
    \begin{subfigure}[t]{0.24\textwidth}
        \includegraphics[width=0.99\textwidth]{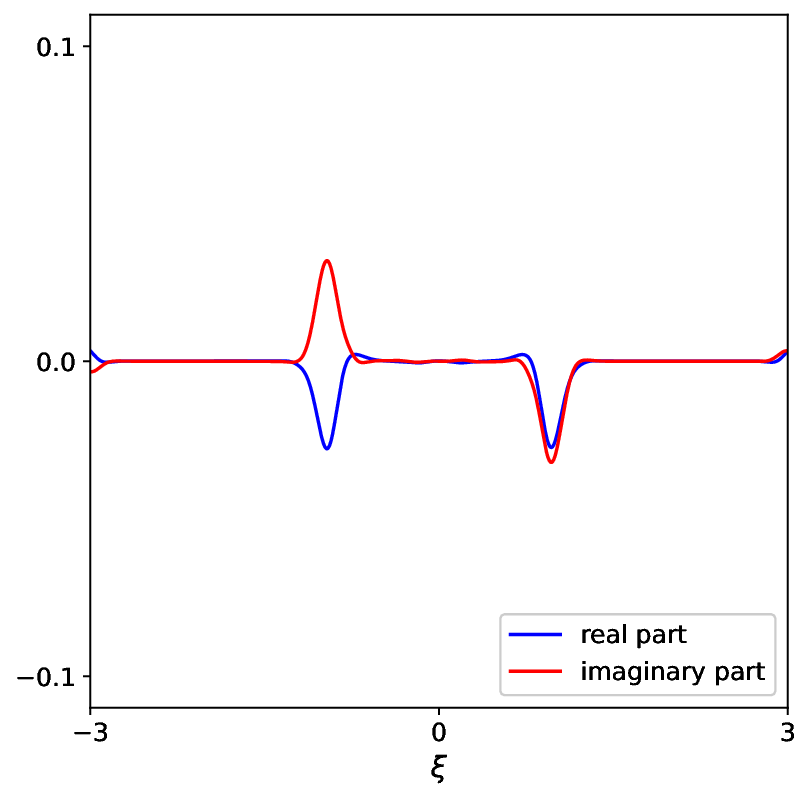}
        \caption{\(\whpsi\)}
        \label{fig:mr1d:example:whpsi}
    \end{subfigure}%
    \begin{subfigure}[t]{0.24\textwidth}
        \includegraphics[width=0.99\textwidth]{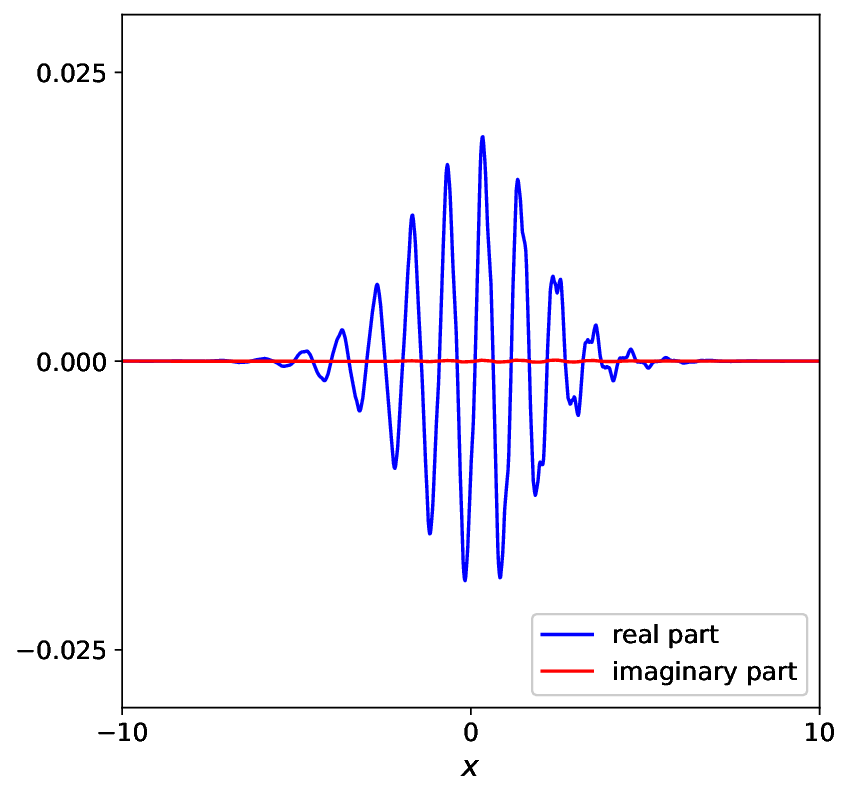}
        \caption{\(\psi\)}
        \label{fig:mr1d:example:psi}
    \end{subfigure}%

    \caption{Multiresolution induced by the collection \(\mU\).
    }
    \label{fig:mr1d:example}
\end{figure*}

Our second experiment shows how from a collection \(\mU\) with 4 elements (Fig.~\ref{fig:mr1d:example:f}) we can also induce a multiresolution. In this case, we compute the optimal multiplier using the regularized fixed-point iteration~\eqref{eq:fixedPointIterationRegularized} with \(W\) corresponding to the unit ball of the nuclear norm, that is, \(\sigma_W(\cdot) = \nrmop{\cdot}\). The parameters that we used were \(\delta = 10^{-6}\), \(\tG = 0.25\), \(\tS = 0.50\) and the algorithm performed a total of 1000 iterations. To compute the integrals, we use Monte Carlo sampling with an increasing number of samples at each iteration, from a minimum of 1000 samples to a maximum of 100000 samples. In this example, we assume that the number of samples is large enough for the quadrature error to be negligible. From the optimal multiplier \(m\opt\) (Fig.~\ref{fig:mr1d:example:m}) we can find an approximation to \(\whphi\) and of \(\phi\) as before (Fig.~\ref{fig:mr1d:example:whphi} and Fig.~\ref{fig:mr1d:example:phi}). Similarly, from the \(1\)-periodization \(\Phi\) (Fig.~\ref{fig:mr1d:example:Phi}) and the wavelet mask \(g\) (Fig.~\ref{fig:mr1d:example:g}) we can compute the Fourier transform \(\whpsi\) of the wavelet \(\psi\) (Fig.~\ref{fig:mr1d:example:whpsi} and Fig.~\ref{fig:mr1d:example:psi}). In this case, the wavelet shares similarities with the Morlet wavelet with center frequency 2.5 and bandwidth 1.5~\cite[cf. Chapter~4]{mallat_wavelet_2009}.

%% ============================================================================
\subsection{Optimal multipliers for the MNIST dataset}
%% ============================================================================

\begin{figure*}[!t]
    \centering
    \begin{subfigure}[t]{0.3\textwidth}
        \includegraphics[width=0.99\textwidth]{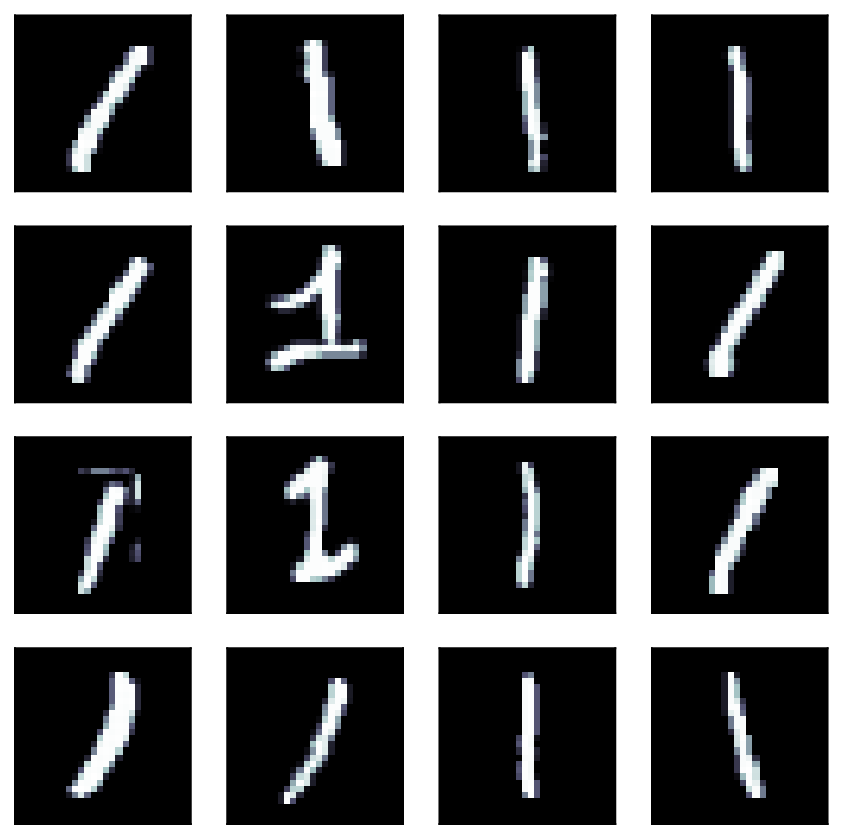}
        \caption{Examples for digit 1}
        \label{fig:mnist:data:d1}
    \end{subfigure}%
    \hspace{6pt}
    \begin{subfigure}[t]{0.3\textwidth}
        \includegraphics[width=0.99\textwidth]{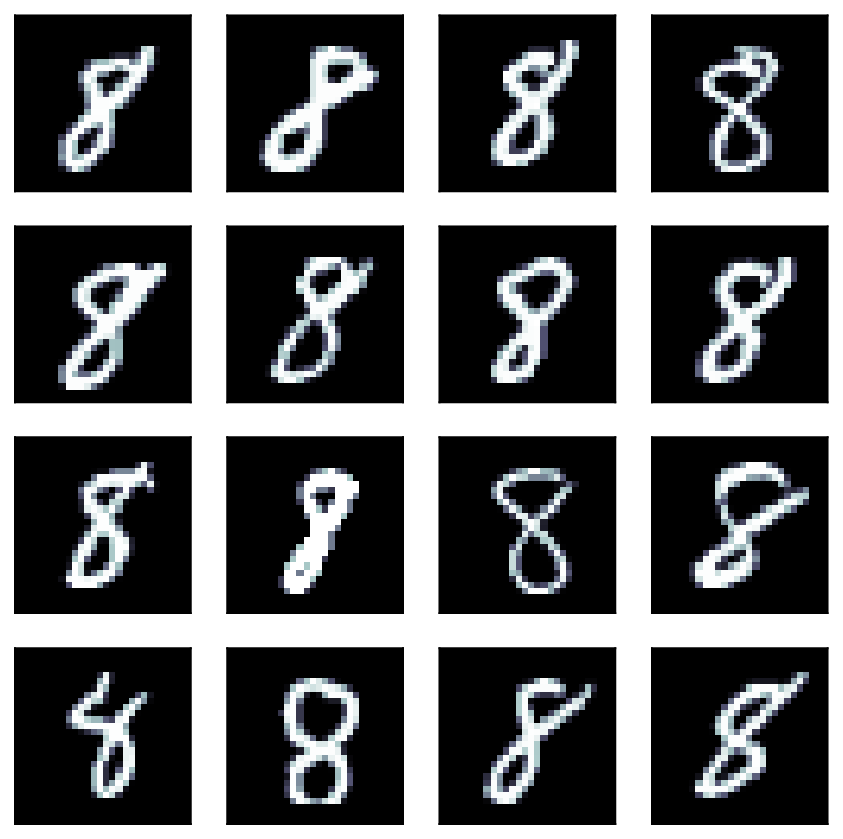}
        \caption{Examples for digit 8}
        \label{fig:mnist:data:d8}
    \end{subfigure}%
    \hspace{6pt}
    \begin{subfigure}[t]{0.3\textwidth}
        \includegraphics[width=0.99\textwidth]{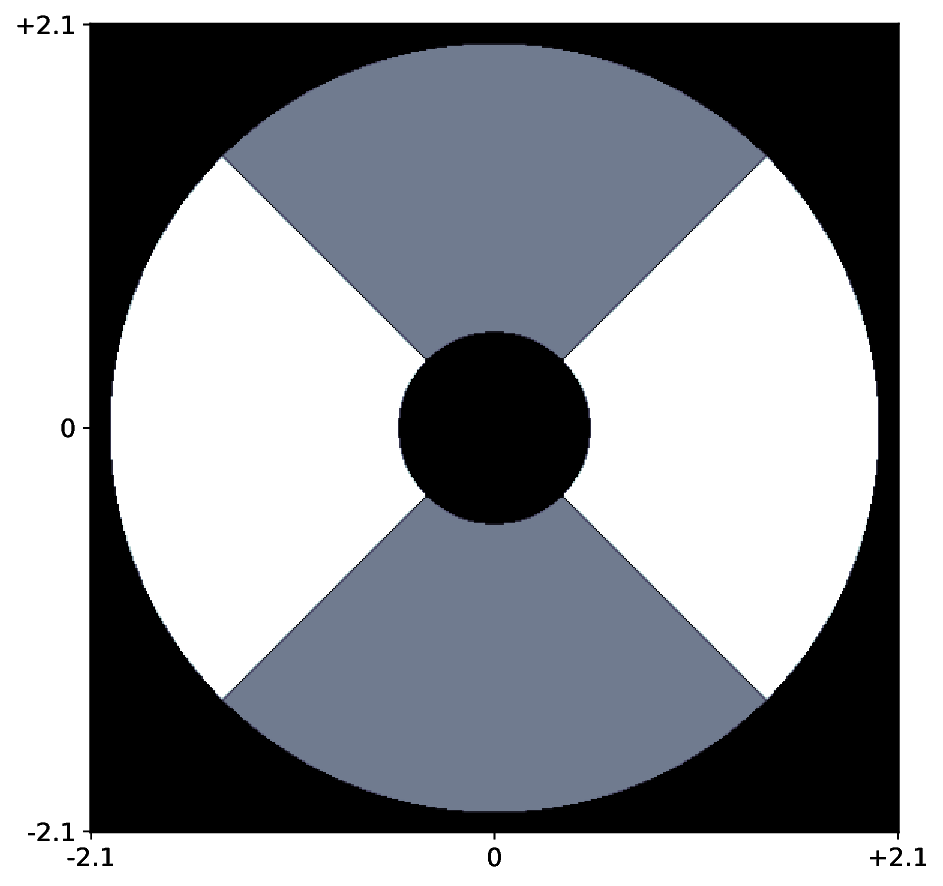}
        \caption{Annular region \(\Omo\)}
        \label{fig:mnist:data:sectors}
    \end{subfigure}%

    \caption{(a, b) Examples for the discrete data used for the digits 1 and 8. (c) Annular region \(\Omo\) and its division into regions.
    }
    \label{fig:mnist:data}
\end{figure*}

In Section~\ref{sec:adaptiveMultipliersDiscreteApproximation} we analyzed the \(\Sigma\)-multiplier associated to a collection \(\mF\) induced by interpolating discrete data. In this case, we consider the MNIST dataset~\cite{lecun_gradient-based_1998,li_deng_mnist_2012} as the source of this discrete data. The MNIST data set consists of images of size 28\(\times\)28 pixels of handwritten digits. To illustrate how an optimal multiplier adapts to this data, we compute it for a subset of \(n = 500\) images of the digits 1 (see Fig.~\ref{fig:mnist:data:d1}), and for a subset of 500 images of the digit 8 (see Fig.~\ref{fig:mnist:data:d8}). In each case the intensities are normalized to \([0, 1]\). To evaluate the elements of \(\mF\) in each case, we let \(N = 28\) and we interpolate each image to the unit square \([0, 1]^2\) as in~\eqref{eq:discreteData:interpolationInSpace} using the indicator function of the unit square as \(\vphi\) and \(\Delta x = 1/N \approx 0.035\). The elements of \(\mF\) are computed explicitly as in~\eqref{eq:discreteData:interpolationInFrequency}.

In this experiment, we use the restriction of the elements in \(\mF\) to the low-frequency domain \(\Omlo := \set{\xi\in \R^2:\,\, \enrm{\xi} \leq 2}\) to find an optimal multiplier that extrapolate their frequency information to the high-frequency domain \(\Omhi := \set{\xi\in\R^2:\,\, \enrm{\xi} \leq 8}\). To do so, we let \(\alpha = 4\) and we define the annular region \(\Omo := \set{\xi\in \R^2:\,\, 1/2 \leq \enrm{\xi} \leq 2}\) shown in Fig.~\ref{fig:mnist:data:sectors}. The optimal multiplier is computed using the regularized fixed-point iteration~\eqref{eq:fixedPointIterationRegularized} with \(W\) corresponding to the unit ball of the nuclear norm, that is, \(\sigma_W(\cdot) = \nrmop{\cdot}\). The parameters that we used were \(\delta = 10^{-6}\), \(\tG = 0.15\), \(\tS = 0.75\) and the algorithm performed a total of 100 iterations. To compute the integrals, we use Monte Carlo sampling with an increasing number of samples at each iteration, from a minimum of 1000 samples to a maximum of 50000 samples. Although in this experiment this does introduce an approximation error that must be taken into account in~\eqref{eq:fixedPointIterationRegularized} the iteration still converges under suitable conditions, namely, an increasing accuracy of the approximations to the integrals. We point the reader to~\cite{bravo_stochastic_2024} for a precise statement of this fact.

Fig.~\ref{fig:mnist:fpePerf} shows some performance metrics for the fixed-point iteration. Remark that the step-size decreases on average, suggesting the convergence of the algorithm on average (Fig.~\ref{fig:mnist:fpePerf:dxnrm}). Although Fig.~\ref{fig:mnist:fpePerf:obj} shows that the objective function {\em increases} we should contrast this fact with the increase in Monte Carlo samples shown in Fig.~\ref{fig:mnist:fpePerf:nnodes}. This suggests that the increase in the objective function is due to the increasing accuracy of the integrals required to perform an iteration.

In Fig.~\ref{fig:mnist:mult} we see the optimal multipliers. In Fig.~\ref{fig:mnist:mult:d1freq} and Fig.~\ref{fig:mnist:mult:d8freq} we see the optimal multipliers in the frequency domain. Both multipliers show a strong diagonal oscillatory component. This is due to the fact that we used \([0, 1]^2\) to interpolate the images instead of the centered domain \([-1/2, 1/2]^2\). In Fig.~\ref{fig:mnist:mult:d1space} and Fig.~\ref{fig:mnist:mult:d8space} we see the optimal multipliers in the spatial domain. We call these the {\em optimal filters}. The optimal filters capture relevant spatial structures that are common to the digit 1 and the digit 8 as we would expect. In particular, the optimal filter for the digit 1 captures both the structure that is common to the digit 1 at the center of the image, and the horizontal strokes that appear in some samples. The optimal filter for the digit 8 captures quite well the two holes and the center region in which the two strokes intersect.

In Fig.~\ref{fig:mnist:extrap:d1} and Fig.~\ref{fig:mnist:extrap:d8} we see the effect of using the optimal multiplier to extrapolate frequency information. In Fig.~\ref{fig:mnist:extrap:d1} we see the extrapolation in the frequency domain for one element of the collection \(\mF\). We see that the frequency extrapolation successfully increases the level of detail in the image. However, the increase is more noticeable near the center, and is less so on the extremes. This is due to the variability in the samples: as shown in Fig.~\ref{fig:mnist:data:d1} there are samples where the digit 1 has horizontal strokes on the extremes. In Fig.~\ref{fig:mnist:extrap:d8} we see the extrapolation in the spatial domain for one element of the collection \(\mF\). In this case, frequency extrapolation successfully reproduces the two holes in the digit 8, even if these are not present in the low-resolution image. Remark that in this case, the increase in the level of detail is more pronounced in the center region where two strokes intersect.

\begin{figure*}[!htb]
    \centering
    \begin{subfigure}[t]{0.24\textwidth}
        \includegraphics[width=0.99\textwidth]{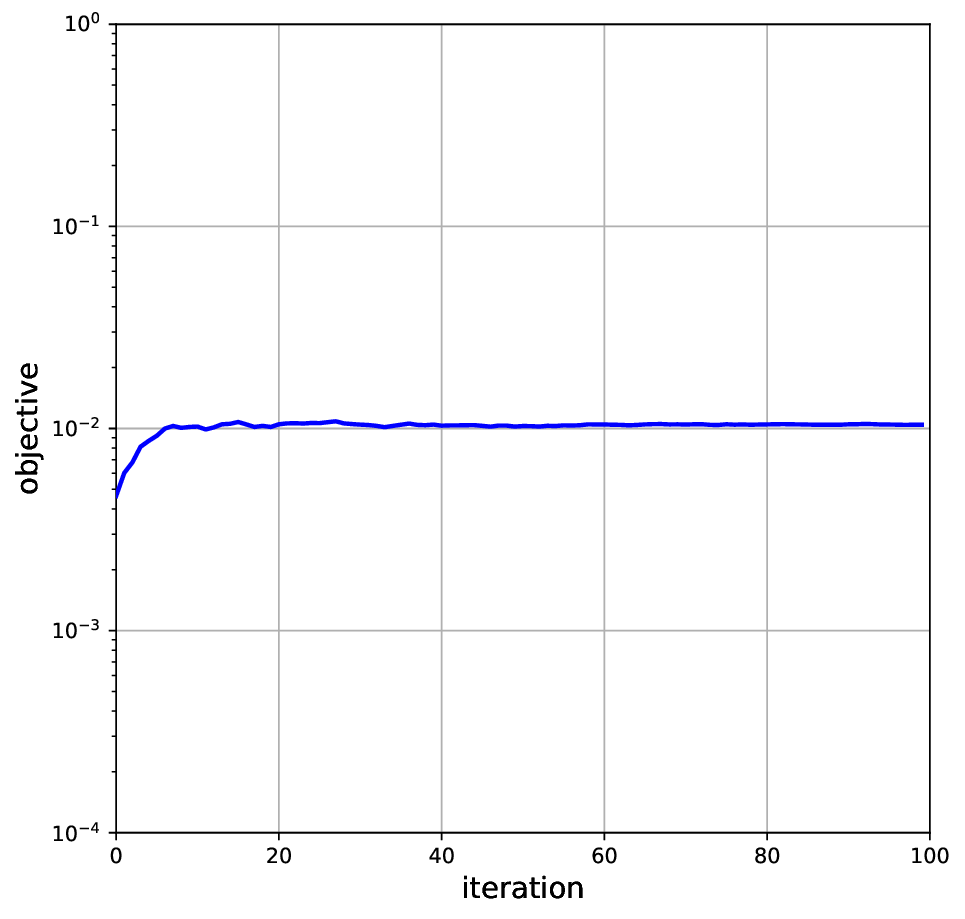}
        \caption{Objective value}
        \label{fig:mnist:fpePerf:obj}
    \end{subfigure}%
    \begin{subfigure}[t]{0.24\textwidth}
        \includegraphics[width=0.99\textwidth]{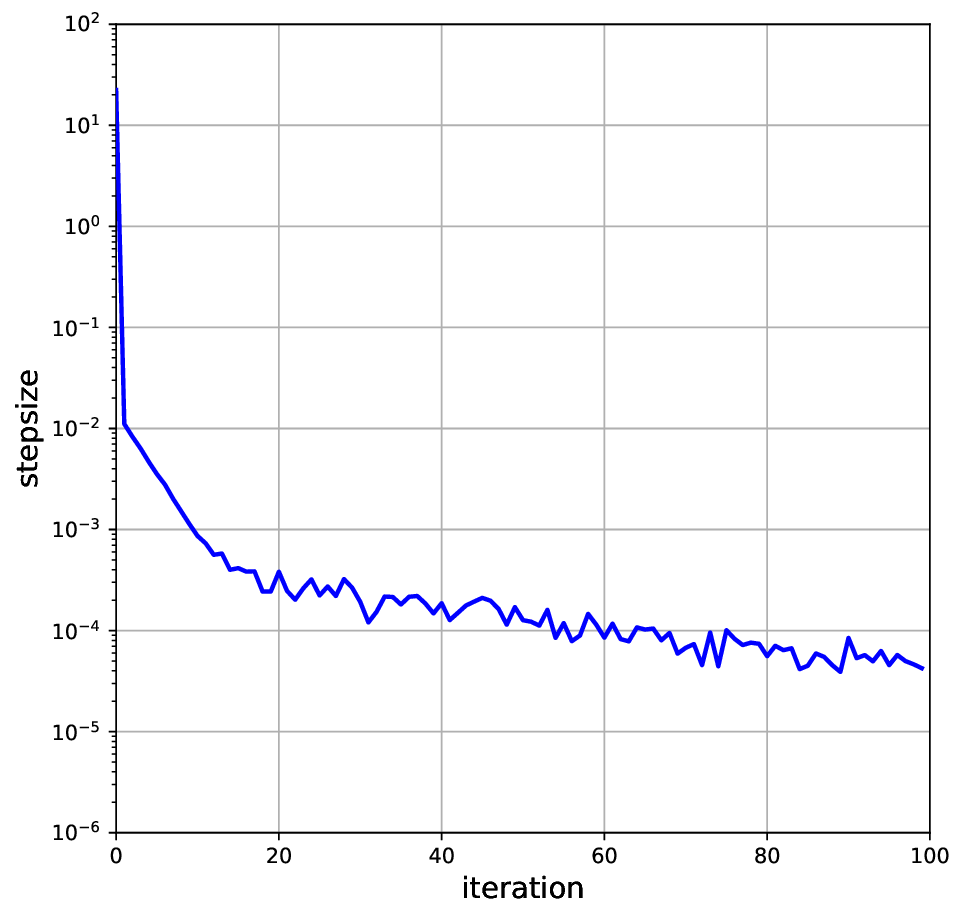}
        \caption{Step size}
        \label{fig:mnist:fpePerf:dxnrm}
    \end{subfigure}%
    \begin{subfigure}[t]{0.24\textwidth}
        \includegraphics[width=0.99\textwidth]{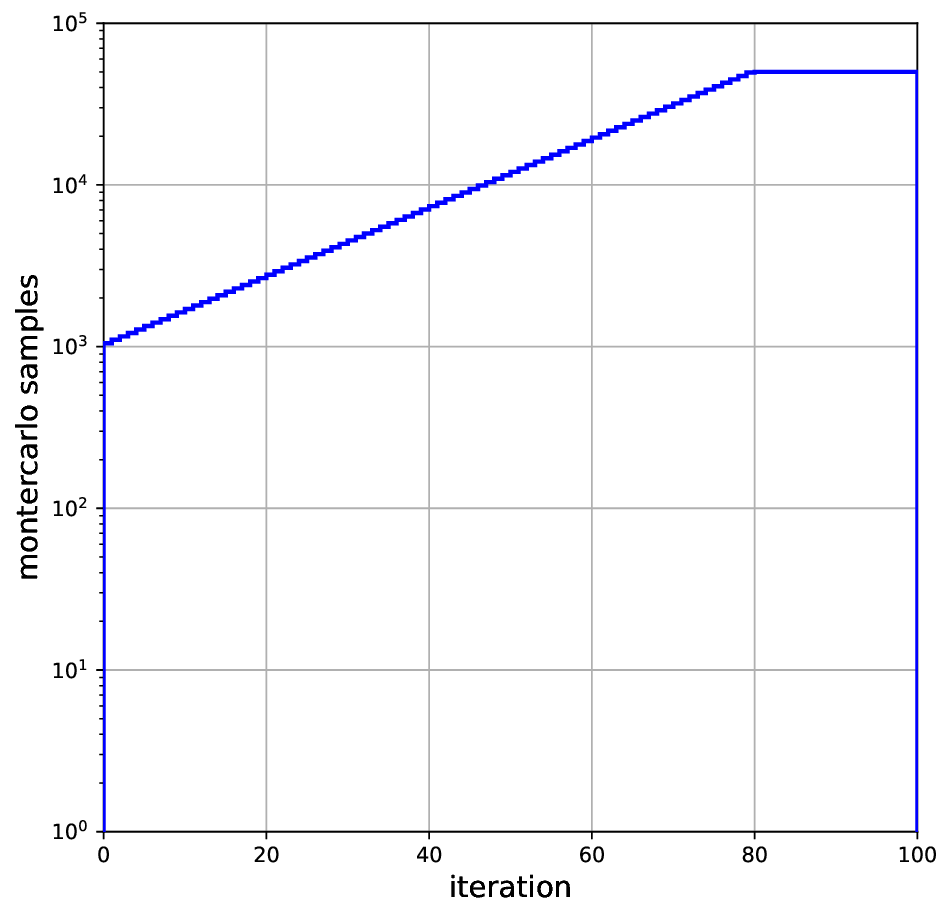}
        \caption{Monte Carlo samples}
        \label{fig:mnist:fpePerf:nnodes}
    \end{subfigure}%
    \begin{subfigure}[t]{0.24\textwidth}
        \includegraphics[width=0.99\textwidth]{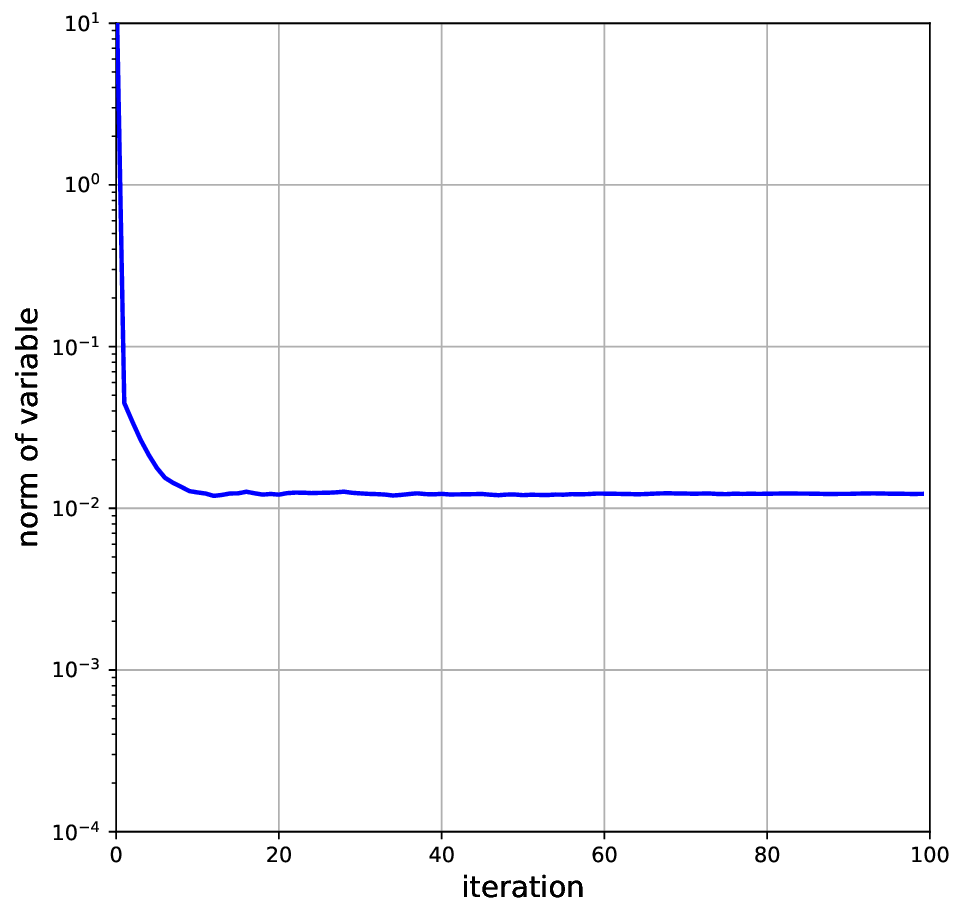}
        \caption{Norm of variable}
        \label{fig:mnist:fpePerf:xnrm}
    \end{subfigure}%

    \caption{Performance metrics for the regularized fixed-point iteration.
    }
    \label{fig:mnist:fpePerf}
\end{figure*}

\begin{figure*}[!htb]
    \centering
    \begin{subfigure}[t]{0.95\textwidth}
        \includegraphics[width=0.99\textwidth]{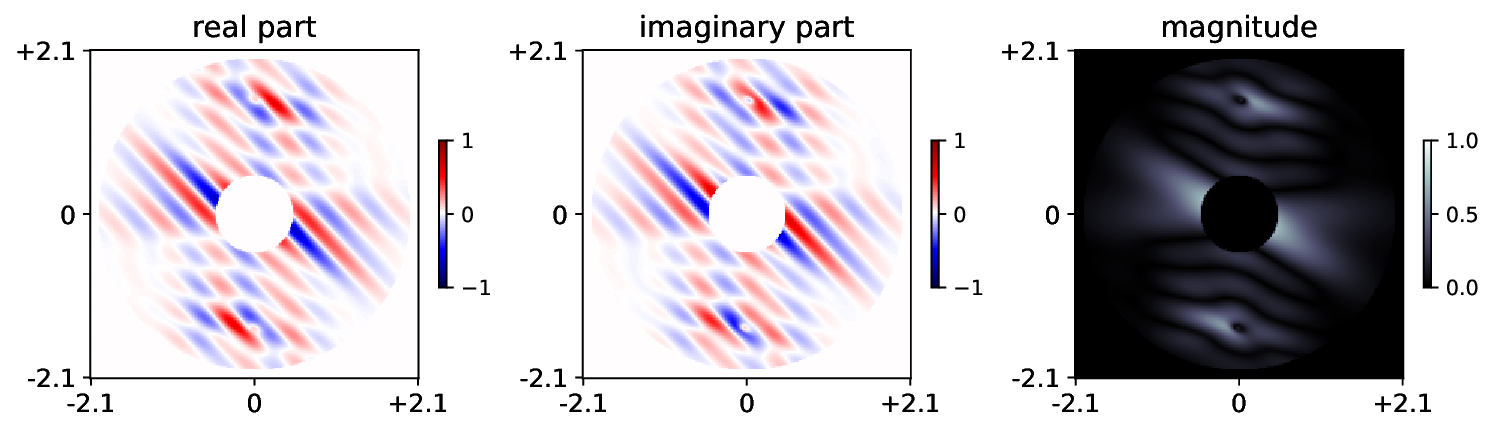}
        \caption{Optimal multiplier for digit 1 in the frequency domain}
        \label{fig:mnist:mult:d1freq}
    \end{subfigure}\\
    \vspace{4pt}
    \begin{subfigure}[t]{0.95\textwidth}
        \includegraphics[width=0.99\textwidth]{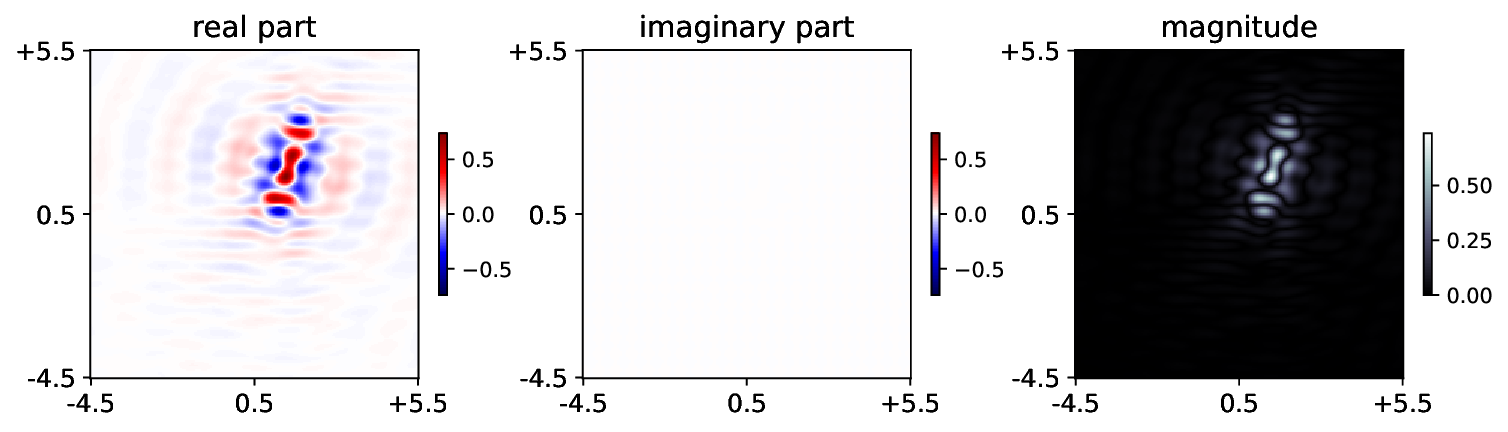}
        \caption{Optimal multiplier for digit 1 in the spatial domain}
        \label{fig:mnist:mult:d1space}
    \end{subfigure}\\
    \vspace{4pt}
    \begin{subfigure}[t]{0.95\textwidth}
        \includegraphics[width=0.99\textwidth]{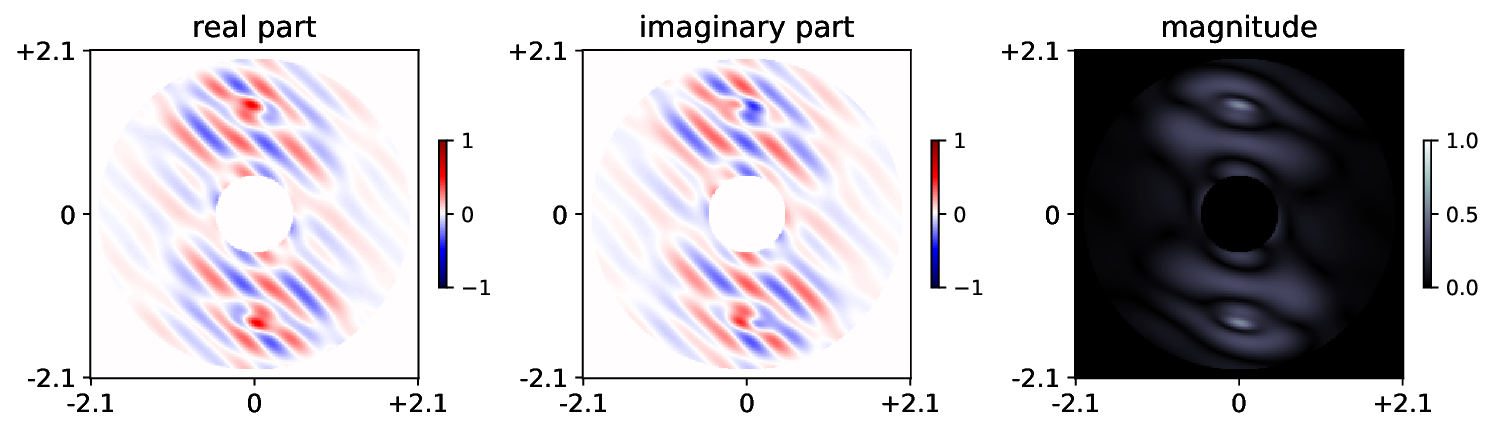}
        \caption{Optimal multiplier for digit 8 in the frequency domain}
        \label{fig:mnist:mult:d8freq}
    \end{subfigure}\\
    \vspace{4pt}
    \begin{subfigure}[t]{0.95\textwidth}
        \includegraphics[width=0.99\textwidth]{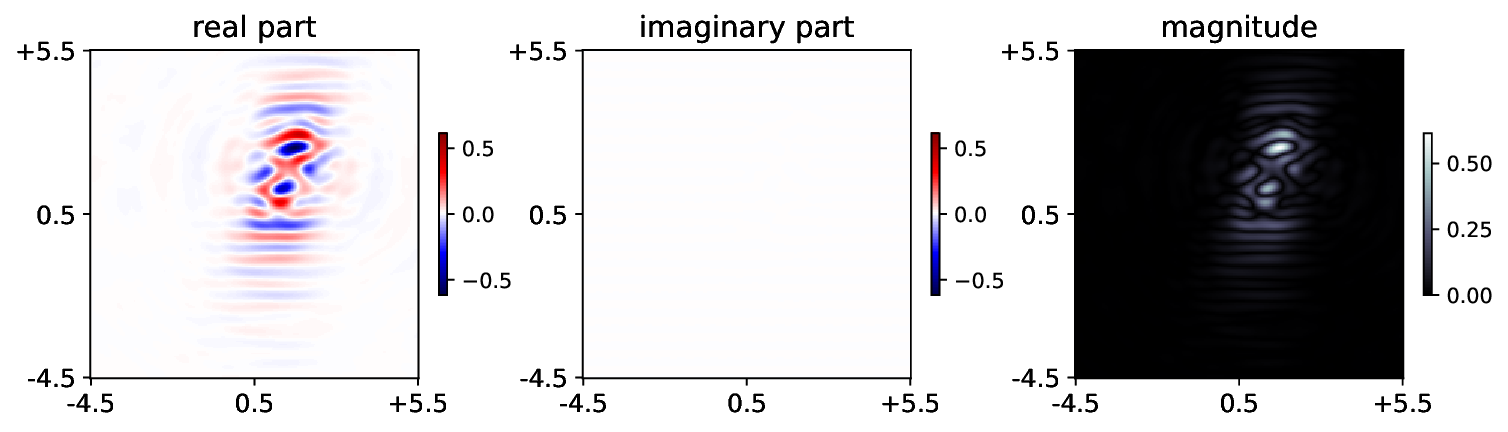}
        \caption{Optimal multiplier for digit 8 in the spatial domain}
        \label{fig:mnist:mult:d8space}
    \end{subfigure}%

    \caption{Optimal multipliers computed from 500 samples for the digit 1 (a, b) and 8 (c, d) in the frequency domain (a, c) and in the spatial domain (b, d).
    }
    \label{fig:mnist:mult}
\end{figure*}

\begin{figure*}[!h]
    \centering
    \begin{subfigure}[t]{0.95\textwidth}
        \includegraphics[width=0.99\textwidth]{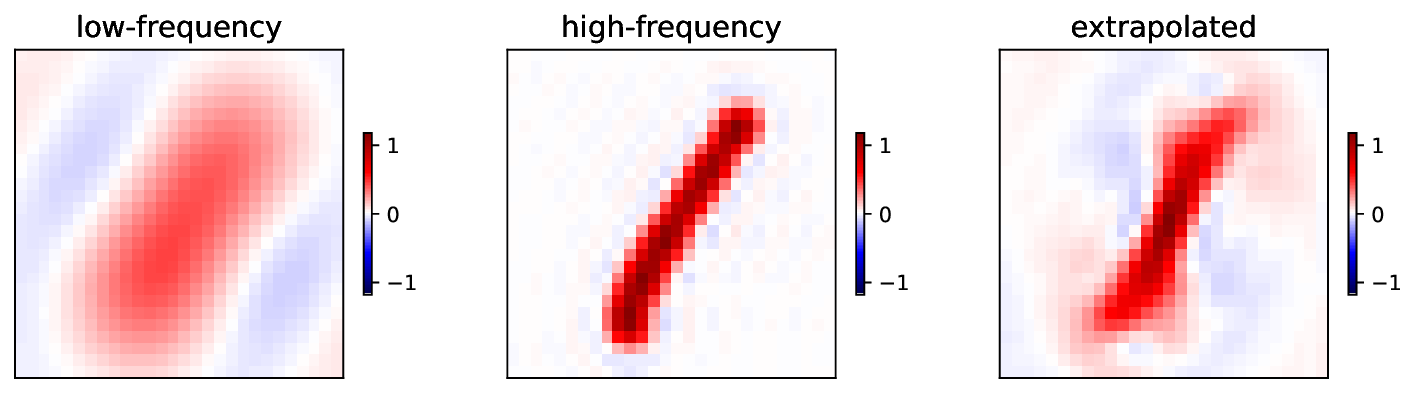}
        \caption{Sample image for the digit 1}
        \label{fig:mnist:extrap:d1}
    \end{subfigure}\\
    \begin{subfigure}[t]{0.95\textwidth}
        \includegraphics[width=0.99\textwidth]{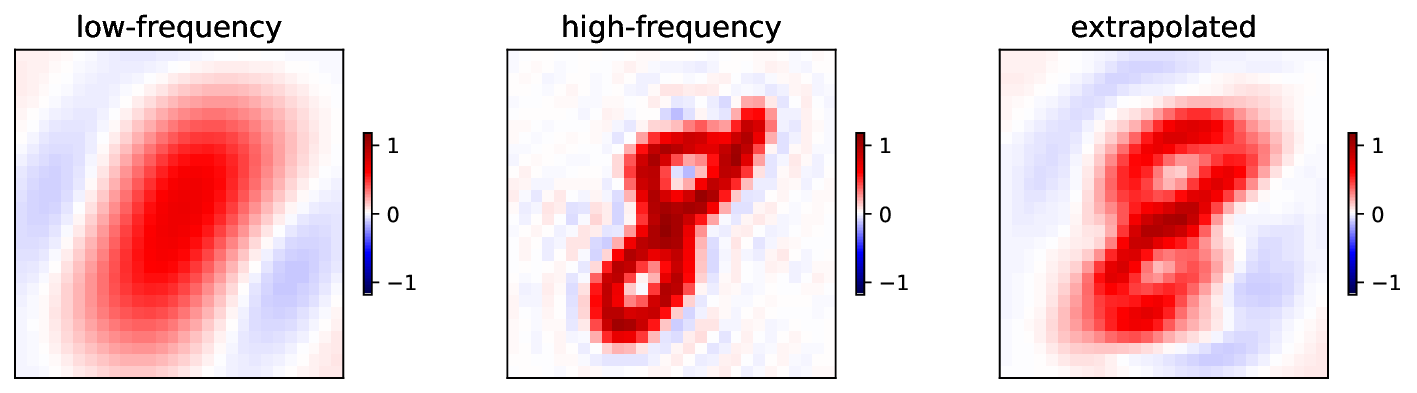}
        \caption{Sample image for the digit 8}
        \label{fig:mnist:extrap:d8}
    \end{subfigure}\\
    \begin{subfigure}[t]{0.95\textwidth}
        \includegraphics[width=0.99\textwidth]{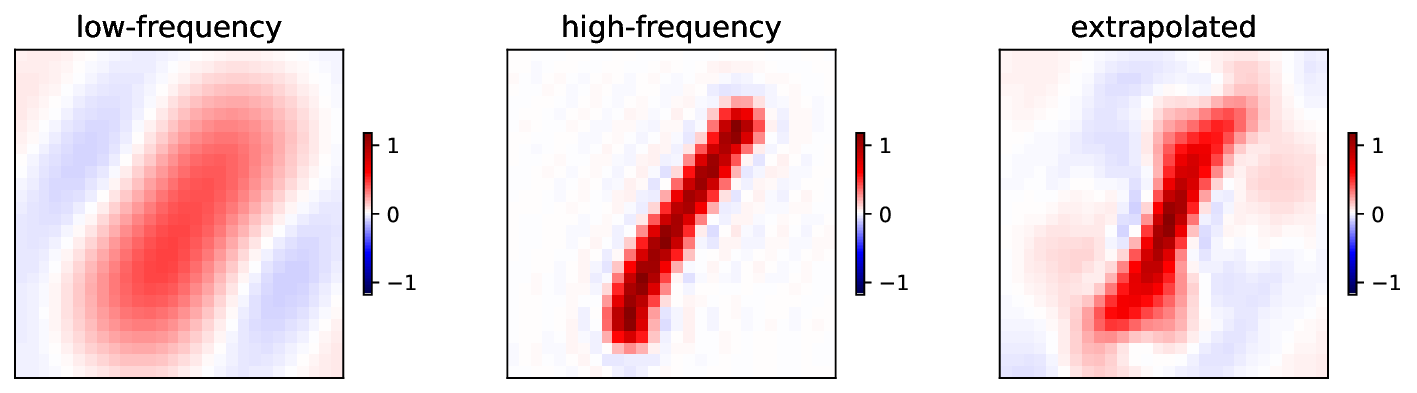}
        \caption{Sample image for the digit 1 using horizontal and vertical multipliers}
        \label{fig:mnist:extrap:d1ani}
    \end{subfigure}\\
    \begin{subfigure}[t]{0.95\textwidth}
        \includegraphics[width=0.99\textwidth]{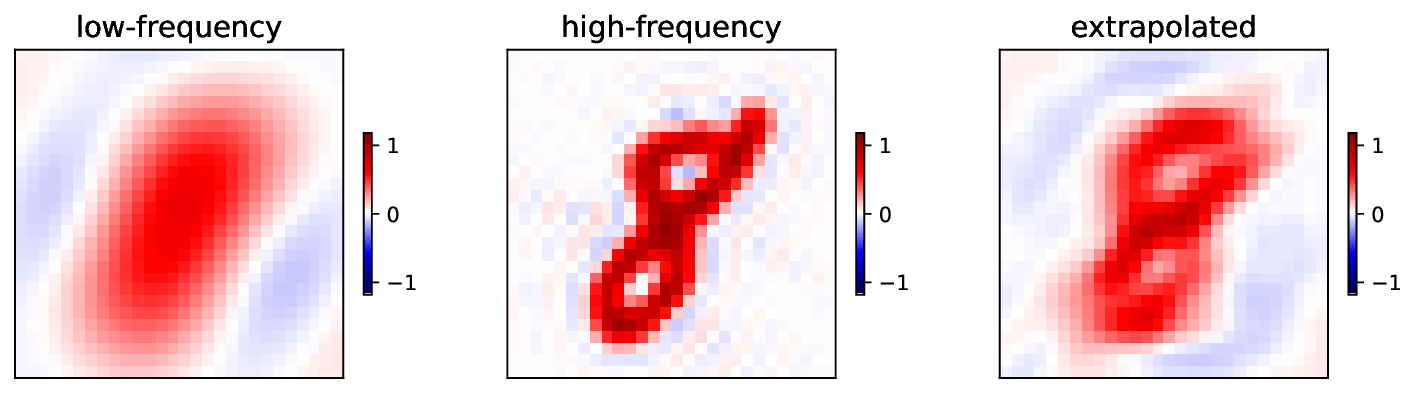}
        \caption{Sample image for the digit 8 using horizontal and vertical multipliers}
        \label{fig:mnist:extrap:d8ani}
    \end{subfigure}%

    \caption{Effect of using the optimal multiplier to perform frequency extrapolation to improve the resolution on a sample image for the digits 1 and 8. The low-frequency image is the approximation of the image obtained from the frequency information on \(\Omlo\) whereas the high-frequency image is the approximation obtained from the frequency information on \(\Omhi\). (a, b) Extrapolation using the optimal multiplier on \(\Omo\). (c, d) Extrapolation using the optimal horizontal and vertical multipliers.
    }
    \label{fig:mnist:extrap}
\end{figure*}

%% ============================================================================
\subsubsection{The effect of localizing horizontal and vertical spatial frequencies}
%% ============================================================================

An interesting question relates to the impact that the domain \(\Omo\) has on the optimal multiplier. In this experiment, we decompose the annular region \(\Omo\) into 4 circular sectors as shown in Fig.~\ref{fig:mnist:data:sectors}. We pair the horizontal sectors and the vertical sectors to preserve the conjugate symmetry in the elements of \(\mF\). This yields a subdivision of \(\Omo\) into two regions representing the essentially horizontal and essentially vertical spatial frequencies. We compute the optimal multipliers associated to each region, which we call the optimal horizontal and vertical multipliers accordingly. Each optimal multiplier is computed in the same manner as in the previous experiment, with the exception that now the algorithm performs 200 iterations.

In Fig.~\ref{fig:mnist:aniMult} we see the optimal filters for the digits 1 and 8 in the spatial domain, i.e., the optimal horizontal and vertical filters. The optimal horizontal filter for the digit 1 in Fig.~\ref{fig:mnist:aniMult:d1H} successfully localizes the orientation of the digit 1. Once again, the filter is localized near the middle of the digit, and becomes diffuse toward its extremes. In contrast, the optimal vertical filter shown in Fig.~\ref{fig:mnist:aniMult:d1V} spreads along the length of the digit. Similarly, the optimal horizontal and vertical filters for the digit 8 in Fig.~\ref{fig:mnist:aniMult:d1H} and Fig.~\ref{fig:mnist:aniMult:d1V} capture the relevant structures, including the presence of two holes in the shape.

In Fig.~\ref{fig:mnist:extrap:d1ani} and Fig.~\ref{fig:mnist:extrap:d8ani} we see the extrapolation in the spatial domain for the same element of the collection \(\mF\) as in Fig.~\ref{fig:mnist:extrap:d1} and Fig.~\ref{fig:mnist:extrap:d8} respectively. Remark that in this simple experiment there does not seem to be any adaptivity: although the optimal horizontal and vertical filters capture different features in the shapes, when put together, the increase in detail in the images is almost the same as that obtained with a single optimal multiplier defined over all of \(\Omo\).

\begin{figure*}[!htb]
    \centering
    \begin{subfigure}[t]{0.4\textwidth}
        \includegraphics[width=0.99\textwidth]{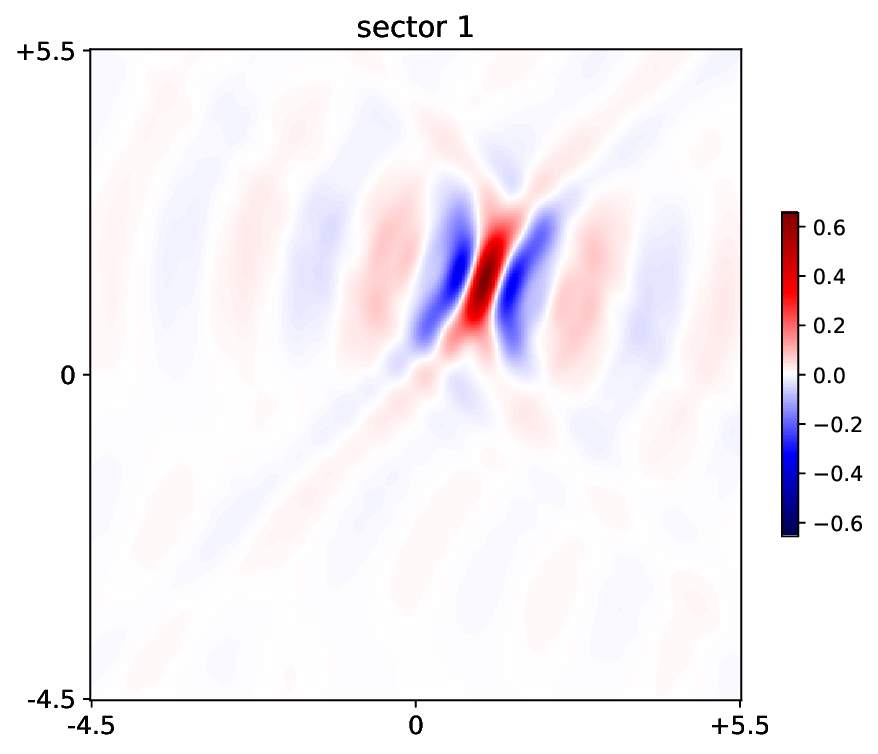}
        \caption{Optimal horizontal multiplier for digit 1}
        \label{fig:mnist:aniMult:d1H}
    \end{subfigure}%
    \begin{subfigure}[t]{0.4\textwidth}
        \includegraphics[width=0.99\textwidth]{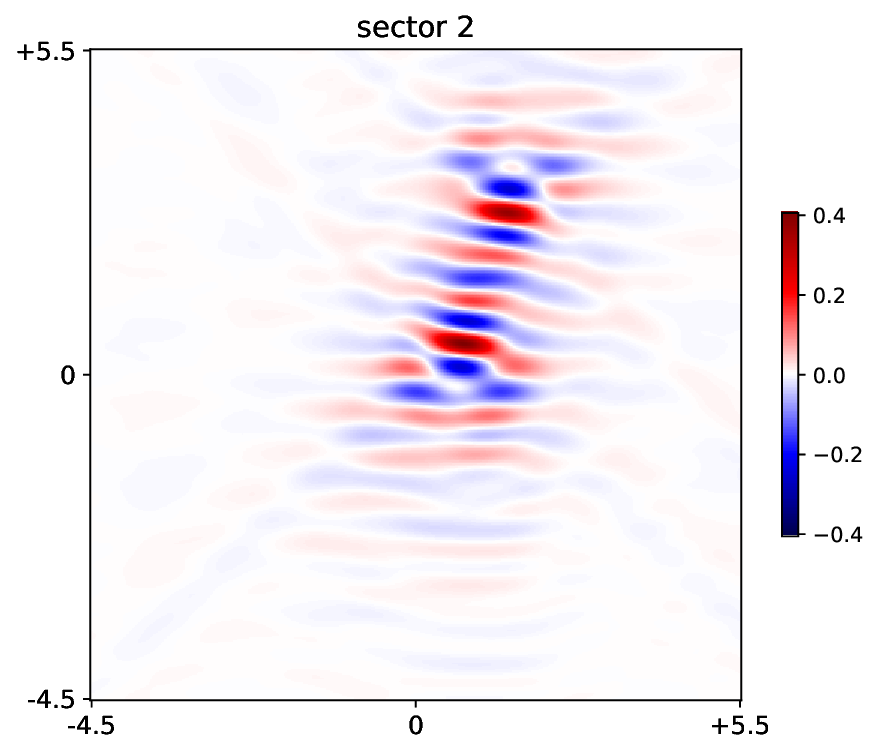}
        \caption{Optimal vertical multiplier for digit 1}
        \label{fig:mnist:aniMult:d1V}
    \end{subfigure}\\
    \vspace{4pt}
    \begin{subfigure}[t]{0.4\textwidth}
        \includegraphics[width=0.99\textwidth]{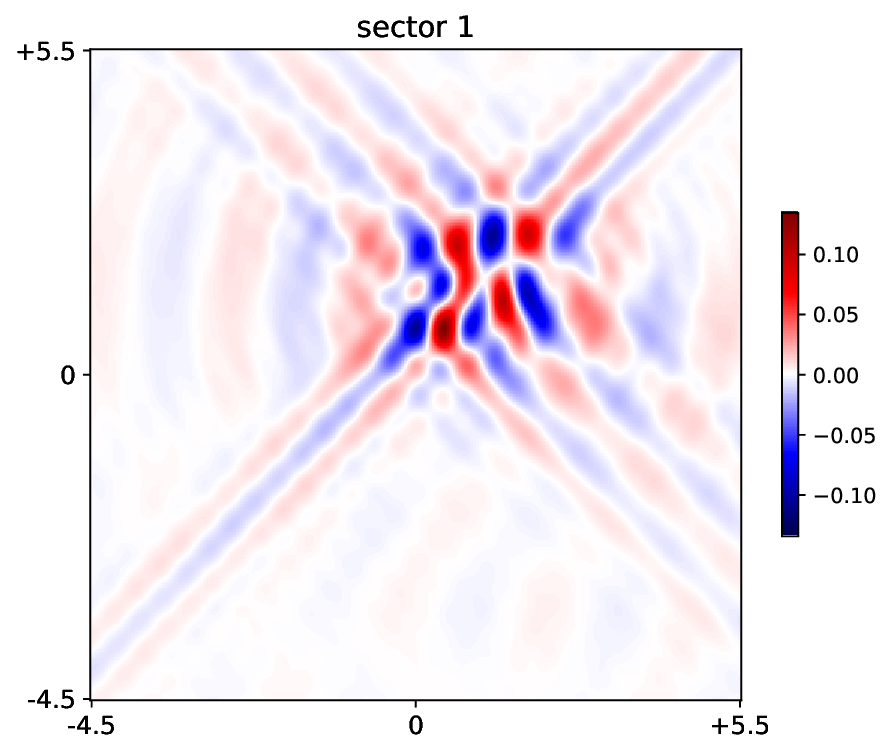}
        \caption{Optimal horizontal multiplier for digit 8}
        \label{fig:mnist:aniMult:d8H}
    \end{subfigure}%
    \begin{subfigure}[t]{0.4\textwidth}
        \includegraphics[width=0.99\textwidth]{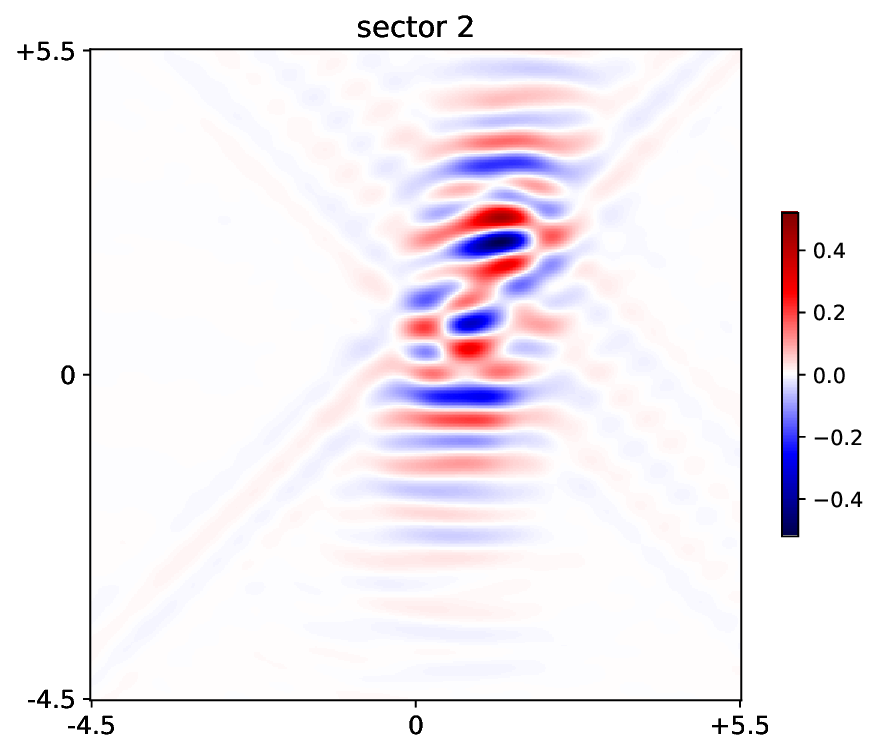}
        \caption{Optimal vertical multiplier for digit 8}
        \label{fig:mnist:aniMult:d8V}
    \end{subfigure}%

    \caption{Optimal horizontal and vertical multipliers. Only the real part is shown as the imaginary part is negligible. (a, b) Optimal multipliers for digit 1. (c, d) Optimal multipliers for digit 8.
    }
    \label{fig:mnist:aniMult}
\end{figure*}

%% ============================================================================
\section{Conclusions and future work}
\label{sec:conclusion}
%% ============================================================================

In this work, we have analyzed the problem of extrapolation in frequency over a collection of functions \(\mU\) supported on the unit cube \(\mQ\) using multipliers. We have shown that there exist multipliers that minimize the worst-case error and that these have a canonical structure. This canonical structure represents how the low-frequency content encodes the high-frequency content accross the collection \(\mU\). In the case of a finite \(\mU\) this canonical structure can be parameterized in terms of an Hermitian matrix, enabling the use of a fixed-point iteration to find the optimal multiplier. Finally, we connect our results to the literature in multiresolution, showing how the optimal multiplier can be used to induce a multiresolution.

There are a several open questions about this approach that are of independent interest. We have shown that when the collection \(\mU\) is formed by translates of a single function, the optimal multiplier depends on their translations. It is of interest to understand how to remove this dependence. Whether this can be achieved, for example by defining a suitable worst-case error or by suitably expanding or reducing the collection \(\mU\), is an open question.

We have only focused on a fixed \(\alpha\). A relevant question is the behavior of the optimal multiplier as \(\alpha \downarrow 1\) and \(\alpha\uparrow \infty\). It is reasonable to think that the behavior of the worst-case error will encode information about the features of the elements of the collection \(\mU\). Finally, our numerical results show that the optimal filters induced by the optimal multipliers can be used to understand global features about the collection \(\mU\). Whether this approach can be used to understand the behavior of data-driven methods for super-resolution, such as convolutional neural networks, is another possible direction for future work.

%% ============================================================================
\section*{Acknowledgements}
%% ============================================================================

D.~C. and C.~SL. were partially funded by grant ANID -- FONDECYT -- 1211643. C.~SL. was also partially funded by the National Center for Artificial Intelligence CENIA FB210017, Basal ANID.

%% ============================================================================
\appendix

%% ============================================================================
\section{Proof of Theorem~\ref{thm:iterationConvergence}}
\label{apx:pf:thm:iterationConvergence}
%% ============================================================================

One of the important properties of orthosymmetric sets is the following.
\begin{proposition}
    Let \(C\subset \Rd\) be convex and orthosymmetric, let \(\vx\in \Rd\) and let \(\vp = \proj_C(\vx)\). The following assertions are true.
    \begin{enumerate}[leftmargin=*,label=(\roman*), itemsep=2pt, parsep=\parskip]
        \item{We have that \(p_i < 0\) if and only if \(x_i < 0\) and \(p_i > 0\) if and only if \(x_i > 0\).
        }
        \item{We have that for any \(\vx\in C\) and \(\vs\in \BInf\) it holds that \(\diag(\vs)\vx \in C\).
        }
    \end{enumerate}
\end{proposition}

We begin with the following auxiliary lemma.

\begin{lemma}\label{lem:iterationConvergence:minEVBoundProjection}
    Suppose that \(W\) satisfies the hypotheses of the statement in Theorem~\ref{thm:iterationConvergence}. If \(\vX \in \Hnn\) is such that \(\vX \succeq 0\) then \(\proj_{W}(\vX) \succeq 0\).
\end{lemma}

\begin{proof}[Proof of Lemma~\ref{lem:iterationConvergence:minEVBoundProjection}]
    \parskip = \pskip

    If \(\vX\in W\) then there is nothing to prove. Thus, without loss we assume that \(\vX\notin W\).

    Let \(\vX = \vV\vLam\vV\adj\) and let \(\lamX \in \Rnp\) be the vector of eigenvalues of \(\vX\) in decreasing order. Then, the objective in
    \[
    \begin{aligned}
        & \underset{}{\text{minimize}}
        & & \nrmF{\vX - \vZ}
        & \text{subject to}
        & & \vZ\in W
    \end{aligned}
    \]
    satisfies
    \[
        \nrmF{\vX - \vZ}^2 = \sums_{i=1}^n(\lam_i - Z'_{i,i})^2 + \sums_{i,j:\, i\neq j} |Z'_{i,j}|^2
    \]
    where \(\vZ' = \vV\adj\vZ\vV\). Therefore, the above problem is equivalent to
    is equivalent to
    \[
        \begin{aligned}
            & \underset{}{\text{minimize}}
            & & \nrm{\lamX - \vd}_2
            & \text{subject to}
            & & \vd\in D_{W}.
        \end{aligned}
    \]
    Let \(\vd\opt\) be the optimal variable. Since \(D_{W}\) is orthosymmetric and \(\lamX \geq 0\) it is straightforward to prove that \(\vd\opt\geq 0\) from where the lemma follows.
\end{proof}

To prove Theorem~\ref{thm:iterationConvergence} we first prove some preliminary bounds and results. First, we find a suitable convex and compact set on which to analyze the fixed-point equation. Choose \(\Delta > 0\) such that
\[
    \Delta \geq \sup\set{\trace(\vX):\,\, \vX\in W}
\]
and \(\Delta > n\) and define the convex and compact set
\[
    \Kd = \set{\vX\in\Hnn:\,\, \vX\succeq 0,\,\, \trace(\vX) \leq \Delta}.
\]
This will be the space on which we will analyze the fixed-point equation. It will be useful to define the condition number \(\kappa = n + \Delta/\delta\).

We now analyze some functions involved in the fixed-point iteration. First, define \(\Md:\Kd \to \LEOmo\) as
\[
    \Md(\vSig) = \frac{\bprod{\vf}{\Da\vf}_{\delta I + \Sig}}{\nrm{\vf}^2_{\delta I + \Sig}}.
\]
% boundedness
We first show it is bounded. Using the Cauchy-Schwarz inequality we obtain
\[
    |\Md(\vSig)| \leq \frac{\nrm{\Da\vf}_{\delta I + \Sig}}{\nrm{\vf}_{\delta I + \Sig}} \leq \frac{1}{\delta} \frac{\trace(\Da\vf\Da\vf^* \vSig)}{\nrmlEn{\vf}}\leq \frac{\delta n + \trace(\vSig)}{\delta}\frac{\nrmlEn{\Da\vf}}{\nrmlEn{\vf}} \leq \kappa\frac{\nrmlEn{\Da\vf}}{\nrmlEn{\vf}}.
\]
This yields both the pointwise and uniform bound
\begin{align*}
    \nrmLE{\Md(\vSig)} &\leq \frac{\sqrt{|\Omo|}(\delta n + \trace(\vSig))}{\delta} \left(\frac{1}{|\Omo|}\int_{\Omo} \frac{\nrmlEn{\Da\vf(\xi)}^2}{\nrmlEn{\vf(\xi)}^2}\, d\xi\right)^{1/2} \\
    &\leq \kappa \sqrt{|\Omo|} \left(\frac{1}{|\Omo|}\int_{\Omo} \frac{\nrmlEn{\Da\vf(\xi)}^2}{\nrmlEn{\vf(\xi)}^2}\, d\xi\right)^{1/2}
\end{align*}
for any \(\vSig \in \Kd\). Let \(R_M = R_M(\delta,\Delta)\) be the uniform bound in the right-hand side.
% Lipschitz continuity
We now prove it is Lipschitz continuous. From
\begin{align*}
    \Md(\vSig_2) - \Md(\vSig_1) &= \frac{\bprodlEn{\vf}{(\vSig_2 - \vSig_1)\vf}}{\nrm{\vf}^2_{\delta I + \Sigma_2}} - \bprodlEn{\vf}{(\delta\vI + \vSig_1)\Da\vf}\frac{\bprodlEn{\vf}{(\vSig_2 - \vSig_1)\Da\vf}}{\nrm{\vf}^2_{\delta I +\Sigma_2}\nrm{\vf}_{\delta I + \Sigma_1}^2}\\
    &= \frac{\trace(\vf\vf^*(\vSig_2 - \vSig_1))}{\nrm{\vf}^2_{\delta I + \Sig_2}} - \trace(\Da\vf\vf^*(\delta \vI + \vSig_1))\frac{\trace(\Da\vf\vf^*(\vSig_2 - \vSig_1))}{\nrm{\vf}^2_{\delta I + \Sigma_2}\nrm{\vf}_{\delta I + \Sigma_1}^2}
\end{align*}
we deduce that
\begin{align*}
    \nrmF{\Md(\vSig_2) - \Md(\vSig_1)} &\leq \frac{\nrmF{\vf\vf^*}}{\delta\nrmlEn{\vf}^2} \nrmF{\vSig_2 - \vSig_1} + \nrmF{\delta \vI + \vSig_1}\frac{\nrmF{\vf\Da\vf^*}^2}{\delta^2 \nrmlEn{\vf}^4}\nrmF{\vSig_2 - \vSig_1}\\
    &\leq \frac{1}{\delta} \nrmF{\vSig_2 - \vSig_1} + \frac{\kappa}{\delta}\frac{\nrmlEn{\Da\vf}^2}{\nrmlEn{\vf}^2}\nrmF{\vSig_2 - \vSig_1}
\end{align*}
where we used the bounds \(\sqrt{n} \leq n\) and \(\nrmF{\vSig_1} \leq \trace(\vSig_1)\). Hence,
\begin{align*}
    \nrmLE{\Md(\vSig_2) - \Md(\vSig_1)} &\leq\frac{\sqrt{\mu(\Omo)}}{\delta} \nrmF{\vSig_2 - \vSig_1} + \frac{\kappa\sqrt{\mu(\Omo)}}{\delta}\left(\frac{1}{\mu(\Omo)}\int_{\Omo}\frac{\nrmlEn{\Da\vf(\xi)}^4}{\nrmlEn{\vf(\xi)}^4}\,d\xi\right)^{1/2}\nrmF{\vSig_2 - \vSig_1}\\
    &\leq \frac{\sqrt{\mu(\Omo)}}{\delta}\left(1 + \kappa\left(\frac{1}{\mu(\Omo)}\int_{\Omo}\frac{\nrmlEn{\Da\vf(\xi)}^4}{\nrmlEn{\vf(\xi)}^4}\,d\xi\right)^{1/2}\right)\nrmF{\vSig_2 - \vSig_1}
\end{align*}
for any \(\vSig_1,\vSig_2\in \Kd\). Let \(L_M= L_M(\delta,\Delta)\) be this Lipschitz constant. Then, we see that
\[
    \Md : \Kd \mapsto R_M\bar{B}_{\LE}
\]
is Lipschitz continuous.
% boundedness
We now analyze the properties of the matrix-valued map \(\vG\). To prove it is Lipschitz continuous, remark that
\begin{align*}
    G_{i,j}(m_2) - G_{i,j}(m_1) &= -\int_{\Omo}(f_i(\xi)\Da f_j(\xi)^*(m_2(\xi)- m_1(\xi)) + \Da f_i(\xi)f_j(\xi)^* (m_2(\xi)- m_1(\xi))^*)\,d\mu(\xi) \\
    &\quad +\:\int_{\Omo} f_i(\xi) f_j(\xi)^* (|m_2(\xi)| + |m_1(\xi)|)(|m_2(\xi)| - |m_1(\xi)|)\, d\mu(\xi).
\end{align*}
It follows that
\begin{align*}
    |G_{i,j}(m_2) - G_{i,j}(m_1)| &\leq  (\nrmVF{f_i}\nrmVF{\Da f_j} + \nrmVF{\Da f_i}\nrmVF{f_j})\nrmLE{m_2 - m_1} \\
    &\quad +\: \nrmVF{f_i}\nrmVF{f_j}(\nrmLE{m_2} + \nrmLE{m_1})\nrmLE{m_2 - m_1}.
\end{align*}
and we obtain the bound
\[
    \nrmF{\vG(m_2) - \vG(m_1)} \leq n R_{\mF} (2 + \nrmLE{m_2} + \nrmLE{m_1})\nrmLE{m_2 - m_1}.
\]
Define the constant
\[
    R_{\mF} := \max\set{\nrmVF{f_i},\, \nrmVF{\Da f_i}:\,\, i,j\in\set{1,\ldots, n}}.
\]
We conclude that for \(\vSig_2,\vSig_1\in \Kd\) we have that
\[
    \nrmF{\vG(\Md(\vSig_2)) - \vG(\Md(\vSig_1))} \leq 2 n R_{\mF} L_M (1 + R_M)\nrmF{\vSig_2 - \vSig_1}.
\]
% fixed-point map
Now, define the map
\[
    H(\vSig) = \proj_{W}(\tS\vSig + \tG G(\Md(\vSig))).
\]
Let \(\vSig \in \Kd\) and observe that from Lemma~\ref{lem:iterationConvergence:minEVBoundProjection} it follows that
\[
    \lam_{\min}(H(\vSig)) \geq 0.
\]
In addition, our choice of \(\Delta\) implies that
\[
    \trace(H(\vSig)) \leq \Delta.
\]
Hence, \(H:\Kd\to\Kd\) and, from our previous result and the fact that projections are non-expansive, it follows that \(H\) is Lipschitz continuous. Therefore, by Schauder's fixed-point theorem~\cite[Theorem~12.19]{fabian_banach_2011}, it has at least one fixed point on \(\Kd\). We can show that it is, in fact, contractive for a suitable choice of \(\tS\) and \(\tG\). In fact,
\begin{align*}
    \nrmF{H(\vSig_2) - H(\vSig_1)} &\leq \tS\nrmF{\vSig_2 - \vSig_1} + \tG\nrmF{G(\Md(\vSig_2)) - G(\Md(\vSig_1))}\\
    &\leq (\tS + 2n \tG R_{\mF} L_M (1 + R_M)) \nrmF{\vSig_2 - \vSig_1}
\end{align*}
whence it suffices to choose
\[
    \tS + 2n \tG R_{\mF} L_M (1 + R_M) < 1.
\]
From the previous calculations, it is clear that there exists constants \(C_1,C_2,C_3 > 0\) independent of \(\delta\) and \(\Delta\) such that
\[
    2n R_{\mF} L_M(1 + R_M) = \sums_{k=1}^3 C_k\delta^{-k}
\]
from where the theorem follows.

\section{Proofs for Section~\ref{sec:connectionsToMultiresolution}}\label{proof:connectionsToMRA}

In this section we assume throughout that \(\Omo = [-1/2, 1/2]^d\) for any \(d\in \N\) unless stated otherwise. The results in Section~\ref{sec:connectionsToMultiresolution} are a consequence of the following two auxiliary results.

\begin{lemma}\label{lem:boundedMultiplierConvergentCascade}
    Let \(m:\Omo \to \C\) be measurable, bounded almost everywhere, and such that \(m(0) = 1\). If \(m\) is \(\gamma\)-Hölder continuous on a neighborhood of the origin for some \(\gamma\in (0, 1]\) then using the \(1\)-periodic extension of \(m\) as a refinement mask produces a sequence that converges pointwise to a function that is finite almost everywhere. Furthermore, the convergence is almost everywhere uniform on compact subsets of \(\Rd\).
\end{lemma}

\begin{proof}[Proof of Lemma~\ref{lem:boundedMultiplierConvergentCascade}]
    \parskip = \pskip

    Let \(Z\subset \Omo\) be the set on which \(m\) is not bounded, and let \(C_m > 0\) be a bound for \(|m|\) on \(\Omo\setminus Z\). Let \(m_1\) be the \(1\)-periodization of \(m\). Then the set \(Z_1\) on which \(m\) is not bounded has measure zero. As a consequence, the set
    \[
        Z_1^* := \bigcups_{j\in\No} 2^{-j} Z_1
    \]
    has measure zero. Remark that
    \[
        \xi\notin Z_1^*\quad\Rightarrow\quad \forall\,j\in\No:\,\, |m_1(2^{-j}\xi)| \leq B.
    \]
    By hypothesis, there exists \(\delta > 0\) such that
    \[
        \xi_1,\xi_2\in \delta \BE:\,\, |m(\xi_1) - m(\xi_2)| \leq L_\gamma \absE{\xi_1 - \xi_2}^{\gamma}.
    \]
    Since without loss we may assume that \(\delta \BE \subset \Omo\) this implies that
    \[
        \xi\in \delta \BE \setminus Z_1^*:\,\, |m_1(\xi) - 1| \leq L_\gamma \absE{\xi}^{\gamma}.
    \]
    It is also apparent that
    \[
        \xi\in (\Rd \setminus \delta \BE) \setminus Z_1^*:\,\, |m_1(\xi) - 1| \leq \frac{1 + C_m}{\delta^{\gamma}}\absE{\xi}^{\gamma}.
    \]
    Combining both inequalities we conclude that for a suitable \(C_1 > 0\) we have that
    \[
        \xi \in \Rd\setminus Z_1^*:\,\, |m_1(\xi) - 1| \leq C_1 \absE{\xi}^{\gamma}.
    \]
    Let \(\xi\in Z_1^*\). Then \(\alpha^{-j} \xi \notin Z_1^*\) for all \(j\in\Z\) and for all \(N\in \N\) the partial products
    \[
        p_N(\xi) := \prods_{j = 1}^{N} m_1(2^{-j}\xi)
    \]
    are finite. From Theorem~250 in Section~57 in~\cite{knopp_theory_1990} the products converge if and only if there exists \(j_0\in\N\) for which the series
    \[
        \sums_{j>j_0}\log(1 + |m_1(2^{-j} \xi) - 1|)
    \]
    converges. Note that
    \[
        \sums_{j>j_0}\log(1 + |m_1(2^{-j} \xi) - 1|)\leq \sums_{j>j_0}|m_1(2^{-j} \xi) - 1| \leq \absE{\xi}^\gamma C_1 \sums_{j > j_0} 2^{-\gamma j} =  \frac{C_1 2^{-\gamma (j_0 + 1)}}{1 - 2^{-\gamma}}\absE{\xi}^{\gamma} .
    \]
    Therefore, we define \(\wh{\vphi}\) as a pointwise limit
    \[
        \wh{\vphi}(\xi) := \begin{cases}
            \prod_{j\in\N} m_1(2^{-j} \xi) & \xi \notin Z_1^*\\
            0 & \xi \in Z_1^*
        \end{cases}
    \]
    Finally, let \(K \subset \Rd\) be a compact set, and let \(\xi\in K\setminus K_1^*\). Let \(N_0\in \N\) and let \(M > N > N_0\). We see that
    \[
        |p_M(\xi) - p_N(\xi)| = |p_N(\xi)|\left|1 - \prods_{j=N+1}^M m_1(2^{-j}\xi)\right| \leq e^{C_\gamma \absE{\xi}^{\gamma}}\left|1 - \prods_{j=N+1}^M m_1(2^{-j}\xi)\right|
    \]
    where \(C_\gamma >0\) is independent of \(\xi\). Hence, the exponential factor remains bounded uniformly on \(M,N\) for \(\xi\) on a compact set. The remaining term can be bounded as
    \[
        \left|1 - \prods_{j=N+1}^M m_1(2^{-j}\xi)\right| \leq |m_1(2^{-M}\xi) - 1|\prods_{j=N+1}^{M-1} |m_1(2^{-j}\xi)| + \left|1 - \prods_{j=N+1}^{M-1} m_1(2^{-j}\xi)\right|.
    \]
    By iterating this inequality we obtain
    \[
        \left|1 - \prods_{j=N+1}^M m_1(2^{-j}\xi)\right| \leq \sums_{k=1}^{M-N} |m_1(2^{-(M-(k- 1))}\xi) - 1|\prods_{j=N+1}^{M-k} (|m_1(2^{-j}\xi) - 1| + 1).
    \]
    If \(N_0\) is sufficiently large so that \(2^{-N} \xi \in \delta \BE\) for \(N > N_0\) then the product in each factor can be bounded by using its logarithm
    \[
        \sums_{j=N+1}^{M-k} \log(1 + |m_1(2^{-j}\xi) - 1|) \leq 2^{-(N+1)\gamma}\absE{\xi}^{\gamma} \frac{L_{\gamma}}{1 - 2^{-\gamma}}.
    \]
    Note that the bound tends to zero and remains bounded uniformly for \(\xi\) on a compact set. Finally, we have the bound
    \[
        \sums_{k=1}^{M-N} |m_1(2^{-(M-(k- 1))}\xi) - 1| \leq L_{\gamma}\absE{\xi}^{\gamma}\sums_{k=N + 1}^{M} 2^{-k\gamma} \leq 2^{-(N+1)\gamma} \absE{\xi}^{\gamma}\frac{L_{\gamma}}{1- 2^{-\gamma}}
    \]
    where the upper bound tends to zero uniformly on \(\xi\) on a compact set. This yields the claim.
\end{proof}

Although this result is similar to others in the literature, we formulate it directly in the frequency domain rather than in terms of the refinement mask.

The following result is a straightforward extension of Lemma~3.2 in~\cite{daubechies_orthonormal_1988}. We omit the proof for brevity.

\begin{lemma}\label{lem:windowedMultiplierDecayingLimitCascade}
    Let \(m:\Omo\to \C\) be bounded almost everywhere. Suppose that \(m\) can be decomposed as
    \[
        m(\xi) = w_N(\xi) m_0(\xi)
    \]
    where \(m_0:\Omo\to \C\) is bounded almost everywhere, it is \(\gamma\)-Hölder continuous on a neighborhood of the origin for some \(\gamma\in (0, 1]\) and \(m_0(0) = 1\). Then there exists a constant \(C_m >0\) such that for almost all \(\xi\in\Rd\) we have that
    \[
        \left|\prods_{j\in\N} m(2^{-j} \xi)\right| \leq C_m (1 + \absE{\xi})^{-N + 2^{C_p}}
    \]
    where \(C_p > 0\) is an almost everywhere bound for \(m_0\).
\end{lemma}

%% ============================================================================
\subsection{Proof of Theorem~\ref{thm:sigmaMultiplierDefiniteInducesMRA}}
%% ============================================================================

Since \(\ZF = \emptyset\) and \(\vSig\) is positive definite, we deduce from Proposition~\ref{prop:sigmaMultiplierBoundedForPSDMatrix} that \(m_{\Sig}\) is real-analytic on \(\Omo\). In particular, it is continuous and Lipschitz continuous near the origin. Finally, it is apparent that \(m_{\Sig}(0) = 1\). Therefore, from Lemma~\ref{lem:boundedMultiplierConvergentCascade} we can construct a continuous refinable function \(\vphi\) using the cascade algorithm. Furthermore, Lemma~\ref{lem:windowedMultiplierDecayingLimitCascade} shows that we can construct a refinable function \(\vphi\) with sufficient decay from \(w_N m_{\Sig}\) for \(N\) sufficiently large.

%% ============================================================================
\subsection{Proof of Theorem~\ref{thm:sigmaMultiplierSemidefinite1DInducesMRA}}
%% ============================================================================

If \(\ZF = \emptyset\) and \(\vSig\) is positive semidefinite, the function \(\nrmS{\vf}^2:\R\to\R\) may become zero on \(\Omo\). If \(\nrmS{\vf}^2 \equiv 0\) then there are non-trivial scalars \(\alpha_1,\ldots,\alpha_n\) for which \(\alpha_1 f_1 + \ldots + \alpha_n f_n \equiv 0\) contradicting the fact that \(f_1,\ldots, f_n\) are linearly independent on \(\Omo\). Hence, \(\nrmS{\vf}^2\) is not identically zero and, as it is real-analytic, it can only have isolated zeros. Since \(\Omo\) is compact, the zeros are finite. Suppose that \(\xio\in\Omo\) is a zero of order \(p > 0\). Then we may write
\begin{align*}
    \nrmS{\vf(\xi)}^2 &= d_p (\xi - \xio)^{p} + (\xi - \xio)^{p+1} R_D(\xi)\\
    \bprodS{\vf}{\Da\vf} &= n_q (\xi - \xio)^{q} + (\xi - \xio)^{q+1} R_N(\xi)
\end{align*}
where \(d_p, n_q\neq 0\), and \(R_D,R_N\) are continuous functions. Since \(m_{\Sig}\in \LEOmo\) by hypothesis, we have that for any open interval \(I\) containing \(\xio\) we have that
\[
    \infty > \int_I |m_{\Sig}(\xi)|^2\, d\xi = \int_I \frac{1}{(\xi - \xio)^{2(p - q)}} \LabsE{\frac{d_p + (\xi - \xio) R_D(\xi)}{n_p + (\xi - \xio)R_D(\xi)}}^2\, d\xi.
\]
This is only possible if \(2(p - q) > -1\) or \(q < p + 1/2\). However, since \(p,q\) are integers, this implies \(q \leq p\) and \(\xio\) must be a zero of \(\bprodS{\vf}{\Da\vf}\) of order at least \(p\). In particular, the singularity of \(m_{\Sig}\) at \(\xio\) is removable, and \(m_{\Sig}\) is bounded on \(\Omo\).

If the origin is not a zero of \(\nrmS{\vf}\) then the same arguments in the proof of Theorem~\ref{thm:sigmaMultiplierDefiniteInducesMRA} allows us to conclude in this case. If the origin is a zero of \(\nrmS{\vf}\) of order \(p\) then let \(\vv = \vSig^{1/2}\vf\) and expand the functions
\begin{align*}
    \vv(\xi) &= \frac{1}{p!} \vv^{(p)}(0) \xi^p + \sums_{k > p} \frac{1}{k!} \vv^{(k)}(0)\xi^k\\
    \Da\vv(\xi) &= \frac{2^p}{p!} \vv^{(p)}(0) \xi^p + \sums_{k > p} \frac{2^k}{k!} \vv^{(k)}(0)\xi^k.
\end{align*}
The expansions are the same up to a factor \(2^p\). Hence,
\[
    \bprodS{\vf(\xi)}{\D\vf(\xi)} = \bprodlEn{\vv(\xi)}{\D\vv(\xi)} = 2^p \nrmlEn{\vv(\xi)}^2 + \sums_{k > p} \frac{2^p(2^{k-p} - 1)}{k!} \bprodlEn{\vv(\xi)}{\vv^{(k)}(0)}\xi^p.
\]
Consequently,
\[
    \lim_{\xi\to 0} m_{\Sig}(\xi) = 2^p.
\]
The same arguments show that \(m_{\Sig}\) is Lipschitz continuous near the origin. Therefore, we may replace \(m_{\Sig}\) by its scaling \(2^{-p}m_{\Sig}\) to apply the same arguments in the proof of  Theorem~\ref{thm:sigmaMultiplierDefiniteInducesMRA} to conclude.

%% ============================================================================
\subsection{Proof of Lemma~\ref{lem:sigmaMultiplierSingularityBehavior}}
%% ============================================================================

If \(m\) is a \(\Sigma\)-multiplier with \(\vSig \succeq 0\), then, by writing \(\vv = \vSig^{1/2}\vf\), we have that
\[
    m = \frac{\bprodlEn{\vv}{\Da \vv}}{\nrmlEn{\vv}^2}.
\]
Note that some of the components of \(\vv\) may be identically zero if \(\vSig\) is singular. Fix \(\xio \in \Omo\) and suppose that \(\vf(\xio) = 0\). Let \(\vv = \vSig^{1/2}\vf\) so that \(\xio\) is also a zero of \(\vv\). Then \(v_1,\ldots,v_n\in \VF\) and, formally,
\[
    m = \frac{\bprodlEm{\vv}{\Da \vv}}{\nrmlEm{\vv}^2}.
\]
We first prove the case \(d=1\) as the proof technique is slightly different. Since \(\nrmlEn{\vv}^2\) is real-analytic, it cannot vanish on an open interval. In particular, its zeros are isolated. In this case, we can write
\begin{align*}
    \bprodlEm{\vv(\xi)}{\Da \vv(\xi)} &= (\xi -\xio)^p N_p(\xi)\\
    \nrmlEn{\vv(\xi)}^2 &= (\xi -\xio)^q D_q(\xi)
\end{align*}
for continuous functions \(N_p, D_q\) for which \(N_p(\xio), D_q(\xio) \neq 0\). Therefore, if \(I\subset \Omor\) is an open interval with \(\xio\in I\) and such that it contains no other zero of \(\nrmlEn{\vv}^2\) then
\[
    \infty > \int_{I} |m(\xi)|^2\, d\xi \geq \int_{I} \LabsE{\frac{N_p(\xi)}{D_q(\xi)}}^2 \frac{d\xi}{(\xi - \xio)^{2(q - p)}}.
\]
Since \(|N_p/D_q|\) is bounded below on \(I\) we conclude that \(q - p \leq 0\). In particular, \(p \geq q > q - 1/2\).

Suppose now that \(d > 1\).  Since the elements of \(\VF\) in the finite-dimensional case are real-analytic on their entire domain, we can find \(p,q\in \N_0\) such that
\begin{align*}
    \bprodlEm{\vv(\xi)}{\Da \vv(\xi)} &= N_p(\xi-\xio) + |\xi - \xio|^{p} R_N(\xi-\xio)\\
    \nrmlEm{\vv(\xi)}^2 &= D_q(\xi - \xio) + |\xi - \xio|^{q} R_D(\xi-\xio)
\end{align*}
where \(N_p\) is a non-zero homogeneous polynomial of degree \(p\), \(D_q\) is a non-zero homogeneous polynomial of degree \(q\), and \(R_N, R_D\) are continuous functions such that \(R_N(\xi) \to 0\) and \(R_D(\xi)\to 0\) as \(\xi\to \xio\). Since \(D_q\) is continuous and not identically zero, for \(\eps > 0\) sufficiently small we must have that the set
\[
    W_{\xio}^{\eps} := \set{\omega\in \mbb{S}^{d-1}:\,\, |D_q(\omega)| > \eps}
\]
has positive measure, i.e., \(\mc{H}^{d-1}(\mc{W}^{\eps}_q) > 0\). As a consequence, for \(\delta > 0\) such that \(\BE(\xio, \delta)\subset \tint{\Omor}\) the set
\[
    A^{\eps,\delta}_{\xio} := \set{\xio + r \omega:\,\, r\in [0,\delta],\,\omega\in W^{\eps}_{\xio}}
\]
has positive Lebesgue measure. Since \(m\in \LEOmor\) we must have that
\[
    \infty > \int_{A^{\eps,\delta}_{\xio}} |m(\xi)|^2\, d\xi = \int_0^\delta r^{d-1}\left(\int_{W_{\xio}^{\eps}} |m(\xio + r\omega)|^2 dS(\omega)\right) dr.
\]
However, for \(\omega\in W_{\xio}\) we have that
\[
    |m(\xio + r\omega)|^2 = r^{2(p-q)}\LabsE{\frac{N_p(\omega) + R_N(r\omega)}{D_p(\omega) + R_D(r\omega)}}^2
\]
whence
\[
    \infty > \int_0^\delta r^{d-1 + 2(p-q)}\left(\int_{W_{\xio}^{\eps}} \LabsE{\frac{N_p(\omega) + R_N(r\omega)}{D_p(\omega) + R_D(r\omega)}}^2 dS(\omega)\right) dr =: \int_0^{\delta} r^{d-1 + 2(p-q)} M_{\xio}^{\eps}(r)\, dr.
\]
The function \(M_{\xio}^{\eps}\) is bounded for all sufficiently small \(r\) and, in fact, from the dominated convergence theorem we obtain
\[
    \limsup_{r\to 0} M_{\xio}^{\eps}(r) = \int_{W_{\xio}^{\eps}} \LabsE{\frac{N_p(\omega)}{D_p(\omega)}}^2 dS(\omega) <\infty.
\]
Hence, for the integral to be finite, we must have that \(d - 1 + 2(p-q) > -1\) or
\[
    p > q - \frac{d}{2},
\]
proving the claim.

%% ============================================================================
\printbibliography[title=References]
%% ============================================================================

\end{document}